\documentclass[oneside,english,reqno]{amsart}
\usepackage[T1]{fontenc}
\usepackage[latin9]{inputenc}
\usepackage{geometry}
\usepackage{color}
\usepackage{babel}
\usepackage{float}
\usepackage{mathrsfs}
\usepackage{amsthm}
\usepackage{amsbsy}
\usepackage{amstext}
\usepackage{amssymb}
\usepackage{graphicx}
\usepackage{esint}
\usepackage[all]{xy}
\PassOptionsToPackage{normalem}{ulem}
\usepackage{ulem}
\usepackage[unicode=true,pdfusetitle,
 bookmarks=true,bookmarksnumbered=false,bookmarksopen=false,
 breaklinks=false,pdfborder={0 0 1},backref=false,colorlinks=false]
 {hyperref}
\hypersetup{
 colorlinks=true,citecolor=blue,linkcolor=blue,linktocpage=true}

\makeatletter

\newcommand{\lyxmathsym}[1]{\ifmmode\begingroup\def\b@ld{bold}
  \text{\ifx\math@version\b@ld\bfseries\fi#1}\endgroup\else#1\fi}

\providecommand{\tabularnewline}{\\}

\numberwithin{equation}{section}
\numberwithin{figure}{section}
\theoremstyle{plain}
\newtheorem{thm}{\protect\theoremname}[section]
  \theoremstyle{plain}
  \newtheorem{lem}[thm]{\protect\lemmaname}
  \theoremstyle{plain}
  \newtheorem{cor}[thm]{\protect\corollaryname}
  \theoremstyle{remark}
  \newtheorem{rem}[thm]{\protect\remarkname}
  \theoremstyle{plain}
  \newtheorem{prop}[thm]{\protect\propositionname}
  \theoremstyle{definition}
  \newtheorem{defn}[thm]{\protect\definitionname}
  \theoremstyle{definition}
  \newtheorem{example}[thm]{\protect\examplename}

\providecommand{\MR}[1]{}

\makeatother

  \providecommand{\corollaryname}{Corollary}
  \providecommand{\definitionname}{Definition}
  \providecommand{\examplename}{Example}
  \providecommand{\lemmaname}{Lemma}
  \providecommand{\propositionname}{Proposition}
  \providecommand{\remarkname}{Remark}
\providecommand{\theoremname}{Theorem}

\begin{document}

\title{SPECTRAL THEORY OF MULTIPLE INTERVALS}

\author{Palle Jorgensen, Steen Pedersen, and Feng Tian}

\address{(Palle E.T. Jorgensen) Department of Mathematics, The University
of Iowa, Iowa City, IA 52242-1419, U.S.A. }

\email{jorgen@math.uiowa.edu }

\urladdr{http://www.math.uiowa.edu/\textasciitilde{}jorgen/}

\address{(Steen Pedersen) Department of Mathematics, Wright State University,
Dayton, OH 45435, U.S.A. }

\email{steen@math.wright.edu }

\urladdr{http://www.wright.edu/\textasciitilde{}steen.pedersen/}

\address{(Feng Tian) Department of Mathematics, Wright State University, Dayton,
OH 45435, U.S.A. }

\email{feng.tian@wright.edu }

\urladdr{http://www.wright.edu/\textasciitilde{}feng.tian/}


\maketitle
\begin{center}
To the memory of William B. Arveson.
\par\end{center}
\begin{abstract}
We present a model for spectral theory of families of selfadjoint
operators, and their corresponding unitary one-parameter groups (acting
in Hilbert space.) The models allow for a scale of complexity, indexed
by the natural numbers $\mathbb{N}$. For each $n\in\mathbb{N}$,
we get families of selfadjoint operators indexed by: (i) the unitary
matrix group $U(n)$, and by (ii) a prescribed set of $n$ non-overlapping
intervals. Take $\Omega$ to be the complement in $\mathbb{R}$ of
$n$ fixed closed finite and disjoint intervals, and let $L^{2}(\Omega)$
be the corresponding Hilbert space. Moreover, given $B\in U(n)$,
then both the lengths of the respective intervals, and the gaps between
them, show up as spectral parameters in our corresponding spectral
resolutions within $L^{2}(\Omega)$. Our models have two advantages:
One, they encompass realistic features from quantum theory, from acoustic
wave equations and their obstacle scattering; as well as from harmonic
analysis. 

Secondly, each choice of the parameters in our models, $n\in\mathbb{N}$,
$B\in U(n)$, and interval configuration, allows for explicit computations,
and even for closed-form formulas: Computation of spectral resolutions,
of generalized eigenfunctions in $L^{2}(\Omega)$ for the continuous
part of spectrum, and for scattering coefficients. Our models further
allow us to identify embedded point-spectrum (in the continuum), corresponding,
for example, to bound-states in scattering, to trapped states, and
to barriers in quantum scattering. The possibilities for the discrete
atomic part of spectrum includes both periodic and non-periodic distributions.
\end{abstract}
\tableofcontents{}

\section{Introduction\label{sec:Intro}}

The study of unitary one-parameter groups (\cite{vN49}) is used in
such areas as quantum mechanics (\cite{Ba49,Ch11,AH11} to mention
a few), in PDE, and more generally in dynamical systems, and in harmonic
analysis, see e.g., \cite{DJ09}. A unitary one-parameter group $U(t)$
is a representation of the additive group of the real line $\mathbb{R}$,
$t\in\mathbb{R}$, with each unitary operator $U(t)$ acting on a
complex Hilbert space $\mathscr{H}$. By a theorem of Stone (see \cite{Sto90,LP68,DS88b}
for details), we know that there is a bijective correspondence between:
(i) strongly continuous unitary one-parameter groups $U(t)$ acting
on $\mathscr{H}$; and (ii) selfadjoint operators $P$ with dense
domain in $\mathscr{H}$.

In quantum mechanics, the unit-norm vectors in the Hilbert space $\mathscr{H}$
correspond to quantum states, and the unitary one-parameter groups
$U(t)$ will represent the solutions to a Schrödinger equation, with
$(\mathscr{H},U(t))$ depending on the preparation of the quantum
system at hand. In linear PDE theory, unitary one-parameter groups
are used to represent time-dependent solutions when a conserved quantity
can be found, for example for the acoustic wave equation, see \cite{LP68}.
In dynamical systems, selfadjoint operators and unitary one-parameter
groups are the ingredients of Sturm-Liouville equations and boundary
value problems.

In these applications, the first question for $(\mathscr{H},U(t))$
relates to spectrum. We take the spectrum for $U(t)$ to be the spectrum
of its selfadjoint generator. Hence one is led to study $(\mathscr{H},U(t))$
up to unitary equivalence. The gist of Lax-Phillips theory \cite{LP68}
is that $(\mathscr{H},U(t))$, up to multiplicity, will be unitarily
equivalent to the translation representation, i.e., to the group of
translation operators acting in $L^{2}(\mathbb{R},\mathscr{M})$,
the square-integrable functions from $\mathbb{R}$ into a complex
Hilbert space $\mathscr{M}$. The dimension of $\mathscr{M}$ is called
multiplicity. For interesting questions one may take $\mathscr{M}$
to be of finite small dimension; see details below, and \cite{JPT11-1,JPT11-2}.

In this paper, we study a setting of scattering via translation representations
in the sense of Lax-Phillips. 

To make concrete the geometric possibilities, we study here $L^{2}(\Omega)$
when $\Omega$ is a fixed open subset of $\mathbb{R}$ with two unbounded
connected components. For many questions, we may restrict to the case
when there is only a finite number of bounded connected components
in $\Omega$.

In other words, $\Omega$ is the complement of a finite number of
closed, bounded and disjoint intervals. We begin with Dirichlet boundary
conditions for the derivative operator $d/dx$, i.e., defined on absolutely
continuous $L^{2}$ functions with $f'\in L^{2}(\Omega)$ and vanishing
on the boundary of $\Omega$, $f=0$ on $\partial\Omega$. Using deficiency
index theory (\cite{vN49,DS88b}), we then arrive at all the skew-selfadjoint
extensions, and the corresponding unitary one-parameter groups $U(t)$
acting on $L^{2}(\Omega)$.

We expect that our present model will have relevance to other boundary
value problems, for example in the study of second order operators,
and regions in higher dimensions; see e.g., \cite{Bra04}.

\subsection{Overview}

In this setting, we resolve the possibilities for spectrum, and we
show how they depend on the respective interval lengths, and their
configuration, i.e., the length of the interval-gaps, as well as of
the assigned boundary conditions. Our conclusions are computational,
in closed-form representations; and expressed in terms of explicit
direct integral formulas for each of the unitary one-parameter groups
$U(t)$.

For each of these unitary one-parameter groups $U_{B}(t)$, we compute
its spectral decomposition as an explicit direct integral of generalized
eigenfunctions $\psi_{\lambda}$, functions of two variables, the
spectral variable $\lambda\in\mathbb{R}$, and of the spatial variable
$x\in\Omega$. The direct dependence of generalized eigenfunctions
on the boundary condition $B$ is computed. The functions $\psi_{\lambda}$
fall within the family known as exponential polynomials (see e.g.,
\cite{AH10,MN10}), or Fourier exponential polynomials. We further
identify in detail those special selfadjoint extension operators for
which there is embedded discrete spectrum.

We now move on to the technical details in our construction, beginning
with operator theory. A more detailed overview is postponed until
the start of section \ref{sec:sp}. In fact, we will have a fuller
discussion of applications in the three sections \ref{sec:sp} through
\ref{sec:deg} below. In each, we begin with an outline of both the
new main ideas introduced, as well as their spectral theoretic relevance
to quantum mechanics, to wave equation scattering, and to harmonic
analysis.

In section \ref{sec:infinite}, for comparison, we consider some cases
when the give open set $\Omega$ has an infinite number of connected
components, still including the two infinite half-lines. This is of
interest for a variety of reasons: One is recent studies of geometric
analysis of Cantor sets \cite{DJ07,DJ11,JP98,PW01}; so the infinite
component case for $\Omega$ includes examples when $\Omega$ is the
complement in $\mathbb{R}$ of a Cantor set of a fixed fractal scaling
dimension. This offers a framework for boundary value problems when
the boundary is different from the more traditional choices. And finally,
the case when the von Neumann-deficiency indices are $(\infty,\infty)$
offers new challenges; see e.g., \cite{DS88b}; involving now reproducing
kernels, and more refined spectral theory. Finally, these examples
offer a contrast to the finite case; for example, for finitely many
intervals (Theorem \ref{thm:Bdecomp}) we prove that the Beurling
density of embedded point spectrum equals the total length of the
finite intervals. By contrast, when $\Omega$ has an infinite number
of connected components, we show in sect \ref{sec:infinite} that
there is the possibility of dense point spectrum.

\subsection{Unbounded Operators}

We recall the following fundamental result of von Neumann on extensions
of Hermitian operators.

In order to make precise our boundary conditions, we need a:
\begin{lem}
\label{lem:cont}Let $\Omega\subset\mathbb{R}$ be as above. Suppose
$f$ and $f'=\frac{d}{dx}f$ (distribution derivative) are both in
$L^{2}(\Omega)$; then there is a continuous function $\tilde{f}$
on $\overline{\Omega}$ (closure) such that $f=\tilde{f}$ a.e. on
$\Omega$, and $\lim_{\left|x\right|\rightarrow\infty}\tilde{f}(x)=0$. \end{lem}
\begin{proof}
Let $p\in\mathbb{R}$ be a boundary point. Then for all $x\in\Omega$,
we have:
\begin{equation}
f(x)-f(p)=\int_{p}^{x}f'(y)dy.\label{eq:tmp-5-1}
\end{equation}
Indeed, $f'\in L_{loc}^{1}$ on account of the following Schwarz estimate
\[
\left|f(x)-f(p)\right|\leq\sqrt{\left|x-p\right|}\left\Vert f'\right\Vert _{L^{2}(\Omega)}.
\]
Since the RHS in (\ref{eq:tmp-5-1}) is well-defined, this serves
to make the LHS also meaningful. Now set
\[
\tilde{f}(x):=f(p)+\int_{p}^{x}f'(y)dy,
\]
and it can readily be checked that $\tilde{f}$ satisfies the conclusions
in the Lemma.\end{proof}
\begin{lem}[see e.g. \cite{DS88b}]
\label{lem:vN def-space}Let $L$ be a \textcolor{black}{closed}
Hermitian operator with dense domain $\mathscr{D}_{0}$ in a Hilbert
space. Set 
\begin{alignat}{1}
\mathscr{D}_{\pm} & =\{\psi_{\pm}\in dom(L^{*})\left.\right|L^{*}\psi_{\pm}=\pm i\psi_{\pm}\}\nonumber \\
\mathscr{C}(L) & =\{U:\mathscr{D}_{+}\rightarrow\mathscr{D}_{-}\left.\right|U^{*}U=P_{\mathscr{D}_{+}},UU^{*}=P_{\mathscr{D}_{-}}\}\label{eq:vN1}
\end{alignat}
where $P_{\mathscr{D}_{\pm}}$ denote the respective projections.
Set 
\[
\mathscr{E}(L)=\{S\left.\right|L\subseteq S,S^{*}=S\}.
\]
Then there is a bijective correspondence between $\mathscr{C}(L)$
and $\mathscr{E}(L)$, given as follows: 

If $U\in\mathscr{C}(L)$, and let $L_{U}$ be the restriction of $L^{*}$
to 
\begin{equation}
\{\varphi_{0}+f_{+}+Uf_{+}\left.\right|\varphi_{0}\in\mathscr{D}_{0},f_{+}\in\mathscr{D}_{+}\}.\label{eq:vN2}
\end{equation}
Then $L_{U}\in\mathscr{E}(L)$, and conversely every $S\in\mathscr{E}(L)$
has the form $L_{U}$ for some $U\in\mathscr{C}(L)$. With $S\in\mathscr{E}(L)$,
take 
\begin{equation}
U:=(S-iI)(S+iI)^{-1}\left.\right|_{\mathscr{D}_{+}}\label{eq:vN3}
\end{equation}
and note that 
\begin{enumerate}
\item $U\in\mathscr{C}(L)$, and
\item $S=L_{U}$.
\end{enumerate}

Vectors $f\in dom(L^{*})$ admit a unique decomposition $f=\varphi_{0}+f_{+}+f_{-}$
where $\varphi_{0}\in dom(L)$, and $f_{\pm}\in\mathscr{D}_{\pm}$.
For the boundary-form $\mathbf{B}(\cdot,\cdot)$, we have
\begin{alignat}{1}
i\mathbf{B}(f,f) & =\left\langle L^{*}f,f\right\rangle -\left\langle f,L^{*}f\right\rangle \nonumber \\
 & =\left\Vert f_{+}\right\Vert ^{2}-\left\Vert f_{-}\right\Vert ^{2}.\label{eq:BoundaryForm}
\end{alignat}

\end{lem}

\subsection{Prior Literature}

There are related investigations in the literature on spectrum and
deficiency indices. For the case of indices $(1,1)$, see for example
\cite{ST10,Ma11}. For a study of odd-order operators, see \cite{BH08}.
Operators of even order in a single interval are studied in \cite{Oro05}.
The paper \cite{BV05} studies matching interface conditions in connection
with deficiency indices $(m,m)$. Dirac operators are studied in \cite{Sak97}.
For the theory of selfadjoint extensions operators, and their spectra,
see \cite{Smu74,Gil72}, for the theory; and \cite{Naz08,VGT08,Vas07,Sad06,Mik04,Min04}
for recent papers with applications. For applications to other problems
in physics, see e.g., \cite{AH11,PoRa76,Ba49,MK08}. And \cite{Ch11}
on the double-slit experiment. For related problems regarding spectral
resolutions, but for fractal measures, see e.g., \cite{DJ07,DJ09,DJ11}.

The study of deficiency indices $(n,n)$ has a number of additional
ramifications in analysis: Included in this framework is Krein's analysis
of Voltera operators and strings; and the determination of the spectrum
of inhomogenous strings; see e.g., \cite{DS01,KN89,Kr70,Kr55}.

Also included is their use in the study of de Branges spaces, see
e.g., \cite{Ma11}, where it is shown that any regular simple symmetric
operator with deficiency indices $(1,1)$ is unitarily equivalent
to the operator of multiplication in a reproducing kernel Hilbert
space of functions on the real line with a sampling property Kramer).
Further applications include signal processing, and de Branges-Rovnyak
spaces: Characteristic functions of Hermitian symmetric operators
apply to the cases unitarily equivalent to multiplication by the independent
variable in a de Branges space of entire functions.

\section{Momentum Operators\label{sec:P}}

In this section we outline our model, and we list the parameters of
the family of boundary value problems to be studied. We will need
a technical lemma on reproducing kernels.

By momentum operator $P$ we mean the generator for the group of translations
in $L^{2}(-\infty,\infty)$, see (\ref{eq:MomentumOperator}) below.
There are several reasons for taking a closer look at restrictions
of the operator $P.$ In our analysis, we study spectral theory determined
by the complement of $n$ bounded disjoint intervals, i.e., the union
of $n$ bounded component and two unbounded components (details below.)
Our motivation derives from quantum theory, and from the study of
spectral pairs in geometric analysis; see e.g., \cite{DJ07}, \cite{Fu74},
\cite{JP99}, \cite{Laba01}, and \cite{PW01}. In our model, we examine
how the spectral theory depends on both variations in the choice of
the $n$ intervals, as well as on variations in the von Neumann parameters. 

Granted that in many applications, one is faced with vastly more complicated
data and operators; nonetheless, it is often the case that the more
subtle situations will be unitarily equivalent to a suitable model
involving $P$. This is reflected for example in the conclusion of
the Stone-von Neumann uniqueness theorem: The Weyl relations for quantum
systems with a finite number of degree of freedom are unitarily equivalent
to the standard model with momentum and position operators $P$ and
$Q$. For details, see e.g., \cite{Jo81}.

\subsection{The boundary form, spectrum, and the group $U(n)$\label{sub:bform}}

Fix $n>2$, let $-\infty<\beta_{1}<\alpha_{1}<\beta_{2}<\alpha_{2}<\cdots<\beta_{n}<\alpha_{n}<\infty$,
and let 
\begin{equation}
\Omega:=\mathbb{R}\backslash\left(\bigcup_{k=1}^{n}[\beta_{k},\alpha_{k}]\right)=\bigcup_{k=0}^{n}J_{k}\label{eq:Omega}
\end{equation}
be the exterior domain, where
\begin{equation}
J_{0}:=\left(-\infty,\beta_{1}\right),J_{1}:=\left(\alpha_{1},\beta_{2}\right),\ldots,J_{n-1}:=(\alpha_{n-1},\beta_{n}),J_{n}:=\left(\alpha_{n},\infty\right).\label{eq:J}
\end{equation}
Moreover, we set 
\begin{equation}
J_{-}:=J_{0},\; J_{+}:=J_{n}\label{eq:Jex}
\end{equation}
for the two unbounded components; see Figure \ref{fig:Omega} below.

\begin{figure}[H]
\includegraphics[scale=0.8]{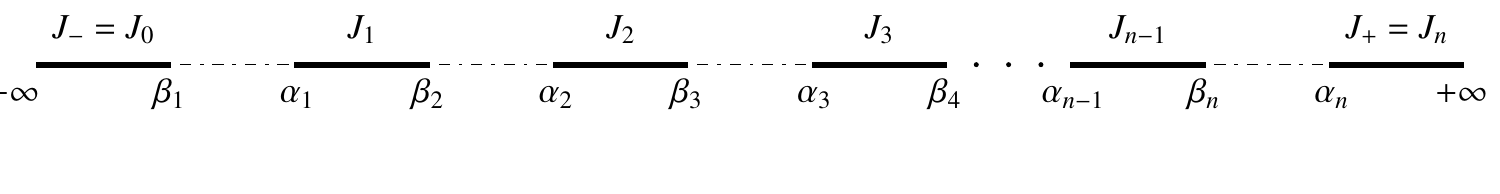}

\caption{\label{fig:Omega}$\Omega=\bigcup_{k=0}^{n}J_{k}=\left(\bigcup_{k=1}^{n-1}J_{k}\right)\cup\left(J_{-}\cup J_{+}\right)$,
i.e., $\Omega=$ the complement in $\mathbb{R}$ of $n$ finite and
disjoint intervals.}
\end{figure}
We shall write $\boldsymbol{\alpha}=(\alpha_{i})$ for all the left-hand
side endpoints, and $\boldsymbol{\beta}=(\beta_{i})$ for the right-hand
side endpoints in $\partial\Omega$. 

Let $L^{2}(\Omega)$ be the Hilbert space with respect to the inner
product
\begin{equation}
\langle f,g\rangle:=\sum_{k=0}^{n}\int_{J_{k}}\overline{f(x)}g(x)dx.\label{eq:InnerProduct}
\end{equation}

The \emph{maximal momentum operator} is 
\begin{equation}
P:=\frac{1}{i2\pi}\frac{d}{dx}\label{eq:MomentumOperator}
\end{equation}
with domain $\mathscr{D}(P)$ equal to the set of absolutely continuous
functions on $\Omega$ where both $f$ and $Pf$ are square-integrable. 

The \emph{boundary form} associated with $P$ is defined as the form
\begin{equation}
i2\pi\,\mathbf{B}\left(g,f\right):=\langle g,Pf\rangle-\langle Pg,f\rangle\label{eq:BoundaryForm1}
\end{equation}
on $\mathscr{D}\left(P\right)$. This is consistent with (\ref{eq:BoundaryForm}):
If $L=P_{min}$, then $L^{*}$ in (\ref{eq:BoundaryForm}) is $P$.
Recall, $\mathscr{D}(P_{min})=\{f\in\mathscr{D}(P)\:;\: f=0\:\mbox{ on }\partial\Omega\}$. 
\begin{lem}
Let $\boldsymbol{\alpha}=(\alpha_{i})$, $\boldsymbol{\beta}=(\beta_{i})$
be the system of interval endpoints in (\ref{eq:J}), and set 
\[
\rho_{1}(f):=f(\boldsymbol{\beta})=\left(\begin{array}{c}
f(\beta_{1})\\
f(\beta_{2})\\
\vdots\\
f(\beta_{n})
\end{array}\right),\;\rho_{2}(f):=f(\boldsymbol{\alpha})=\left(\begin{array}{c}
f(\alpha_{1})\\
f(\alpha_{2})\\
\vdots\\
f(\alpha_{n})
\end{array}\right)
\]
for all $f\in\mathscr{D}(P)$; then
\begin{equation}
i2\pi\,\mathbf{B}(g,f)=\left\langle g(\boldsymbol{\alpha}),f(\boldsymbol{\alpha})\right\rangle _{\mathbb{C}^{n}}-\left\langle g(\boldsymbol{\beta}),f(\boldsymbol{\beta})\right\rangle _{\mathbb{C}^{n}}\label{eq:BoundaryForm3}
\end{equation}
where $\left\langle \cdot,\cdot\right\rangle _{\mathbb{C}^{n}}$ is
the usual Hilbert-inner product in $\mathbb{C}^{n}$.\end{lem}
\begin{proof}
First note that for the domain of the operator $L^{*}$ in $L^{2}(\Omega)$,
we have
\[
dom(L^{*})=\{f\in L^{2}(\Omega)\:;\: f'\in L^{2}(\Omega)\}.
\]
This means that every $f\in dom(L^{*})$ has a realization in $C(\overline{\Omega})$,
so continuous up to the boundary. As a result the following boundary
analysis is justified by von Neumann's formula (\ref{eq:BoundaryForm})
in Lemma \ref{lem:vN def-space}; and valid for for all $f,g\in dom(L^{*})$:
\begin{eqnarray*}
-i2\pi\boldsymbol{B}(g,f) & = & \left\langle L^{*}g,f\right\rangle _{\Omega}-\left\langle g,L^{*}f\right\rangle _{\Omega}\\
 & = & \int_{\Omega}\frac{d}{dx}\left(\overline{g(x)}f(x)\right)dx\\
 & = & \left(\int_{-\infty}^{\beta_{1}}+\sum_{j=1}^{n-1}\int_{\alpha_{j}}^{\beta_{j+1}}+\int_{\alpha_{n}}^{\infty}\right)\frac{d}{dx}\left(\overline{g(x)}f(x)\right)dx\\
 & = & \overline{g(\beta_{1})}f(\beta_{1})+\sum_{j=1}^{n-1}\left(\overline{g(\beta_{j+1})}f(\beta_{j+1})-\overline{g(\alpha_{j})}f(\alpha_{j})\right)-\overline{g(\alpha_{n})f(\alpha_{n})}\\
 & = & \left\langle g(\boldsymbol{\beta}),f(\boldsymbol{\beta})\right\rangle _{\mathbb{C}^{n}}-\left\langle g(\boldsymbol{\alpha}),f(\boldsymbol{\alpha})\right\rangle _{\mathbb{C}^{n}}.
\end{eqnarray*}
\end{proof}
\begin{cor}
It follows that the system $\left(\mathbb{C}^{n},\rho_{1},\rho_{2}\right)$,
$\rho_{1}(f)=f(\boldsymbol{\beta})$ and $\rho_{2}(f)=f(\boldsymbol{\alpha})$,
represents a boundary triple, and we get all the selfadjoint extension
operators for $P_{min}$ indexed by $B\in U(n)$; we shall write $P_{B}$.
Explicitly, see e.g., \cite{dO09}, 
\begin{equation}
\mathscr{D}\left(P_{B}\right):=\left\{ f\in\mathscr{D}\left(P\right)\mid B\rho_{1}(f)=\rho_{2}(f)\right\} .\label{eq:ExtensionDomain1}
\end{equation}
\end{cor}
\begin{rem}[The acoustic wave equation]
Below we sketch the use of our interval-model for Lax-Phillips obstacle
scattering (\cite{LP68}) for the acoustic wave equation, with water
waves; i.e., waves in a 2D medium. By \cite{LP68}, one knows that
the solution to the wave equation, subject to obstacle scattering,
may be presented by a unitary one parameter group $U(t)$ acting on
an energy Hilbert space $\mathscr{H}_{E}$ consisting of states representing
initial waves as initial position and wave velocity. But, via a Radon
transform (see \cite{He84,LP68}), $U(t)$, acting on the energy Hilbert
space $\mathscr{H}_{E}$, is in turn unitarily equivalent to a translation
representation acting on $L^{2}(\mathbb{R},\mathscr{M})$. The Hilbert
space $\mathscr{M}$ encodes the direction of the waves under consideration.
In Figure \ref{fig:AC} we illustrate a fixed compact planar obstacle,
and four different states $f_{1}$, $f_{2}$, $f_{3}$, and $f_{4}$,
each one with a different scattering profile. The first state $f_{1}$
transforms under $U(t)$ in a manner unitarily equivalent to an interval
model $\Omega$ with two bounded component-intervals, see eqs (\ref{eq:Omega})
and (\ref{eq:J}). For the second state $f_{2}$ the interval model
has only one bounded component. The third state $f_{3}$ has no bounded
component, but as with all four cases, the $\Omega$ model will have
two unbounded infinite half-lines. The interval model for $f_{4}$
corresponds to $\Omega=$ the complement in $\mathbb{R}$ of a single
point.

\begin{figure}
\begin{tabular}{cc}
\includegraphics{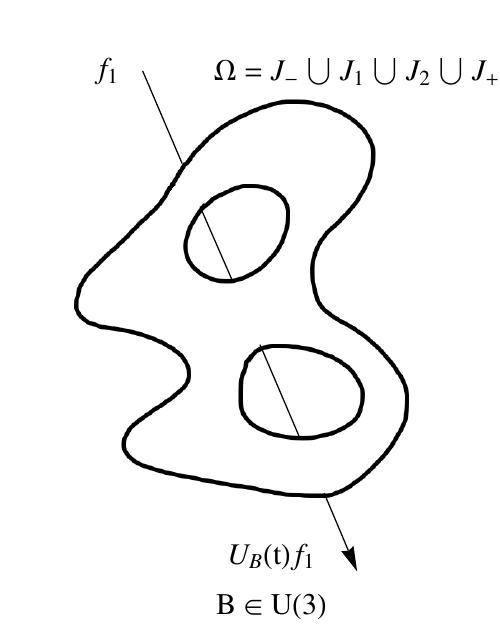} & \includegraphics{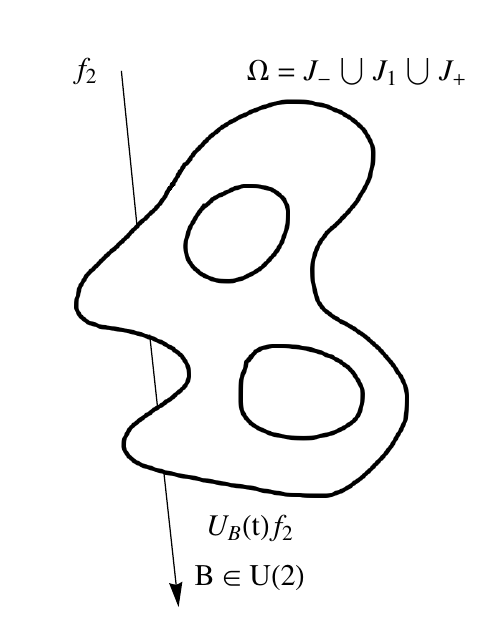}\tabularnewline
\includegraphics{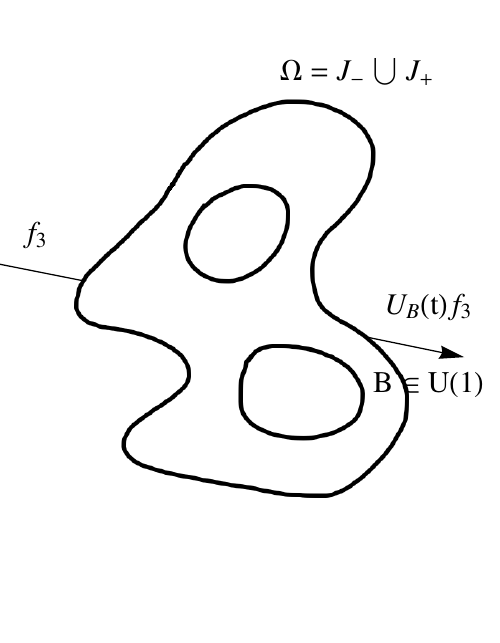} & \includegraphics{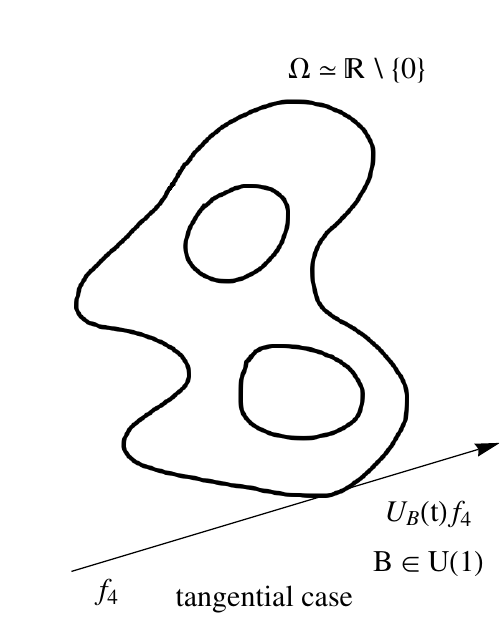}\tabularnewline
\end{tabular}

\caption{\label{fig:AC}Obstacle scattering data as cross-sectional scans of
a bounded planar object.}

\end{figure}

\end{rem}

\subsection{Some results in the paper}

We identify a number of sub-classes within the family of all selfadjoint
extensions $P_{B}$ of the minimal operator in $L^{2}(\Omega)$.

If the open set $\Omega$ is chosen (as the complement of a fixed
system consisting of $n$ bounded, closed and disjoint intervals),
then the set of all selfadjoint extensions is indexed by elements
$B$ in the matrix group $U(n)$. The possibilities for the spectral
resolution of a particular $P_{B}$ are twofold: (i) pure Lebesgue
spectrum with uniform multiplicity one; or (ii) still Lebesgue spectrum
but with embedded point spectrum (within the continuum).

While all the operators within class (i) are unitarily equivalent,
it is still the case that, within each of the two sides in the rough
subdivision, there is a rich variety of possibilities: Via a set of
scattering poles, we show that the fine-structure of the spectral
theory for each of the selfadjoint operators of $P_{B}$, and the
corresponding unitary one-parameter groups $U_{B}(t)$, depends on
all the geometric data: The number $n$, the choice of intervals,
their respective lengths, and the location of the gaps; see Figure
\ref{fig:Omega}. More precisely, these spectral/scattering differences
reflect themselves in detailed properties of an associated system
of scattering coefficients, see eq (\ref{eq:eigen}) in subsection
\ref{sub:sp-eig} below.) To identifying particulars for a given unitary
one-parameter group $U_{B}(t)$ we study the location of a set of
scattering poles.

The resolution of these questions is closely related with a more coarse
distinction: This has to do with decomposition properties for the
unitary one-parameter groups $U_{B}(t)$ in $L^{2}(\Omega)$; a question
taken up in the last three sections of the paper.

In sections \ref{sec:P} and \ref{sec:sp} below we prove the following
theorem. 
\begin{thm}
\label{thm:UB}If $B\in U(n)$ is non-degenerate (see Definition \ref{def:degenerate}),
then there is a system of bounded generalized eigenfunctions $\{\psi_{\lambda}^{(B)};\lambda\in\mathbb{R}\}$,
and a positive Borel function $F_{B}(\cdot)$ on $\mathbb{R}$ such
that the unitary one-parameter group $U_{B}(t)$ in $L^{2}(\Omega)$
generated by $P_{B}$ has the form
\begin{equation}
\left(U_{B}(t)f\right)(x)=\int_{\mathbb{R}}e_{\lambda}(-t)\left\langle \psi_{\lambda}^{(B)},f\right\rangle _{\Omega}\psi_{\lambda}^{(B)}(x)F_{B}(\lambda)d\lambda\label{eq:UB-3}
\end{equation}
for all $f\in L^{2}(\Omega)$, $x\in\Omega$, and $t\in\mathbb{R}$;
where
\[
\left\langle \psi_{\lambda}^{(B)},f\right\rangle _{\Omega}:=\int_{\Omega}\overline{\psi_{\lambda}^{(B)}(y)}f(y)dy.
\]

\end{thm}
In section \ref{sec:P}, we prepare with some technical lemmas; and
in section \ref{sec:sp} we compute explicit formulas for the expansion
(\ref{eq:UB-3}) above, and we discuss their physical significance. 

Our study of duality pairs $x$ and $\lambda$ in systems of generalized
eigenfunctions $\psi_{\lambda}$ is related to, but different from
another part of spectral theory, dual variables for bispectral problems;
see e.g., \cite{Gru11,GrRa10,DuGr09}.
\begin{thm}
\label{thm:sp}Let $d\sigma_{B}(\cdot)$ be the measure in (\ref{eq:UB-3})
and let $V_{B}:L^{2}(\Omega)\rightarrow L^{2}(\mathbb{R},\sigma_{B})$
be the spectral transform in (\ref{eq:VB}) with adjoint operator
$V_{B}^{*}:L^{2}(\mathbb{R},\sigma_{B})\rightarrow L^{2}(\Omega)$.
Then 
\begin{align*}
V_{B}V_{B}^{*} & =I_{L^{2}(\sigma_{B})}\;\mbox{ and }\\
V_{B}^{*}V_{B} & =I_{L^{2}(\Omega)}.
\end{align*}
Moreover,
\begin{equation}
V_{B}U_{B}(t)V_{B}^{*}=M_{t}\label{eq:Mt}
\end{equation}
where $M_{t}$ is the unitary one-parameter group acting on $L^{2}(\mathbb{R},\sigma_{B})$
as follows 
\[
\left(M_{t}g\right)(\lambda)=e_{\lambda}(-t)g(\lambda)
\]
for all $t,\lambda\in\mathbb{R}$, and all $g\in L^{2}(\mathbb{R},\sigma_{B})$.
\end{thm}

\subsection{Reproducing Kernel Hilbert Space}

In this section we introduce a certain reproducing kernel Hilbert
space $\mathscr{H}_{1}(\Omega)$; a first order Sobolev space, hence
the subscript 1. Its reproducing kernel is found (Lemma \ref{lem:k-1}),
and it serves two purposes: First, we show that each of the unbounded
selfadjoint extension operators $P_{B}$, defined from (\ref{eq:ExtensionDomain1})
in sect \ref{sub:bform}, have their graphs naturally embedded in
$\mathscr{H}_{1}(\Omega)$. Secondly, for each $P_{B}$, the reproducing
kernel for $\mathscr{H}_{1}(\Omega)$ helps us pin down the generalized
eigenfunctions for $P_{B}$. The arguments for this are based in turn
on Lemma \ref{lem:vN def-space} and the boundary form $\boldsymbol{B}$
from (\ref{eq:BoundaryForm3}).
\begin{lem}
\label{lem:k-1}Let 
\begin{equation}
\Omega=\bigcup_{k=0}^{n}J_{k}\label{eq:Omega-2}
\end{equation}
be as above, and $L^{2}(\Omega)$ be the Hilbert space of all $L^{2}$-functions
on $\Omega$ with inner product $\left\langle \cdot,\cdot\right\rangle _{\Omega}$
and norm $\left\Vert \cdot\right\Vert _{\Omega}$. Set 
\[
\mathscr{H}_{1}(\Omega)=\{f\in L^{2}(\Omega)\left|\right.Df=f'\in L^{2}(\Omega)\};
\]
then $\mathscr{H}_{1}(\Omega)$ is a reproducing kernel Hilbert space
of functions on $\overline{\Omega}$ (closure). \end{lem}
\begin{proof}
For the special case where $\Omega=\mathbb{R}$, the details are in
\cite{Jo81}. For the case where $\Omega$ is the exterior domain
from (\ref{eq:Omega-2}), we already noted (Lemma \ref{lem:cont})
that each $f\in\mathscr{H}_{1}(\Omega)$ has a continuous representation
$\tilde{f}$, and that $\tilde{f}$ vanishes at $\pm\infty$. The
inner product in $\mathscr{H}_{1}(\Omega)$ is
\begin{equation}
\left\langle f,g\right\rangle _{\mathscr{H}_{1}(\Omega)}=\left\langle f,g\right\rangle _{\Omega}+\left\langle f',g'\right\rangle _{\Omega}.\label{eq:H-inner}
\end{equation}
Let $x\in\overline{\Omega}=\cup_{k=0}^{n}\overline{J}_{k}$, and denote
by $J$ the interval containing $x;$ and let $p$ be a boundary point
in $J$. Then an application of Cauchy-Schwarz yields
\begin{alignat*}{1}
\left|\tilde{f}(x)\right|^{2}-\left|\tilde{f}(p)\right|^{2} & =2\Re\int_{p}^{x}\overline{f(y)}f'(y)dy\\
 & \leq\left\Vert f\right\Vert _{J}^{2}+\left\Vert f'\right\Vert _{J}^{2}\leq\left\Vert f\right\Vert _{\mathscr{H}_{1}(\Omega)}^{2}.
\end{alignat*}

We conclude that the linear functional 
\[
\mathscr{H}_{1}(\Omega)\ni f\rightsquigarrow\tilde{f}(x)\in\mathbb{C}
\]
is continuous on $\mathscr{H}_{1}(\Omega)$ with respect to the norm
from (\ref{eq:H-inner}). By Riesz, applied to $\mathscr{H}_{1}(\Omega)$,
we conclude that there is a unique $k_{x}\in\mathscr{H}_{1}(\Omega)$
such that
\begin{equation}
\tilde{f}(x)=\left\langle k_{x},f\right\rangle _{\mathscr{H}_{1}(\Omega)}\label{eq:rep}
\end{equation}
for all $f\in\mathscr{H}_{1}(\Omega)$. 

If $x$ in (\ref{eq:rep}) is a boundary point, then the formula must
be modified using instead $\tilde{f}(x_{+})=$ limit from the right
if $x$ is a left-hand side end-point in $J$. If $x$ is instead
a right-hand side end-point in $J$, then use $\tilde{f}(x_{-})$
in formula (\ref{eq:rep}). This concludes the proof of the Lemma.\end{proof}
\begin{rem}
\label{rem:n}If $\Omega$ is the union of a finite number of bounded
components, and two unbounded, i.e., (See (\ref{eq:Omega})-(\ref{eq:Jex})
and Figure \ref{fig:Omega}. ) 
\begin{equation}
\Omega=(-\infty,\beta_{1})\cup\bigcup_{i=1}^{n-1}(\alpha_{i},\beta_{i+1})\cup(\alpha_{n},\infty)\label{eq:Omega-4}
\end{equation}
where
\[
-\infty<\beta_{1}<\alpha_{1}<\beta_{2}<\alpha_{2}<\cdots<\alpha_{n-1}<\beta_{n}<\alpha_{n}<\infty.
\]
Set 
\[
k_{R}=\left(\begin{array}{c}
k_{\beta_{1}}\\
k_{\beta_{2}}\\
\vdots\\
k_{\beta_{n}}
\end{array}\right)\;\mbox{and }\: k_{L}=\left(\begin{array}{c}
k_{\alpha_{1}}\\
k_{\alpha_{2}}\\
\vdots\\
k_{\alpha_{n}}
\end{array}\right)
\]
in $\bigoplus_{i=1}^{n}\mathscr{H}_{1}(\Omega)$. Let $B$ be a unitary
complex $n\times n$ matrix, i.e., $B\in U(n)$; then there is a unique
selfadjoint operator $P_{B}$ with dense domain $\mathscr{D}(P_{B})$
in $L^{2}(\Omega)$ such that
\begin{equation}
\mathscr{D}(P_{B})=\left\{ f\in\mathscr{H}_{1}(\Omega);\underset{n\mbox{ times}}{\underbrace{f\oplus\cdots\oplus f}}\perp(k_{R}-Bk_{L})\mbox{ in }\bigoplus_{i=1}^{n}\mathscr{H}_{1}(\Omega)\right\} ;\label{eq:tmp-11}
\end{equation}
and all the selfadjoint extensions of the minimal operator $D_{min}$
in $L^{2}(\Omega)$ arise this way. In particular, the deficiency
indices are $(n,n)$.\end{rem}
\begin{prop}[Boundstates]
\label{prop:n}Let $n\geq2$; and set $J_{i}=(\alpha_{i},\beta_{i+1})$,
$J_{-}=(-\infty,\beta_{1})$, $J_{+}=(\alpha_{n},\infty)$ as in (\ref{eq:Omega-4}).
Set $\tilde{\Omega}=\cup_{i=1}^{n-1}J_{i}$, so
\begin{equation}
L^{2}(\Omega)\cong L^{2}(\tilde{\Omega})\oplus L^{2}(J_{-}\cup J_{+}).\label{eq:tmp-38}
\end{equation}
Of the selfadjoint extension operators $P_{B}$, indexed by $B\in U(n)$,
we get the $\oplus$ direct decomposition 
\begin{equation}
P_{B}\cong P_{\tilde{\Omega}}\oplus P_{ext}\label{eq:tmp-39}
\end{equation}
where $P_{\tilde{\Omega}}$ is densely defined and s.a. in $L^{2}(\tilde{\Omega})$
and $P_{ext}$ is densely defined and s.a. in $L^{2}(J_{-}\cup J_{+})$,
if and only if $B$ (in $U(n)$) has the form

\begin{equation}
\left(\begin{array}{ccccc}
0 & \vline\\
\vdots & \vline &  & \huge\mbox{\ensuremath{\tilde{B}}}\\
0 & \vline\\
\hline e(\theta) & \vline & 0 & \cdots & 0
\end{array}\right)\label{tmp}
\end{equation}
for some $\theta\in\mathbb{R}/\mathbb{Z}$, and $\tilde{B}\in U(n-1)$.\end{prop}
\begin{proof}
Note that presentation (\ref{tmp}) for some $B\in U(n)$ implies
the boundary condition $f(\alpha_{n})=e(\theta)f(\beta_{1})$ for
$f\in\mathscr{D}(P_{B})$ when $P_{B}$ is the selfadjoint operator
in $L^{2}(\Omega)$ determined in Remark \ref{rem:n}. And, moreover,
the $\oplus$ sum decomposition (\ref{eq:tmp-39}) will be satisfied.

One checks that the converse holds as well. 
\end{proof}
Let $B=\left(\begin{array}{cc}
\boldsymbol{u} & B'\\
c & \boldsymbol{w}^{*}
\end{array}\right)\in U(n)$, where $\boldsymbol{u},\boldsymbol{w}\in\mathbb{C}^{n-1}$, and $c\in\mathbb{C}$.
In section \ref{sec:sp}, we consider the subset in $U(n)$ given
by $\boldsymbol{u}\neq0$ (see Corollary \ref{cor:B-1}), but it is
of interest to isolate the subfamily specified by $\boldsymbol{u}=0$. 

For $n=2$, the unitary one-parameter group $U_{B}(t)$, acting on
$L^{2}(\Omega)$, is unitarily equivalent to a direct sum of two one-parameter
groups, $T_{p}(t)$ and $T_{c}(t)$. See Figure \ref{fig:w0-1}. These
two one-parameter groups are obtained as follows:

(i) Start with $T(t)$, the usual one-parameter group of right-translation
by $t$, $f\mapsto f(\cdot-t)$. The subscript $p$ indicates periodic
translation, i.e., translation by $t$ modulo $1$, and with a phase
factor. Hence, $T_{p}(t)$ accounts for the boundstates. 

(ii) By contrast, the one-parameter group $T_{c}(t)$ is as follows:
Glue the rightmost endpoint of the interval $J_{-}$ starting at $-\infty$
to the leftmost endpoint in the interval $J_{+}$ out to $+\infty$.
These two finite end-points are merged onto a single point, say $0$,
on $\mathbb{R}$ (the whole real line.) This way, the one-parameter
group $T_{c}(t)$ becomes a summand of $U_{B}(t)$. $T_{c}(t)$ is
just translation in $L^{2}(\mathbb{R})$ modulo a phase factor $e(\varphi)=e^{i2\pi\varphi}$
at $x=0$.

For $n>2$ (Fig \ref{fig:decouple}), note the $\tilde{B}$-part ($\tilde{B}\in U(n-1)$)
in the orthogonal splitting
\[
U_{B}(t)\cong U_{\tilde{B}}(t)\oplus T_{c}(t),\: t\in\mathbb{R}
\]
in 
\[
L^{2}(\Omega)\cong L^{2}(\bigcup{}_{i=1}^{n-1}J_{i})\oplus L^{2}(\mathbb{R})
\]
allows for a rich variety of inequivalent unitary one-parameter groups
$U_{\tilde{B}}(t)$. The case $L^{2}(J_{1}\cup J_{2})$ is covered
in \cite{JPT11-1}. 

\begin{figure}
\includegraphics[scale=0.75]{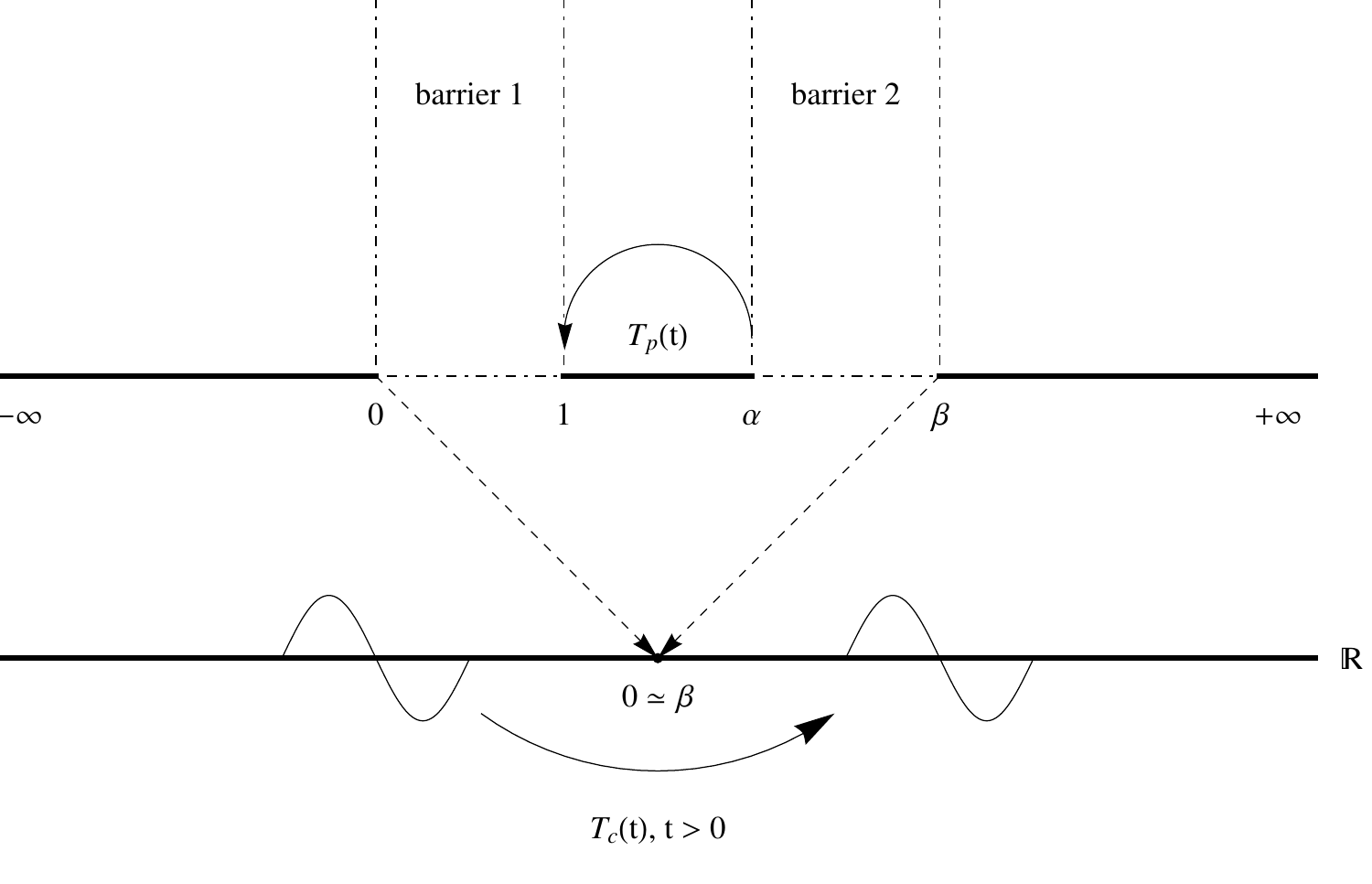}

\caption{\label{fig:w0-1}Infinite barriers ($n=2$). Boundstates in one interval. }
\end{figure}

\begin{figure}
\includegraphics[scale=0.75]{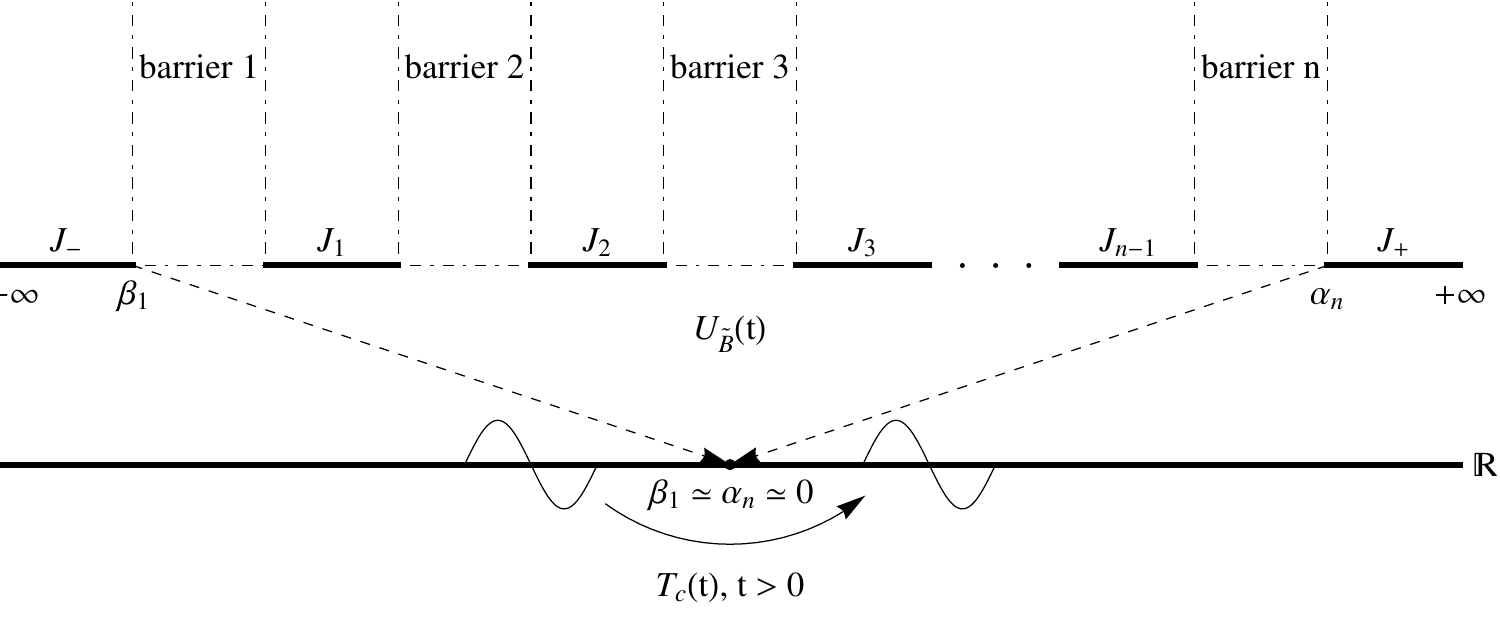}

\caption{\label{fig:decouple}The complement of $n$ bounded intervals in $\mathbb{R}$
($n>2$). Boundstates in the union of $n-1$ intervals, and tunneling.}
\end{figure}

\section{Spectral Theory\label{sec:sp}}

In this section we establish a number of theorems giving detailed
properties of each of the selfadjoint extension operators introduced
in subsection \ref{sub:bform} above. In Theorem \ref{thm:soln} (the
general case), we present the spectral resolutions as direct integrals:
We give explicit formulas for the associated generalized eigenfunctions;
and we study their properties. Among other things, we prove that they
have meromorphic extensions to the complex plane $\mathbb{C}$ minus
isolated poles, we give explicit formulas; and we study the scattering
poles, both those falling on the real axis, as well as the complex
poles.

We now turn to some detailed spectral analysis of the operators acting
on $L^{2}(\Omega)$. The first issue addressed may be summarized briefly
as follows: 

We study three equivalent conditions 1 through 3 below, where:

1. An element $B\in U(n)$ is decomposable as a unitary matrix, i.e.,
it has at least two non-trivial unitary summands $B_{1}$ and $B_{2}$.
Note however, that this definition presupposes a choice of an ordered
orthonormal basis (ONB) in $\mathbb{C}^{n}$. 

2. As a self-adjoint operator in $L^{2}(\Omega)$, $P_{B}$ is a corresponding
orthogonal sum of the two operators $P_{i}$, $i=1,2$.

3. The unitary one-parameter group $U_{B}(t)$ generated by $P_{B}$
decomposes as an orthogonal sum of two one-parameter groups with generators
$P_{i}$, each unitary in a proper subspace in $L^{2}(\Omega)$.

Some details about the corresponding summands in $L^{2}(\Omega)$,
infinite vs finite.

\textbf{The two infinite intervals:} If a particular $B$ in $U(n)$
is decomposable, then the corresponding summands in $L^{2}(\Omega)$
arise from lumping together the $L^{2}$ spaces of the intervals $J_{j}$,
$j$ from $0$ to $n$, each corresponding to a closed subspace in
$L^{2}(\Omega)$. But when lumping together these closed subspaces,
there is the following restriction: one of the two infinite half-lines
cannot occur alone: the two infinite half-lines must merge together.
The reason is that $L^{2}$ for an infinite half-line, by itself yields
deficiency indices $(1,0)$ or $(0,1)$. 

\textbf{The finite intervals:} If a subspace $L^{2}(J_{j})$ for $j$
from $1$ to $n-1$ occurs as a summand, there must be embedded point-spectrum
(called boundstates in physics), embedded in the continuum. 

Caution about \textquotedblleft{}matrix decomposition.\textquotedblright{}
The notion of decomposition for $B$ in $U(n)$ is basis-dependent
in a strong sense: it depending on prescribing an ONB in $\mathbb{C}^{n}$,
as an ordered set, so depends on permutations of a chosen basis. Hence
an analysis of an action of the permutation group $S_{n}$ enters.
So a particular property may hold before a permutation is applied,
but not after. 

This means that some $B$ in $U(n)$ might be decomposable in some
ordered ONB (in $\mathbb{C}^{n}$) , but such a decomposition may
\emph{not} lead to an associated ($P_{B}$, $L^{2}(\Omega)$)- decomposition.

For our matrix analysis we work with two separate notions, \textquotedblleft{}non-degenerate\textquotedblright{}
and \textquotedblleft{}indecomposable\textquotedblright{}, but a direct
comparison is not practical. The reason is that they naturally refer
to different orderings of the canonical ONB in $\mathbb{C}^{n}$.

\subsection{\label{sub:sp-eig}Spectrum and Eigenfunctions}

Fix $n>2$, let $\Omega$ be the exterior domain (\ref{eq:Omega}),
see Figure \ref{fig:eigen} below. 

\begin{figure}[H]
\includegraphics[scale=0.8]{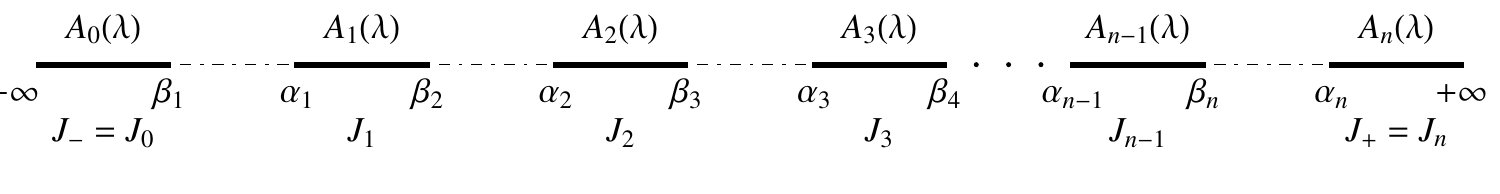}

\caption{\label{fig:eigen}$\psi_{\lambda}^{\left(B\right)}\left(x\right)=\left(\sum_{k=0}^{n}A_{k}\left(\lambda\right)\chi_{J_{k}}\left(x\right)\right)e_{\lambda}\left(x\right)$}
\end{figure}

Let $B=\left(b_{ij}\right)\in U\left(n\right)$. Define the generalized
eigenfunction by
\begin{equation}
\psi_{\lambda}^{\left(B\right)}\left(x\right):=\left(\sum_{k=0}^{n}A_{k}^{\left(B\right)}\left(\lambda\right)\chi_{J_{k}}\left(x\right)\right)e_{\lambda}\left(x\right),\:\lambda\in\mathbb{R}\label{eq:eigen}
\end{equation}
where $e_{\lambda}\left(x\right):=e^{i2\pi\lambda x}$. The function
\begin{equation}
\mathbf{a}(\cdot,\cdot):U(n)\times\mathbb{R}\rightarrow\mathbb{C}^{n+1}\label{eq:A}
\end{equation}
given by
\begin{equation}
\mathbf{a}\left(B,\lambda\right):=\left(A_{0}^{\left(B\right)}\left(\lambda\right),\ldots,A_{n}^{\left(B\right)}\left(\lambda\right)\right)\label{eq:A-1}
\end{equation}
satisfies the boundary condition
\begin{equation}
B\left(\begin{array}{c}
A_{0}^{\left(B\right)}\left(\lambda\right)e_{\lambda}\left(\beta_{1}\right)\\
A_{1}^{\left(B\right)}\left(\lambda\right)e_{\lambda}\left(\beta_{2}\right)\\
\vdots\\
A_{n-1}^{\left(B\right)}\left(\lambda\right)e_{\lambda}\left(\beta_{n}\right)
\end{array}\right)=\left(\begin{array}{c}
A_{1}^{\left(B\right)}\left(\lambda\right)e_{\lambda}\left(\alpha_{1}\right)\\
A_{2}^{\left(B\right)}\left(\lambda\right)e_{\lambda}\left(\alpha_{2}\right)\\
\vdots\\
A_{n}^{\left(B\right)}\left(\lambda\right)e_{\lambda}\left(\alpha_{n}\right)
\end{array}\right),\label{eq:bd}
\end{equation}
with matrix-action on the LHS in (\ref{eq:bd}). 

Setting 
\begin{alignat*}{1}
D_{\alpha}\left(\lambda\right) & :=diag\left(e_{\lambda}\left(\alpha_{1}\right),\ldots,e_{\lambda}\left(\alpha_{n}\right)\right)\\
D_{\beta}\left(\lambda\right) & :=diag\left(e_{\lambda}\left(\beta_{1}\right),\ldots,e_{\lambda}\left(\beta_{n}\right)\right)
\end{alignat*}
and let 
\begin{equation}
B_{\alpha,\beta}\left(\lambda\right):=D_{\alpha}^{*}\left(\lambda\right)BD_{\beta}\left(\lambda\right)\label{eq:bd-2}
\end{equation}
where $B$ is the matrix from (\ref{eq:bd}). Then (\ref{eq:bd})
can be written as 
\begin{equation}
B_{\alpha,\beta}\left(\lambda\right)\left(\begin{array}{c}
A_{0}^{\left(B\right)}\left(\lambda\right)\\
A_{1}^{\left(B\right)}\left(\lambda\right)\\
\vdots\\
A_{n-1}^{\left(B\right)}\left(\lambda\right)
\end{array}\right)=\left(\begin{array}{c}
A_{1}^{\left(B\right)}\left(\lambda\right)\\
A_{2}^{\left(B\right)}\left(\lambda\right)\\
\vdots\\
A_{n}^{\left(B\right)}\left(\lambda\right)
\end{array}\right),\label{eq:bd-1}
\end{equation}
where the matrix $B_{\alpha,\beta}(\lambda)$ is acting on the column
vector $\left(\begin{array}{c}
A_{0}^{\left(B\right)}\left(\lambda\right)\\
A_{1}^{\left(B\right)}\left(\lambda\right)\\
\vdots\\
A_{n-1}^{\left(B\right)}\left(\lambda\right)
\end{array}\right)$. In other words, with the definition (3.5), the two problems (\ref{eq:bd})
and (\ref{eq:bd-1}) are equivalent.
\begin{rem}
Specifically,
\begin{equation}
B_{\alpha,\beta}(\lambda)=\left(\begin{array}{ccc}
b_{11}\, e_{\lambda}\left(\beta_{1}-\alpha_{1}\right) & \cdots & b_{1n}\, e_{\lambda}\left(\beta_{n}-\alpha_{1}\right)\\
\vdots & \ddots & \vdots\\
b_{n1}\, e_{\lambda}\left(\beta_{1}-\alpha_{n}\right) & \cdots & b_{nn}\, e_{\lambda}\left(\beta_{n}-\alpha_{n}\right)
\end{array}\right)\label{eq:bd-3}
\end{equation}

\end{rem}

\subsection{The Role of $U(n)$}

The role of the group $U(n)$ of all unitary complex matrices is as
follows:

On $\mathbb{C}^{n}\times\mathbb{C}^{n}$ ($\simeq\mathbb{C}^{2}$),
we introduce the form $\boldsymbol{B}(\cdot,\cdot)$ from (\ref{eq:BoundaryForm3});
\begin{equation}
\boldsymbol{B}(z,\zeta)=\left\Vert z\right\Vert ^{2}-\left\Vert \zeta\right\Vert ^{2}\label{eq:Bform}
\end{equation}
where $\left\Vert z\right\Vert ^{2}=\sum_{1}^{n}|z_{j}|^{2}$ is the
usual Hilbert norm-squared. 

The projective space $P_{n,n}$ is the complex manifold \cite{Wel08}
consisting of all complex subspaces $L\subset\mathbb{C}^{n}\times\mathbb{C}^{n}$
such that $Pr_{1}L=\mathbb{C}^{n}$, and 
\begin{equation}
\boldsymbol{B}(z,\zeta)=0,\;\mbox{ for all }(z,\zeta)\in L.\label{eq:Bform-1}
\end{equation}
We use the notation $Pr_{1}(z,\zeta)=z$. 

The direction from $U(n)$ to $P_{n,n}$ is easy: If $B\in U(n)$,
set 
\begin{equation}
L(B):=\{(z,Bz)\:;\: z\in\mathbb{C}^{n}\};\label{eq:LB}
\end{equation}
it is then clear that $L(B)\in P_{n,n}$.

For the converse argument, show that $U(n)\ni B\mapsto L(B)$ maps
\emph{onto} $P_{n,n}$, see for example \cite{Wel08}.

\subsection{A Linear Algebra Problem}

To understand the coefficients $A_{i}(\lambda)$ in the representation
(\ref{eq:A-1}) of the generalized eigenfunctions, we will need a
little complex geometry and linear algebra.

Fix $n>2$, and let 
\begin{equation}
B=\left(\begin{array}{cc}
\boldsymbol{u} & B'\\
c & \boldsymbol{w}^{*}
\end{array}\right)\in U\left(n\right)\label{eq:B}
\end{equation}
where $\boldsymbol{u},\boldsymbol{w}\in\mathbb{C}^{n-1}$, and $c\in\mathbb{C}$. 
\begin{defn}
An element $B\in U(n)$ is said to be \emph{indecomposable} iff it
does not have a presentation 
\begin{equation}
B=\left(\begin{array}{cc}
B_{1}\\
 & B_{2}
\end{array}\right),\label{eq:B-4}
\end{equation}
$1\leq k<n$, $B_{1}\in U(k)$, $B_{2}\in U(n-k)$; i.e., iff $B$
as a transformation in $\mathbb{C}^{n}$ does not have a \emph{non-trivial}
splitting $B_{1}\oplus B_{2}$ as a sum of two unitaries.

(The blank blocks in the block-matrix from (\ref{eq:B-4}) are understood
to be a zero-operator between the respective subspaces. For more details,
see section \ref{sec:Decomposability}. )
\end{defn}

\begin{defn}
\label{def:degenerate}Let $B\in U(n)$ as in (\ref{eq:B}). We say
$B$ is \emph{degenerate} if $1\in sp(B')$, i.e., there exists $\boldsymbol{\zeta}\in\mathbb{C}^{n-1}\backslash\{0\}$
such that $B'\boldsymbol{\zeta}=\boldsymbol{\zeta}$. \end{defn}
\begin{thm}
\label{thm:LAP}Let $B=\left(\begin{array}{cc}
\boldsymbol{u} & B'\\
c & \boldsymbol{w}^{*}
\end{array}\right)\in U(n)$ as in (\ref{eq:B}), where $\boldsymbol{u},\boldsymbol{w}\in\mathbb{C}^{n-1}$,
and $c\in\mathbb{C}$. Then the solution to 
\begin{equation}
B\left(\begin{array}{c}
v_{0}\\
v_{1}\\
\vdots\\
v_{n-1}
\end{array}\right)=\left(\begin{array}{c}
v_{1}\\
v_{2}\\
\vdots\\
v_{n}
\end{array}\right)\label{eq:LAP}
\end{equation}
are as follows:
\begin{enumerate}
\item If $B$ is non-degenerate: 
\begin{equation}
\left(\begin{array}{c}
v_{0}\\
\vdots\\
v_{n}
\end{array}\right)=x_{0}\left(\begin{array}{c}
1\\
\left(I_{n-1}-B'\right)^{-1}\boldsymbol{u}\\
c+\left\langle \boldsymbol{w},\left(I_{n-1}-B'\right)^{-1}\boldsymbol{u}\right\rangle 
\end{array}\right)\label{eq:soln}
\end{equation}
for some constant $x_{0}\in\mathbb{C}$.
\item If $B$ is degenerate: let $\boldsymbol{\zeta}\in\ker\left(I_{n-1}-B'\right)$,
$\zeta\in\mathbb{C}^{n-1}\backslash\{0\}$, then

\begin{enumerate}
\item If $\boldsymbol{u}$ not in the range of $I_{n-1}-B'$, \textup{
\begin{equation}
\left(\begin{array}{c}
v_{0}\\
\vdots\\
v_{n}
\end{array}\right)=\left(\begin{array}{c}
0\\
\boldsymbol{\zeta}\\
\left\langle \boldsymbol{w},\boldsymbol{\zeta}\right\rangle 
\end{array}\right);\label{eq:soln-1}
\end{equation}
}
\item For $\boldsymbol{u}$ in the range of $I_{n-1}-B'$, 
\begin{equation}
\left(\begin{array}{c}
v_{0}\\
\vdots\\
v_{n}
\end{array}\right)=x_{0}\left(\begin{array}{c}
1\\
\boldsymbol{\zeta}_{0}\\
c+\left\langle \boldsymbol{w},\boldsymbol{\zeta}_{0}\right\rangle 
\end{array}\right)+\left(\begin{array}{c}
0\\
\boldsymbol{\zeta}\\
\left\langle \boldsymbol{w},\boldsymbol{\zeta}\right\rangle 
\end{array}\right)\label{eq:soln-2}
\end{equation}
for some constant $x_{0}\in\mathbb{C}$ and some fixed $\boldsymbol{\zeta}_{0}$
such that $\mathbf{u}=(I_{n-1}-B')\boldsymbol{\zeta}_{0}$. 
\end{enumerate}
\end{enumerate}
\end{thm}
\begin{proof}
Note (\ref{eq:LAP}) is equivalent to 
\begin{equation}
\boldsymbol{u}\: v_{0}+B'\left(\begin{array}{c}
v_{1}\\
\vdots\\
v_{n-1}
\end{array}\right)=\left(\begin{array}{c}
v_{1}\\
\vdots\\
v_{n-1}
\end{array}\right),\;\mbox{and}\label{eq:tmp}
\end{equation}
\begin{equation}
c\: v_{0}+\left\langle \boldsymbol{w},\left(\begin{array}{c}
v_{1}\\
\vdots\\
v_{n-1}
\end{array}\right)\right\rangle =v_{n}.\label{eq:tmp-1}
\end{equation}

If $1\notin sp(B')$, solving (\ref{eq:tmp}) \& (\ref{eq:tmp-1})
gives rise to (\ref{eq:soln}). The remaining cases are similar. \end{proof}
\begin{example}
Suppose $n=3$, then 
\begin{equation}
B=\left(\begin{array}{ccc}
0 & 1 & 0\\
0 & 0 & -1\\
1 & 0 & 0
\end{array}\right)\label{eq:deg}
\end{equation}
is degenerate. Here, $B'=\left(\begin{array}{cc}
1 & 0\\
0 & -1
\end{array}\right)$ and $B'\zeta=\zeta$, where $\boldsymbol{\zeta}=\left(\begin{array}{c}
1\\
0
\end{array}\right)$. 
\end{example}

\begin{example}
For $n=3$, let 
\begin{equation}
B=\left(\begin{array}{ccc}
0 & 0 & -1\\
0 & -1 & 0\\
1 & 0 & 0
\end{array}\right)\label{eq:deg-1}
\end{equation}
so that $B'=\left(\begin{array}{cc}
0 & -1\\
-1 & 0
\end{array}\right)$. Note $B'\zeta=\zeta$, where $\boldsymbol{\zeta}=\left(\begin{array}{c}
1\\
-1
\end{array}\right)$; hence $B$ is degenerate.
\begin{example}
For $n=4,$ 
\[
B=\left(\begin{array}{cccc}
0 & 1 & 0 & 0\\
1/2 & 0 & 1/\sqrt{2} & 1/2\\
1/2 & 0 & -1/\sqrt{2} & 1/2\\
1/\sqrt{2} & 0 & 0 & -1/\sqrt{2}
\end{array}\right)
\]
is degenerate and 
\[
I_{3}-B'=\left(\begin{array}{ccc}
0 & 0 & 0\\
0 & 1-\frac{1}{\sqrt{2}} & -\frac{1}{2}\\
0 & \frac{1}{\sqrt{2}} & \frac{1}{2}
\end{array}\right).
\]
Hence $\mathbf{u}=\left(\begin{array}{c}
0\\
1/2\\
1/2
\end{array}\right)$ is in the range of $I_{3}-B'$ and consequently we get an example
for case (2)(b) of Theorem \ref{thm:LAP}. 
\end{example}
\end{example}

\begin{example}
For $n=2$, let 
\[
B=\left(\begin{array}{cc}
a & b\\
-\overline{b} & \overline{a}
\end{array}\right)\in SU(2),
\]
i.e., $\left|a\right|^{2}+\left|b\right|^{2}=1$. Suppose 
\begin{equation}
\left(\begin{array}{cc}
a & b\\
-\overline{b} & \overline{a}
\end{array}\right)\left(\begin{array}{c}
v_{0}\\
v_{1}
\end{array}\right)=\left(\begin{array}{c}
v_{1}\\
v_{2}
\end{array}\right).\label{eq:tmp-3}
\end{equation}
That is, 
\begin{alignat*}{1}
av_{0}+bv_{1} & =v_{1}\\
-\overline{b}v_{0}+\overline{a}v_{1} & =v_{2}.
\end{alignat*}
If $b\neq1$ (non-degenerate), then 
\[
\left(\begin{array}{c}
v_{0}\\
v_{1}\\
v_{2}
\end{array}\right)=x_{0}\left(\begin{array}{c}
\underset{}{1}\\
\underset{}{\frac{a}{1-b}}\\
\frac{1-\overline{b}}{1-b}
\end{array}\right),\; x_{0}\in\mathbb{C};
\]
If $b=1$ (degenerate, $a=0$), the solution space is two dimensional,
given by 
\[
x_{0}\left(\begin{array}{c}
1\\
0\\
-1
\end{array}\right)+y_{0}\left(\begin{array}{c}
0\\
1\\
0
\end{array}\right),\; x_{0},y_{0}\in\mathbb{C}.
\]

\end{example}

\subsection{The Generalized Eigenfunctions}

We apply results in the previous section to the generalized eigenfunction
in (\ref{eq:eigen})-(\ref{eq:GEF-1}). 
\begin{thm}
\label{thm:soln}Fix $B\in U(n)$, and let 
\begin{equation}
\psi_{\lambda}^{\left(B\right)}\left(x\right)=\left(\sum_{k=0}^{n}A_{k}^{(B)}\left(\lambda\right)\chi_{J_{k}}\left(x\right)\right)e_{\lambda}\left(x\right),\:\lambda\in\mathbb{R}\label{eq:GEF}
\end{equation}
be the generalized eigenfunction in (\ref{eq:eigen}) satisfying the
boundary condition (\ref{eq:bd}). Then $\mathbf{a}\left(B,\lambda\right)=\left(A_{0}^{(B)}\left(\lambda\right),\ldots,A_{n}^{(B)}\left(\lambda\right)\right)$
in (\ref{eq:A-1}) is a solution to 
\begin{equation}
B_{\alpha,\beta}\left(\lambda\right)\left(\begin{array}{c}
A_{0}^{\left(B\right)}\left(\lambda\right)\\
\vdots\\
A_{n-1}^{\left(B\right)}\left(\lambda\right)
\end{array}\right)=\left(\begin{array}{c}
A_{1}^{\left(B\right)}\left(\lambda\right)\\
\vdots\\
A_{n}^{\left(B\right)}\left(\lambda\right)
\end{array}\right);\label{eq:gef-1}
\end{equation}
where $B_{\alpha,\beta}=D_{\alpha}^{*}BD_{\beta}$, see (\ref{eq:bd-2})
and (\ref{eq:bd-3}). Moreover, writing 
\[
B_{\alpha,\beta}\left(\lambda\right)=\left(\begin{array}{cc}
\underset{}{\boldsymbol{u}\left(\lambda\right)} & \underset{}{B'_{\alpha,\beta}\left(\lambda\right)}\\
c\left(\lambda\right) & \boldsymbol{w}\left(\lambda\right)^{*}
\end{array}\right)
\]
where
\[
c\left(\lambda\right)=b_{n,1}e_{\lambda}\left(\beta_{1}-\alpha_{n}\right),
\]
 
\[
\boldsymbol{u}\left(\lambda\right)=\left(\begin{array}{c}
b_{11}\: e_{\lambda}\left(\beta_{1}-\alpha_{1}\right)\\
b_{21}\: e_{\lambda}\left(\beta_{1}-\alpha_{2}\right)\\
\vdots\\
b_{n-1,1}\: e_{\lambda}\left(\beta_{1}-\alpha_{n-1}\right)
\end{array}\right),\;\boldsymbol{w}\left(\lambda\right)=\left(\begin{array}{c}
b_{n,2}\: e_{\lambda}\left(\beta_{2}-\alpha_{n}\right)\\
b_{n,3}\: e_{\lambda}\left(\beta_{3}-\alpha_{n}\right)\\
\vdots\\
b_{n,n}\: e_{\lambda}\left(\beta_{n}-\alpha_{n}\right)
\end{array}\right),
\]
and
\begin{equation}
B'_{\alpha,\beta}\left(\lambda\right)=\left(\begin{array}{ccc}
b_{12}\, e_{\lambda}\left(\beta_{2}-\alpha_{1}\right) & \cdots & b_{1n}\, e_{\lambda}\left(\beta_{n}-\alpha_{1}\right)\\
\vdots & \ddots & \vdots\\
b_{n-1,2}\, e_{\lambda}\left(\beta_{1}-\alpha_{n-1}\right) & \cdots & b_{n-1,n}\, e_{\lambda}\left(\beta_{n}-\alpha_{n-1}\right)
\end{array}\right);\label{eq:B-2}
\end{equation}
then the solution to (\ref{eq:gef-1}) are as follows:

Setting \textup{
\begin{equation}
\Lambda_{p}=\left\{ \lambda\in\mathbb{R}\left|\right.\det\left(I_{n-1}-B'_{\alpha,\beta}(\lambda)\right)=0\right\} .\label{eq:pt-1}
\end{equation}
}
\begin{enumerate}
\item If $\Lambda_{p}=\phi$, then $B_{\alpha,\beta}(\lambda)$ is non-degenerate,
and 
\begin{equation}
\left(\begin{array}{c}
A_{0}^{\left(B\right)}\left(\lambda\right)\\
\vdots\\
A_{n}^{\left(B\right)}\left(\lambda\right)
\end{array}\right)=x_{0}\left(\begin{array}{c}
1\\
\left(I_{n-1}-B'_{\alpha,\beta}(\lambda)\right)^{-1}\boldsymbol{u}\left(\lambda\right)\\
c\left(\lambda\right)+\left\langle \boldsymbol{w}\left(\lambda\right),\left(I_{n-1}-B'_{\alpha,\beta}\left(\lambda\right)\right)^{-1}\boldsymbol{u}\left(\lambda\right)\right\rangle 
\end{array}\right)\label{eq:soln-3}
\end{equation}
for some constant $x_{0}\in\mathbb{C}$. The points $\lambda\in\Lambda_{p}$
from (\ref{eq:pt-1}) are the real poles in the functions $A_{j}$
from (\ref{eq:soln-3}).
\item Suppose $\Lambda_{p}\neq\phi$. For all\textup{ $\lambda\in\Lambda_{p}$,}
$B_{\alpha,\beta}(\lambda)$ \textup{is degenerate, and there is }$\zeta\left(\lambda\right)\in\mathbb{C}^{n-1}\backslash\{0\}$,
such that $\zeta\left(\lambda\right)\in\ker\left(I_{n-1}-B'_{\alpha,\beta}(\lambda)\right)$.
Then

\begin{enumerate}
\item If $\boldsymbol{u}\left(\lambda\right)$ is not in the range of $I_{n-1}-B'_{\alpha,\beta}(\lambda)$
and $\left(I_{n-1}-B'_{\alpha,\beta}(\lambda)\right)\boldsymbol{\zeta}_{0}\left(\lambda\right)=u(\lambda)$
\textup{
\[
\left(\begin{array}{c}
A_{0}^{\left(B\right)}\left(\lambda\right)\\
\vdots\\
A_{n}^{\left(B\right)}\left(\lambda\right)
\end{array}\right)=\left(\begin{array}{c}
0\\
\boldsymbol{\zeta}\left(\lambda\right)\\
\left\langle \boldsymbol{w}\left(\lambda\right),\boldsymbol{\zeta}\left(\lambda\right)\right\rangle 
\end{array}\right);
\]
}
\item If $\boldsymbol{u}\left(\lambda\right)$ is in the range of \textup{$I_{n-1}-B'_{\alpha,\beta}(\lambda)$
and} 
\[
\left(\begin{array}{c}
A_{0}^{\left(B\right)}\left(\lambda\right)\\
\vdots\\
A_{n}^{\left(B\right)}\left(\lambda\right)
\end{array}\right)=x_{0}\left(\begin{array}{c}
1\\
\boldsymbol{\zeta}_{0}\left(\lambda\right)\\
c\left(\lambda\right)+\left\langle \boldsymbol{w}\left(\lambda\right),\boldsymbol{\zeta}_{0}\left(\lambda\right)\right\rangle 
\end{array}\right)+\left(\begin{array}{c}
0\\
\boldsymbol{\zeta}\left(\lambda\right)\\
\left\langle \boldsymbol{w}\left(\lambda\right),\boldsymbol{\zeta}\left(\lambda\right)\right\rangle 
\end{array}\right)
\]
for some constant $x_{0}\in\mathbb{C}$. In particular, $\Lambda_{p}$
consists of eigenvalues for $P_{B}.$  
\end{enumerate}
\end{enumerate}
\end{thm}
\begin{proof}
This follows directly from Theorem \ref{thm:LAP}.\end{proof}
\begin{cor}
\label{cor:pt}Fix a system of interval endpoints $\boldsymbol{\alpha}=\left(\alpha_{i}\right)$
and $\boldsymbol{\beta}=\left(\beta_{i}\right)$. Then the subset
of $\mathbb{R}$
\begin{equation}
\Lambda_{p}=\left\{ \lambda\in\mathbb{R}\left|\right.\det\left(I_{n-1}-B'_{\alpha,\beta}(\lambda)\right)=0\right\} \label{eq:pt}
\end{equation}
consists of isolated points, i.e., has no accumulation points.\end{cor}
\begin{proof}
It follows from (\ref{eq:B-2}) that the function
\begin{equation}
\lambda\mapsto D\left(\lambda\right)=\det\left(I_{n-1}-B'_{\alpha,\beta}\left(\lambda\right)\right)\label{eq:tmp-9}
\end{equation}
is entire analytic, i.e., is a restriction to $\mathbb{R}$ of an
entire analytic function.

To see this, note that $\lambda\mapsto B'_{\alpha,\beta}(\lambda)$
in (\ref{eq:B-2}) is entire; and since the determinant is multilinear,
it follows $D(\cdot)$ in (\ref{eq:tmp-9}) is also entire. Since
it is non-constant the properties of $\Lambda_{p}$ (see (\ref{eq:pt-1}))
follow from analytic function theory.\end{proof}
\begin{cor}
\label{cor:pole}Let $\Omega$ be fixed as before, and select a $B\in U(n)$;
then the functions $A_{j}^{(B)}(\cdot)$ in (\ref{eq:soln-3}) and
(\ref{eq:GEF-1}) have meromorphic extensions to $\mathbb{C}$; the
extension is obtained by replacing $\lambda$ in (\ref{eq:B-2}),
(\ref{eq:pt-1}) and (\ref{eq:soln-3}) with $z\in\mathbb{C}$. The
poles in the function $\mathbb{C}\ni z\mapsto A_{j}^{(B)}(z)$ occur
at the roots 
\begin{equation}
\det\left(I_{n-1}-B'_{\boldsymbol{\alpha},\boldsymbol{\beta}}(z)\right)=0\label{eq:det-8}
\end{equation}
and the embedded point-spectrum of the selfadjoint operator $P_{B}$
(in $L^{2}(\Omega)$) are the \uline{real} solutions to (\ref{eq:det-8}).\end{cor}
\begin{proof}
The assertions in the corollary follow directly from the formulas
(\ref{eq:B-2}) and (\ref{eq:soln-3}) in Theorem \ref{thm:soln}. \end{proof}
\begin{rem}
To find the meromorphic extension of the function
\begin{equation}
\mathbb{R}\ni\lambda\mapsto\left(I_{n-1}-B'_{\boldsymbol{\alpha},\boldsymbol{\beta}}(\lambda)\right)^{-1}\label{eq:ext}
\end{equation}
from (\ref{eq:soln-3}) in Theorem \ref{thm:soln}, we proceed as
follow: Extend (\ref{eq:ext}) by formally substituting $z\in\mathbb{C}$
for $\lambda$; and then proceed to compute the formal power series
expansion for the function
\begin{equation}
\mathbb{C}\ni z\mapsto R_{\boldsymbol{\alpha},\boldsymbol{\beta}}(z,B'):=\left(I_{n-1}-D_{\boldsymbol{\alpha}}(-z)B'D_{\boldsymbol{\beta}}(z)\right)^{-1}\label{eq:ext-1}
\end{equation}
in the complement of the set of isolated poles. (The $(n-1)\times(n-1)$
matrix $B'$ in (\ref{eq:ext-1}) is fixed, but it is assumed to come
from some $B=\left(\begin{array}{cc}
\boldsymbol{u} & B'\\
c & \boldsymbol{w}^{*}
\end{array}\right)\in U\left(n\right)$ as in (\ref{eq:B}).) For iteration of the $\frac{d}{dz}$-derivatives
in (\ref{eq:ext-1}), it will be convenient to introduce $\frac{-1}{2\pi i}\frac{d}{dz}$,
$L_{\boldsymbol{\alpha}}=\mbox{diag}\left(\alpha_{j}\right)_{j=1}^{n-1}$,
$L_{\boldsymbol{\beta}}=\mbox{diag}\left(\beta_{j}\right)_{j=2}^{n}$,
and 
\begin{equation}
\delta_{\boldsymbol{\alpha},\boldsymbol{\beta}}(M):=ML_{\boldsymbol{\beta}}-L_{\boldsymbol{\alpha}}M\label{eq:del}
\end{equation}
defined for all $(n-1)\times(n-1)$ matrices $M$.

Then in the complement of the complex poles of $R_{\boldsymbol{\alpha},\boldsymbol{\beta}}(z,B')$
in (\ref{eq:ext-1}), we get 
\begin{equation}
\left(\frac{-i}{2\pi i}\frac{d}{dz}\right)R_{\boldsymbol{\alpha},\boldsymbol{\beta}}(z,B')=R_{\boldsymbol{\alpha},\boldsymbol{\beta}}(z,B')\delta_{\boldsymbol{\alpha},\boldsymbol{\beta}}(B')R_{\boldsymbol{\alpha},\boldsymbol{\beta}}(z,B').\label{eq:del-1}
\end{equation}
And, as a result the higher order complex derivatives $\left(\frac{-i}{2\pi i}\frac{d}{dz}\right)^{n}$
may be obtained from (\ref{eq:del-1}), and a recursion which we leave
to the reader. It introduces a little combinatorics and an iteration
of $\delta_{\boldsymbol{\alpha},\boldsymbol{\beta}}$ in (\ref{eq:del}). 

In conclusion, we note that the complex extension
\[
\mathbb{C}\ni z\mapsto R_{\boldsymbol{\alpha},\boldsymbol{\beta}}(z,B')
\]
is entire analytic in the complement of its isolated poles.\end{rem}
\begin{example}
\label{ex:pole}Let $n=2$, and fix $-\infty<\beta_{1}<\alpha_{1}<\beta_{2}<\alpha_{2}<\infty$.
Let $B=\left(\begin{array}{cc}
a & b\\
-\overline{b} & \overline{a}
\end{array}\right)$, where $a,b\in\mathbb{C}$, and $\left|a\right|^{2}+\left|b\right|^{2}=1$.
Then 
\[
D\left(\lambda\right)=1-b\: e_{\lambda}(\beta_{2}-\alpha_{1}),\:\lambda\in\mathbb{R}.
\]
As a result, 
\[
\Lambda_{p}=\phi\Longleftrightarrow\left|b\right|<1\Longleftrightarrow a\neq0.
\]
If $a=0$, then there is a $\theta\in\mathbb{R}$, such that $b=e(\theta)$;
and then
\[
\Lambda_{p}=\left(\beta_{2}-\alpha_{1}\right)^{-1}\left(-\theta+\mathbb{Z}\right).
\]

\end{example}
\begin{rem}
Note that the complex poles discussed in Corollary \ref{cor:pole}
for Example \ref{ex:pole} ($b\neq0$) may be presented as follows:
Select a branch of the complex logarithm ``$\log$''; then the complex
poles are 
\begin{equation}
\left\{ z\in\mathbb{C},\; z\in\frac{1}{\mbox{length}\,(J_{1})}\left(\frac{-1}{2\pi i}\log b+\mathbb{Z}\right)\right\} .\label{eq:pole}
\end{equation}

\end{rem}

\subsection{The Groups $U(n)$ and $U(n-1)$}

In the proof of Theorem \ref{thm:LAP}, we considered the following
operator/matrix block presentation of elements $B\in U(n)$,
\begin{equation}
B=\left(\begin{array}{cc}
\boldsymbol{u} & B'\\
c & \boldsymbol{w}^{*}
\end{array}\right)\label{eq:B-1}
\end{equation}
where $\boldsymbol{u},\boldsymbol{w}\in\mathbb{C}^{n-1}$, $c\in\mathbb{C}$,
and $B'$ is the $(n-1)\times(n-1)$ matrix in the NE corner in (\ref{eq:B-1}). 

We consider the coordinates in $\boldsymbol{u}$ as the matrix entries
\begin{equation}
b_{i1}=u_{i},\:1\leq i\leq n-1.\label{eq:tmp-2}
\end{equation}
For $c\in\mathbb{C}$, we have
\begin{equation}
b_{n1}=c.\label{eq:tmp-6}
\end{equation}
The notation $\boldsymbol{w}^{*}$ indicates that $\boldsymbol{w}$
is a row-vector; we have
\[
b_{n,j+1}=w_{j},\:1\leq j\leq n-1.
\]

Finally we denote the Hilbert inner product $\left\langle \cdot,\cdot\right\rangle $
and it is taken to be linear in the second variable. With this convention
we have
\[
\boldsymbol{w}^{*}\boldsymbol{u}=\left\langle \boldsymbol{w},\boldsymbol{u}\right\rangle \in\mathbb{C}.
\]

\begin{thm}
\label{thm:action}If $B=\left(\begin{array}{cc}
\boldsymbol{u} & B'\\
c & \boldsymbol{w}^{*}
\end{array}\right)\in U\left(n\right)$ and $g\in U\left(n-1\right)$, assume $1\notin sp\left(B'\right)$.
Then 
\begin{equation}
\alpha_{g}(B):=\left(\begin{array}{cc}
\underset{}{g\boldsymbol{u}} & \underset{}{gB'g^{-1}}\\
c & \left(g\boldsymbol{w}\right)^{*}
\end{array}\right)\in U\left(n\right).\label{eq:act}
\end{equation}
If $v=\left(v_{i}\right)_{i=0}^{n}\in\mathbb{C}^{n+1}$ solves 
\begin{equation}
B\left(\begin{array}{c}
v_{0}\\
v_{1}\\
\vdots\\
v_{n-1}
\end{array}\right)=\left(\begin{array}{c}
v_{1}\\
v_{2}\\
\vdots\\
v_{n}
\end{array}\right)\label{eq:lap}
\end{equation}
then
\begin{equation}
\boldsymbol{v}^{g}:=\left(\begin{array}{c}
v_{0}\\
g\left(\begin{array}{c}
v_{1}\\
v_{2}\\
\vdots\\
v_{n-1}
\end{array}\right)\\
v_{n}
\end{array}\right)\label{eq:tmp-7}
\end{equation}
solves 
\begin{equation}
\alpha_{g}(B)\left(\begin{array}{c}
v_{0}\\
g\left(\begin{array}{c}
v_{1}\\
v_{2}\\
\vdots\\
v_{n-1}
\end{array}\right)
\end{array}\right)=\left(\begin{array}{c}
\begin{array}{c}
g\left(\begin{array}{c}
v_{1}\\
v_{2}\\
\vdots\\
v_{n-1}
\end{array}\right)\end{array}\\
v_{n}
\end{array}\right)\label{eq:tmp-8}
\end{equation}
\end{thm}
\begin{proof}
Since $B$ in (\ref{eq:B-1}) is in $U(n)$, we get the following
presentation of the $\mathbb{C}^{n}$ norm: 
\[
\left\Vert x_{0}\boldsymbol{u}+B'\boldsymbol{x}\right\Vert ^{2}+\left|cx_{0}+\left\langle \boldsymbol{w},\boldsymbol{x}\right\rangle \right|^{2}=\left|x_{0}\right|^{2}+\left\Vert \boldsymbol{x}\right\Vert ^{2}
\]
for all $\left(\begin{array}{c}
x_{0}\\
\boldsymbol{x}
\end{array}\right)\in\mathbb{C}^{n}$. We choose coordinates such that $x_{0}\in\mathbb{C}$, and $\boldsymbol{x}\in\mathbb{C}^{n-1}$.
Since $g\in U(n-1)$, we get
\[
\left\Vert x_{0}g\boldsymbol{u}+gB'g^{-1}\boldsymbol{x}\right\Vert ^{2}+\left|cx_{0}+\left\langle g\boldsymbol{w},\boldsymbol{x}\right\rangle \right|^{2}=\left|x_{0}\right|^{2}+\left\Vert \boldsymbol{x}\right\Vert ^{2}
\]
for all $\left(\begin{array}{c}
x_{0}\\
\boldsymbol{x}
\end{array}\right)\in\mathbb{C}^{n}$. The assertion in (\ref{eq:act}) follows from this.

We now use Theorem \ref{thm:LAP} to solve the problem for $\alpha_{g}(B)$.
Hence the solution $\boldsymbol{v}^{g}$ to the $\alpha_{g}(B)$ problem
is
\begin{alignat*}{1}
\left(\begin{array}{c}
v_{0}^{g}\\
\\
\left(\begin{array}{c}
v_{1}^{g}\\
\vdots\\
v_{n-1}^{g}
\end{array}\right)\\
\\
v_{n}^{g}
\end{array}\right) & =v_{0}\left(\begin{array}{c}
1\\
\left(I_{n-1}-gB'g^{-1}\right)^{-1}g\boldsymbol{u}\\
c+\left\langle g\boldsymbol{w},\left(I_{n-1}-gB'g^{-1}\right)^{-1}g\boldsymbol{u}\right\rangle 
\end{array}\right)\\
 & =v_{0}\left(\begin{array}{c}
1\\
g\left(I_{n-1}-B'\right)^{-1}\boldsymbol{u}\\
c+\left\langle \boldsymbol{w},\left(I_{n-1}-B'\right)^{-1}\boldsymbol{u}\right\rangle 
\end{array}\right)
\end{alignat*}
which is the desired conclusion in (\ref{eq:tmp-7}).

Inside the computation, we use the following formula from matrix theory
\[
\left(I_{n-1}-gB'g^{-1}\right)^{-1}=g\left(I_{n-1}-B'\right)^{-1}g^{-1}
\]
and as a result
\begin{alignat*}{1}
\left\langle g\boldsymbol{w},\left(I_{n-1}-gB'g^{-1}\right)^{-1}g\boldsymbol{u}\right\rangle  & =\left\langle g\boldsymbol{w},g\left(I_{n-1}-B'\right)^{-1}\boldsymbol{u}\right\rangle \\
 & =\left\langle \boldsymbol{w},\left(I_{n-1}-B'\right)^{-1}\boldsymbol{u}\right\rangle 
\end{alignat*}
where we used $g^{*}g=I_{n-1}$, i.e., $g\in U(n-1)$. \end{proof}
\begin{cor}
Let $B=\left(\begin{array}{cc}
\boldsymbol{u} & B'\\
c & \boldsymbol{w}^{*}
\end{array}\right)$ be such that, for some $g\in SU(n-1)$, we have $gB'g^{-1}=\mbox{diag}(z_{j})_{j=1}^{n-1}$,
$z_{j}\in\mathbb{C}$, $\left|z_{j}\right|\leq1$, then
\begin{equation}
\det\left(I_{n-1}-\left(gB'g^{-1}\right)_{\boldsymbol{\alpha},\boldsymbol{\beta}}(\lambda)\right)=\prod_{k=1}^{n-1}\left(1-z_{k}\, e(\lambda L_{k})\right)\label{eq:det-9}
\end{equation}
where $L_{k}=\mbox{length}(J_{k})$, $1\leq k<n$.\end{cor}
\begin{lem}
Given $B\in U(n)$, then the following are equivalent:
\begin{enumerate}
\item $\alpha_{g}(B)=B$, for all $g\in U(n-1)$, and
\item $B$ has the form 
\begin{equation}
B=\left(\begin{array}{cc}
\boldsymbol{0} & I_{n-1}\\
c & \boldsymbol{0}
\end{array}\right),\: c\in\mathbb{C},\left|c\right|=1.\label{eq:tmp-10}
\end{equation}

\end{enumerate}
\end{lem}
\begin{proof}
Immediate from the definition of $\alpha_{g}$, (\ref{eq:act}), i.e.,
\begin{equation}
\alpha_{g}\left(\left(\begin{array}{cc}
\boldsymbol{u} & B'\\
c & \boldsymbol{w}^{*}
\end{array}\right)\right)=\left(\begin{array}{cc}
g\boldsymbol{u} & gB'g^{-1}\\
c & \left(g\boldsymbol{w}\right)^{*}
\end{array}\right);\label{eq:act-1}
\end{equation}
see Corollary \ref{cor:pt}.\end{proof}
\begin{lem}
\label{lem:zeta}If $B$ is degenerate and $\boldsymbol{\zeta}\in\mathbb{C}^{n-1}$
is an eigenvector of $B'$ with eigenvalue $1$, then $P_{0}\boldsymbol{\zeta}=P_{n}\boldsymbol{\zeta}=0$.\end{lem}
\begin{proof}
Note $B'=P_{n}^{\perp}BP_{0}^{\perp}$ is contractive, and $B'^{*}=P_{0}^{\perp}B'^{*}P_{n}^{\perp}$;
and $B'\zeta=\zeta$ implies that $B'^{*}\boldsymbol{\zeta}=\boldsymbol{\zeta}$.
Hence $P_{n}^{\perp}\boldsymbol{\zeta}=\boldsymbol{\zeta}$, $P_{0}^{\perp}\boldsymbol{\zeta}=\boldsymbol{\zeta}$,
and so $P_{n}\boldsymbol{\zeta}=P_{0}\boldsymbol{\zeta}=0$. \end{proof}
\begin{cor}
\label{cor:zeta}Let $B$ be degenerate. Then $\boldsymbol{u}$ and
$\boldsymbol{w}$ are orthogonal to $\boldsymbol{\zeta}$, where $\boldsymbol{\zeta}$
is an eigenvector as above; i.e., 
\[
\left\langle \boldsymbol{u},\boldsymbol{\zeta}\right\rangle =0=\left\langle \boldsymbol{w},\boldsymbol{\zeta}\right\rangle .
\]
\end{cor}
\begin{proof}
By definition
\[
Be_{1}=\left(\begin{array}{c}
\boldsymbol{u}\\
c
\end{array}\right)\in\mathbb{C}^{n-1}\oplus\mathbb{C}.
\]
But recall $P_{1}\boldsymbol{\zeta}=P_{n}\boldsymbol{\zeta}=0$, by
Lemma \ref{lem:zeta}. Then
\[
\left\langle \boldsymbol{u},\boldsymbol{\zeta}\right\rangle =\left\langle \boldsymbol{u},B'\boldsymbol{\zeta}\right\rangle =\left\langle \boldsymbol{u},B\boldsymbol{\zeta}\right\rangle =\left\langle B^{*}\boldsymbol{u},\boldsymbol{\zeta}\right\rangle =\left\langle e_{1},\boldsymbol{\zeta}\right\rangle =0
\]
since $P_{1}\boldsymbol{\zeta}=0$. The same argument yields $\left\langle \boldsymbol{w},\boldsymbol{\zeta}\right\rangle =0$. 
\end{proof}

\begin{thm}
\label{thm:Bdecomp}Set $J_{0}=J_{-}$, $J_{n}=J_{+}$. Let $B$ be
determined by (\ref{eq:tmp-10}) that is 
\begin{equation}
B=\left(\begin{array}{cc}
\boldsymbol{0} & I_{n-1}\\
c & \boldsymbol{0}
\end{array}\right),\: c\in\mathbb{C},\left|c\right|=1.\label{eq:tmp-10-1}
\end{equation}
then the continuous spectrum of $P_{B}$ is the real line and the
discrete spectrum of $P_{B}$ is $\bigcup_{k=1}^{n-1}\frac{1}{\ell_{k}}\mathbb{Z},$
where $\ell_{k}=\beta_{k+1}-\alpha_{k}$ is the length of the $k$th
bounded interval. The multiplicity of each eigenvalue $\lambda$ is
$\#\left\{ 1\leq k\leq n-1\mid\ell_{k}\lambda\in\mathbb{Z}\right\} .$
Hence, $0$ is an eigenvalue with multiplicity $n-1$ and counting
multiplity the discrete spectrum has uniform density $\sum_{k=1}^{n-1}\ell_{k}$,
in the sense that, for any $a$ we have 
\[
\frac{\text{number of eigenvalues in }[a-n,a+n]}{2n}\to\sum_{k=1}^{n-1}\ell_{k}
\]
as $n\to\infty.$ \end{thm}
\begin{proof}
Note
\begin{equation}
B\left(\begin{array}{c}
A_{0}\left(\lambda\right)e_{\lambda}\left(\beta_{1}\right)\\
A_{1}\left(\lambda\right)e_{\lambda}\left(\beta_{2}\right)\\
\vdots\\
A_{n-1}\left(\lambda\right)e_{\lambda}\left(\beta_{n}\right)
\end{array}\right)=\left(\begin{array}{c}
A_{1}\left(\lambda\right)e_{\lambda}\left(\alpha_{1}\right)\\
A_{2}\left(\lambda\right)e_{\lambda}\left(\alpha_{2}\right)\\
\vdots\\
A_{n}\left(\lambda\right)e_{\lambda}\left(\alpha_{n}\right)
\end{array}\right).\label{eq:bd-7}
\end{equation}
is equivalent to 
\begin{align*}
A_{k}\left(\lambda\right)e_{\lambda}\left(\beta_{k+1}\right) & =A_{k}\left(\lambda\right)e_{\lambda}\left(\alpha_{k}\right),k=1,\ldots,n-1\\
cA_{0}\left(\lambda\right)e_{\lambda}\left(\beta_{1}\right) & =A_{n}\left(\lambda\right)e_{\lambda}\left(\alpha_{n}\right).
\end{align*}
Consequently, $L^{2}(\Omega)=L^{2}(J_{-}\cup J_{+})\oplus\bigoplus_{k=1}^{n-1}L^{2}(J_{k})$
and 
\[
P_{B}=P_{0}\oplus\bigoplus_{k=1}^{n-1}P_{k}
\]
where $P_{0}$ is a selfadjoint operator acting in $L^{2}(J_{-}\cup J_{+})$
determined by $cf(\beta_{1})=f(\alpha_{n})$ and has Lebesgue spectrum,
see Fig. \ref{fig:decouple}, by \cite{JPT11-1}, and $P_{k}$ acting
in $L^{2}(J_{k})$ is determined by $f(\beta_{k})=f(\alpha_{k})$
and has spectrum $\frac{1}{\beta_{k}-\alpha_{k}}\mathbb{Z}.$ Since
the set $\frac{1}{\ell}\mathbb{Z}$ has uniform density $\ell,$ the
density claim follows. 
\end{proof}
The same argument shows 
\begin{cor}
If $B=\left(\begin{array}{cc}
\boldsymbol{u} & B'\\
c & \boldsymbol{w}^{*}
\end{array}\right),$ where $c=e(\theta_{1}),$ $B'=\mathrm{diag}\left(e(\theta_{2}),\ldots,e(\theta_{n})\right)$
then the continuous spectrum of $P_{B}$ is the real line and the
discrete spectrum of $P_{B}$ is $\bigcup_{k=1}^{n-1}\left(\frac{\theta_{k+1}}{\ell_{k}}+\frac{1}{\ell_{k}}\mathbb{Z}\right),$
the multiplicity of each eigenvalue $\lambda$ is $\#\left\{ 2\leq k\leq n\mid\ell_{k}\lambda-\theta_{k}\in\mathbb{Z}\right\} ,$
and counting multiplicities the discrete spectrum has density $\sum_{k=1}^{n-1}\ell_{k}.$
 \end{cor}
\begin{rem}
Recall the cyclic permutation matrix:
\[
S=\left(\begin{array}{cccc}
0 & \cdots & 0 & 1\\
1 & 0 &  & 0\\
 & \ddots & \ddots & \vdots\\
\huge\mbox{0} &  & 1 & 0
\end{array}\right),\;\mbox{and }S^{-1}=S^{*}=\left(\begin{array}{cccc}
0 & 1 &  & \huge\mbox{0}\\
\vdots & 0 & \ddots\\
0 &  & \ddots & 1\\
1 & 0 & \cdots & 0
\end{array}\right)
\]
then 
\begin{equation}
BS=\underset{\mbox{matrix product}}{\underbrace{\left(\begin{array}{cc}
\boldsymbol{u} & B'\\
c & \boldsymbol{w}^{*}
\end{array}\right)S}}=\left(\begin{array}{cc}
B' & \boldsymbol{u}\\
\boldsymbol{w}^{*} & c
\end{array}\right)\label{eq:BS}
\end{equation}
For an application of this Remark, see Section \ref{sec:deg}, cases
1 and 2 in subsections \ref{sub:case1} and \ref{sub:case2}. \end{rem}
\begin{cor}
\label{cor:B}A $n\times n$ complex matrix $\left(\begin{array}{cc}
\boldsymbol{u} & B'\\
c & \boldsymbol{w}^{*}
\end{array}\right)$ is in $U\left(n\right)$ if and only if the following list of conditions
hold:
\begin{equation}
\begin{cases}
\underset{}{B'^{*}B'+\left\Vert \boldsymbol{w}\right\Vert ^{2}P_{\boldsymbol{w}}=I_{n-1}}\\
\underset{}{B'B'^{*}+\left\Vert \boldsymbol{u}\right\Vert ^{2}P_{\boldsymbol{u}}=I_{n-1}}\\
\underset{}{B'\boldsymbol{w}+\overline{c}\boldsymbol{u}=0}\\
\left\Vert \boldsymbol{w}\right\Vert ^{2}+\left|c\right|^{2}=\left\Vert \boldsymbol{u}\right\Vert ^{2}+\left|c\right|^{2}=1
\end{cases}\label{eq:unitary}
\end{equation}
where the following notation is used for vectors $\boldsymbol{x}\in\mathbb{C}^{n-1}$.
We denote the projection in $\mathbb{C}^{n-1}$ onto the one-dimensional
subspace $\mathbb{C}\boldsymbol{x}$ by $P_{\boldsymbol{x}}$.\end{cor}
\begin{proof}
Combine (\ref{eq:BS}) and (\ref{eq:unitary}). \end{proof}
\begin{cor}
\label{cor:normal}Let $B\in U(n)$ have the representation given
in Corollary \ref{cor:B} with entries, matrix corner $B'$, vectors
$\boldsymbol{u}$, $\boldsymbol{w}$, and scalar $c$. Then $B'$
is a normal matrix if and only if the vectors $\boldsymbol{u}$ and
$\boldsymbol{w}$ are proportional, with the constant of proportion
of modulus $1$. \end{cor}
\begin{proof}
Immediate from the system of equations (\ref{eq:unitary}). \end{proof}
\begin{rem}[A dichotomy]
 Corollary \ref{cor:normal} tells us precisely when $B$ has its
matrix corner $B'$ a normal matrix. 

Combining Corollary \ref{cor:normal} and Theorem \ref{thm:action},
then note that, by the spectral theorem for normal matrices, we may
pick $g\in U(n\lyxmathsym{\textendash}1)$ in order to diagonalize
the normal matrix $B'$; i.e., with $gB'g^{-1}=\mbox{diag}(z_{1},\lyxmathsym{\ldots}.,z_{n-1})$.

For this case, we then get the following dichotomy: 

\textbf{(i)} The set $\Lambda_{p}$ ($=$ the real poles) is non-empty
if and only the eigenvalue list $\{z_{j}\}$ contains an element of
modulus $1$. The corresponding selfadjoint operator $P_{B}$ in $L^{2}(\Omega)$
then has embedded point-spectrum. 

\textbf{(ii)} If every number in the list $\{z_{j}\}$ has modulus
strictly smaller than $1$, then all the poles are off the real line,
and as a result, $P_{B}$ has purely continuous spectrum.

Now the matrices in Corollary \ref{cor:normal} account for only a
sub-variety in $U(n)$, but \textquotedblleft{}large.\textquotedblright{}
Recall $B'$, being a corner of a unitary, is typically not unitary;
but, if it\textquoteright{}s normal, then the spectral theorem applies.
\end{rem}

\subsection{Permutation Matrices}
\begin{example}
\label{ex:nonnormal}Let 
\begin{equation}
B=\left(\begin{array}{cccc}
0 & 0 & 1 & 0\\
0 & 0 & 0 & 1\\
1 & 0 & 0 & 0\\
0 & 1 & 0 & 0
\end{array}\right)\in U(4).\label{eq:nonnormal}
\end{equation}
That is, $\boldsymbol{u}=e_{3}$, $\boldsymbol{w}=e_{1}$, $c=0$.
Note $B'$ is non-normal. We have 
\[
B'_{\boldsymbol{\alpha},\boldsymbol{\beta}}(\lambda)=\left(\begin{array}{ccc}
0 & e_{\lambda}(\beta_{3}-\alpha_{2}) & 0\\
0 & 0 & e_{\lambda}(\beta_{4}-\alpha_{2})\\
0 & 0 & 0
\end{array}\right),
\]
 
\[
D(\lambda)=\det\left[I_{3}-B'_{\boldsymbol{\alpha},\boldsymbol{\beta}}(\lambda)\right]\equiv1,\;\forall\lambda\in\mathbb{C},
\]
and 
\[
\left(I_{3}-B'_{\boldsymbol{\alpha},\boldsymbol{\beta}}(\lambda)\right)^{-1}=\left(\begin{array}{ccc}
1 & e_{\lambda}(\beta_{3}-\alpha_{2}) & e_{\lambda}(\beta_{3}+\beta_{4}-\alpha_{1}-\alpha_{2})\\
0 & 1 & e_{\lambda}(\beta_{4}-\alpha_{2})\\
0 & 0 & 1
\end{array}\right).
\]
Hence, by Theorem \ref{thm:soln} (1), we get
\begin{equation}
\begin{cases}
A_{1}(\lambda)=e_{\lambda}(\beta_{1}+\beta_{3}+\beta_{4}-\alpha_{1}-\alpha_{2}-\alpha_{3})\\
A_{2}(\lambda)=e_{\lambda}(\beta_{1}+\beta_{4}-\alpha_{2}-\alpha_{3})\\
A_{3}(\lambda)=e_{\lambda}(\beta_{1}-\alpha_{3})
\end{cases}\label{eq:soln-5}
\end{equation}
for all $\lambda\in\mathbb{C}$. See Figure \ref{fig:nonnormal}. 

\uline{Conclusions:}\textbf{\uline{ }}
\begin{enumerate}
\item The functions $\lambda\mapsto A_{i}(\lambda)$, $1\leq i\leq3$, have
no poles in $\mathbb{C}$.
\item The functions $\lambda\mapsto A_{i}(\lambda)$, $1\leq i\leq3$, are
complex exponentials, depending only on interval endpoints.
\item The spectrum of $P_{B}$ is purely continuous Lebesgue spectrum.
\end{enumerate}
\end{example}
\begin{figure}[H]
\includegraphics[scale=0.8]{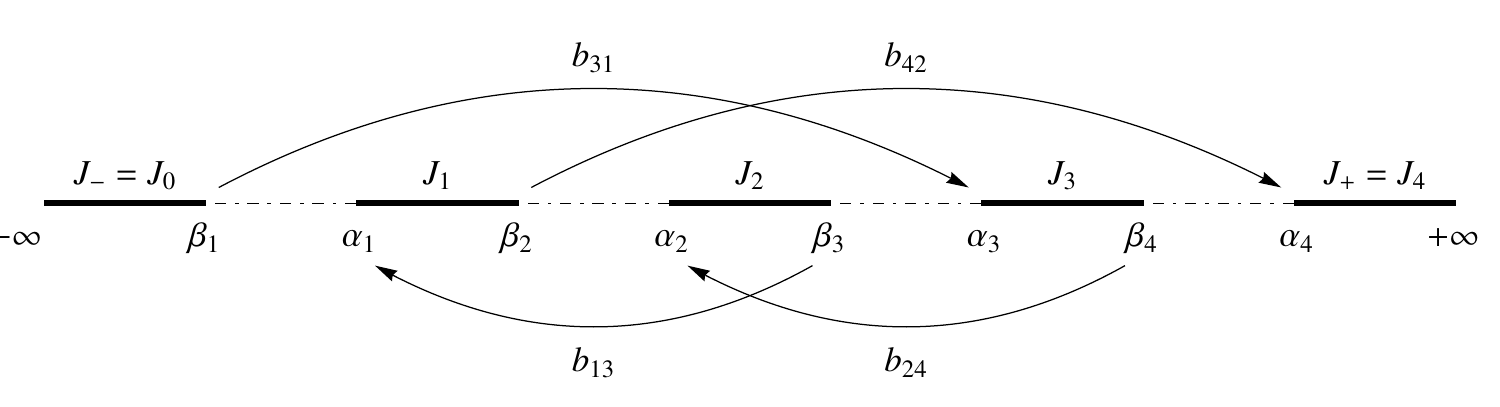}

\caption{\label{fig:nonnormal}$B$ is selfadjoint and $B'$ is non-normal.}
\end{figure}

Figure \ref{fig:nonnormal} illustrates the action of the unitary
one-parameter group $U_{B}(t)$ acting in $L^{2}(\Omega)$, and generated
by the selfadjoint operator $P_{B}$ coming from the boundary matrix
$B$ ($\in U(4)$) from Example \ref{ex:nonnormal}, see eq (\ref{eq:nonnormal})). 

The unitary one-parameter group $U_{B}(t)$ in Example \ref{ex:nonnormal}
acts by local translations to the right, acting on $L^{2}$ functions
in $\Omega$. Action \textquotedbl{}locally\textquotedbl{} by translation
here refers to translations to the right within the individual connected
components in $\Omega$. Moreover, Figure \ref{fig:nonnormal} illustrates
these local translations when interval-endpoints are encountered. 

For comparison, we sketch, in Fig \ref{fig:nonnormal-1} below, the
modification of the diagram (in Fig \ref{fig:nonnormal}) when the
boundary matrix $B$ from eq (\ref{eq:nonnormal}) is changed into
the $4\times4$ identity matrix $I_{4}$. 

Conclusion: If $B=I_{4}$, then the associated group $U_{B}(t)$ is
acting in $L^{2}(\Omega)$ by simply crossing over the gaps between
successive components in $\Omega$, moving from the left to the right,
jumping between neighboring boundary-points. 

Both the illustrations with the two versions of $B$ correspond to
a hit-and-run driver, constant speed, instantaneous jumps between
components in $\Omega$. The second one ($B=I_{4}$) rides right through
without changing direction, but the first one is drunk and jumps between
components in $\Omega$, in either direction, until eventually escaping
to $+\infty$.

Caution: With $B=I_{4}$, the associated corner $3\times3$ matrix
$B'$ is still non-normal. In fact, as for the boundary matrix $B$
in Example \ref{ex:nonnormal}, the $B'$ from $B=I_{4}$ is nilpotent. 

\begin{figure}[H]
\includegraphics[scale=0.8]{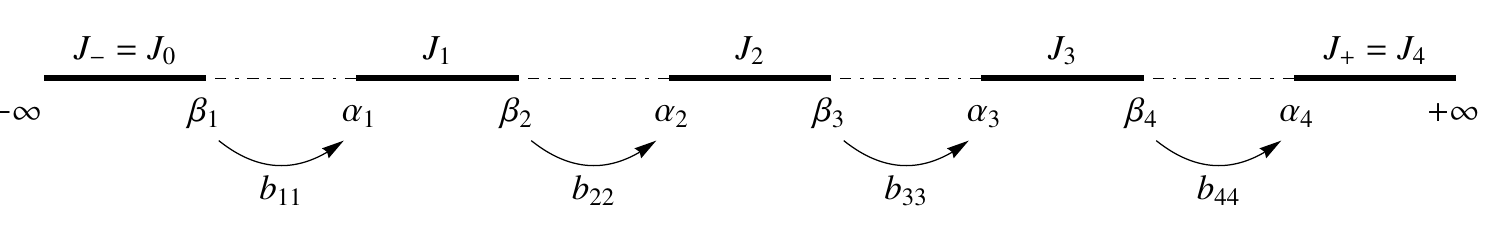}\caption{\label{fig:nonnormal-1}$B=I_{4}$, and $B'$ is non-normal.}
\end{figure}

\begin{cor}
Let $n>2$, fix $\Omega$ as above; and let $B\in U(n)$ be a permutation
matrix. Let $P_{B}$ be the associated selfadjoint operator. Then
the unitary one-parameter group $U_{B}(t)$ is acting in $L^{2}(\Omega)$
by local translations to the right of velocity $1$. Let $L$ be the
sum of the lengths of the $n-1$ bounded components in $\Omega$.
Then, as $t$ increases from $-\infty$ to $+\infty$, $U_{B}(t)$
acts in an interval of length $L$ by simply crossing over the gaps
between components in $\Omega$, jumping between boundary-points,
in either direction. In time interval of length $L$, $U_{B}(t)$
makes a permutation of the $n-1$ bounded components $J_{j}$ in $\Omega$,
where $1\leq j<n$. For every $n>2$, there is a permutation matrix
$B\in U(n)$ which makes the permutation of the intervals into the
identity permutation, i.e., the local translations riding right through,
jumps between successive components in $\Omega$, until eventually
escaping to $+\infty$. 
\end{cor}

\begin{cor}
\label{cor:B-1}Suppose $\left(\begin{array}{cc}
\boldsymbol{u} & B'\\
c & \boldsymbol{w}^{*}
\end{array}\right)\in U\left(n\right)$. Then 
\begin{equation}
B'\in U\left(n-1\right)\Longleftrightarrow\boldsymbol{u}=0\Longleftrightarrow\boldsymbol{w}=0.\label{eq:tmp-4}
\end{equation}
\end{cor}
\begin{proof}
Immediate from (\ref{eq:unitary}).\end{proof}
\begin{cor}
\label{cor:c}Let $n>2$, and let $B\in U(n)$. Consider the presentation
$B=\left(\begin{array}{cc}
\boldsymbol{u} & B'\\
c & \boldsymbol{w}^{*}
\end{array}\right)$ in (\ref{eq:B-1}). Then the following bi-implication holds:
\begin{equation}
c=0\Longleftrightarrow B'^{*}B'\mbox{ is a non-zero orthogonal projction in }\mathbb{C}^{n-1}.\label{eq:tmp-21}
\end{equation}
In particular, if (\ref{eq:tmp-21}) holds, then $\left\Vert B'\right\Vert =1$,
and so $B$ is degenerate.\end{cor}
\begin{proof}
($\Longrightarrow$) Assuming $c=0$, then from the equations in the
system (\ref{eq:unitary}), we get
\begin{equation}
\begin{cases}
\underset{}{\left\Vert \boldsymbol{u}\right\Vert =\left\Vert \boldsymbol{w}\right\Vert =1;}\\
\underset{}{B'^{*}B'=I_{n-1}-P_{\boldsymbol{w}};}\\
B'B'^{*}=I_{n-1}-P_{\boldsymbol{u}};
\end{cases}\label{eq:tmp-22}
\end{equation}
and so in particular, both $B'^{*}B'$ and $B'B'^{*}$ are orthogonal
projections in $\mathbb{C}^{n-1}$. It is known that orthogonal projections
have norm $1$, so 
\[
\left\Vert B'\right\Vert ^{2}=\left\Vert B'^{*}B'\right\Vert =\left\Vert I_{n-1}-P_{\boldsymbol{w}}\right\Vert =1
\]
as asserted. Indeed, the projection $P_{\boldsymbol{w}}^{\perp}=I_{n-1}-P_{\boldsymbol{w}}$
has rank $n-2\geq1$ by the assumption in the corollary.

The converse implication ($\Longleftarrow$) may be proved by the
same reasoning.\end{proof}
\begin{rem}
We will show in section \ref{sec:deg} that the unitary one-parameter
group $U_{B}(t)$ in $L^{2}(\Omega)$ has boundstates if and only
if the condition (\ref{eq:tmp-4}) in Corollary \ref{cor:c} holds;
see also Figure \ref{fig:decouple}.\end{rem}
\begin{example}
Let $\boldsymbol{U}=\left(\begin{array}{cc}
a & b\\
-\overline{b} & \overline{a}
\end{array}\right)\in SU(2)$, $\left|a\right|^{2}+\left|b\right|^{2}=1$, $z\in\mathbb{C}$, $\left|z\right|=1$,
and set $B=\left(\begin{array}{cccc}
a & b & \vline & 0\\
-\overline{b} & \overline{a} & \vline & 0\\
\hline 0 & 0 & \vline & z
\end{array}\right)\in U(3)$, decomposable. Then $B'=\left(\begin{array}{cc}
b & 0\\
\overline{a} & 0
\end{array}\right)$, and $B'^{*}B'=\left(\begin{array}{cc}
1 & 0\\
0 & 0
\end{array}\right)$ is the orthogonal projection onto the one-dimensional subspace in
$\mathbb{C}^{2}$ spanned by $\boldsymbol{e}_{1}$.\end{example}
\begin{lem}
Let $n>2$, and let $\boldsymbol{\alpha}=\left(\alpha_{i}\right)$
and $\boldsymbol{\beta}=\left(\beta_{i}\right)$ be a system of interval
endpoints in $\Omega_{\boldsymbol{\alpha},\boldsymbol{\beta}}=\cup_{i=0}^{n}J_{i}$,
where $J_{0}=J_{-}=(-\infty,\beta_{1})$, $J_{i}=(\alpha_{i},\beta_{i})$,
$1\leq i<n$, $J_{n}=J_{+}=(\alpha_{n},\infty)$ and with interval
length $L_{i}:=\beta_{i+1}-\alpha_{i}$, see Figure \ref{fig:lengths}.
Let $B\in U(n)$ have the form
\[
B=\left(\begin{array}{cc}
\boldsymbol{u} & B'\\
c & \boldsymbol{w}^{*}
\end{array}\right)
\]
then 
\begin{equation}
D(\lambda):=\det(I_{n-1}-B'_{\boldsymbol{\alpha},\boldsymbol{\beta}}),\label{eq:D-1}
\end{equation}
as a function on $\mathbb{R}$ ($\lambda\in\mathbb{R}$), only depends
on the interval lengths $L_{i}$, $i=1,2,\ldots,n-1$.
\end{lem}
\begin{figure}[H]
\includegraphics[scale=0.8]{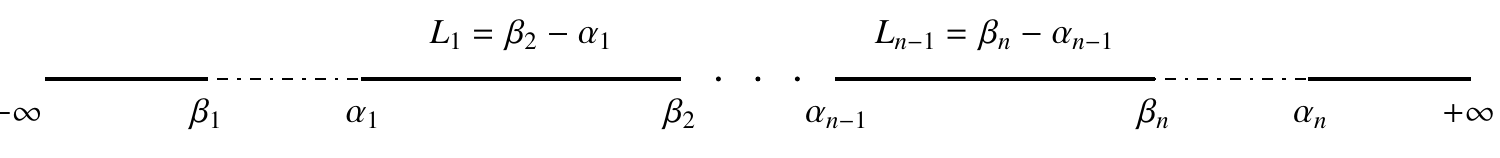}

\caption{\label{fig:lengths}Interval lengths.}

\end{figure}

\begin{proof}
Write $B'=(g_{ij})$, for $\mathbb{R}\ni\lambda\rightarrow D(\lambda)$
in (\ref{eq:D-1}), we then get the inside matrix as follows:
\begin{equation}
B'_{\boldsymbol{\alpha},\boldsymbol{\beta}}(\lambda)=\left(\begin{array}{cccc}
\underset{}{g_{12}\, e_{\lambda}(L_{1})} & g_{13}\, e_{\lambda}(\beta_{3}-\alpha_{1}) & \cdots & g_{1,n}\, e_{\lambda}(\beta_{n}-\alpha_{1})\\
g_{22}\, e_{\lambda}(\beta_{2}-\alpha_{2}) & g_{23}\, e_{\lambda}(L_{2}) &  & g_{2,n-1}\, e_{\lambda}(\beta_{n-1}-\alpha_{2})\\
\vdots &  & \ddots & \vdots\\
g_{n-1,2}\, e_{\lambda}(\beta_{2}-\alpha_{n-1}) & \cdots &  & g_{n-1,n}\, e_{\lambda}(L_{n-1})
\end{array}\right).\label{eq:B-7}
\end{equation}
As a result, for $D(\lambda)$, we get:
\begin{equation}
\det\left(\begin{array}{cccc}
\underset{}{1-g_{12}\, e_{\lambda}(L_{1})} & -g_{13}\, e_{\lambda}(\beta_{3}-\alpha_{1}) & \cdots & -g_{1,n}\, e_{\lambda}(\beta_{n}-\alpha_{1})\\
-g_{22}\, e_{\lambda}(\beta_{2}-\alpha_{2}) & 1-g_{23}\, e_{\lambda}(L_{2}) &  & -g_{2,n-1}\, e_{\lambda}(\beta_{n-1}-\alpha_{2})\\
\vdots &  & \ddots & \vdots\\
-g_{n-1,2}\, e_{\lambda}(\beta_{2}-\alpha_{n-1}) & \cdots &  & 1-g_{n-1,n}\, e_{\lambda}(L_{n-1})
\end{array}\right).\label{eq:det-2}
\end{equation}
The conclusion now follows by induction: The determinant may be computed
from its $(n-2)\times(n-2)$ sub-determinants, doing the computations
entry-by-entry in the first row of the inside matrix in (\ref{eq:det-2}).\end{proof}
\begin{rem}
A second proof may be obtained from the following determinant identity:
Let $T$ be a $k\times k$ matrix, and let $D_{i}$, $i=1,2$, be
two unitary $k\times k$ matrices; then the following formula holds:
\begin{equation}
\det\left(I-D_{1}^{*}TD_{2}\right)=\det\left(D_{1}^{*}D_{2}\right)\det\left(D_{1}D_{2}^{*}-T\right).\label{eq:det-3}
\end{equation}
The proof of the latter follows directly from a use of the multiplicative
property of the determinant. 

From Figure \ref{fig:eigen}, note that the numbers $G_{i}=\alpha_{i}-\beta_{i}$
are the lengths of the gaps between the successive intervals, i.e.,
between $J_{i-1}$ and $J_{i}$, $1\leq i\leq n$. Set $G_{tot}=\sum_{i}G_{i}=$
total gap-length. For the unitary diagonal matrices $D_{1}$ and $D_{2}$
in (\ref{eq:det-3}), take $D_{1}=D_{\boldsymbol{\alpha}}(\lambda)$
and $D_{2}=D_{\boldsymbol{\beta}}(\lambda)$, where
\begin{alignat*}{1}
D_{\boldsymbol{\alpha}}(\lambda) & =\underset{}{\left(\begin{array}{cccc}
e(\lambda\alpha_{1}) & 0 & \cdots & 0\\
0 & e(\lambda\alpha_{2}) & \ddots & \vdots\\
\vdots & \ddots & \ddots & 0\\
0 & \cdots & 0 & e(\lambda\alpha_{n-1})
\end{array}\right)},\\
D_{\boldsymbol{\beta}}(\lambda) & =\left(\begin{array}{cccc}
e(\lambda\beta_{2}) & 0 & \cdots & 0\\
0 & e(\lambda\beta_{3}) & \ddots & \vdots\\
\vdots & \ddots & \ddots & 0\\
0 & \cdots & 0 & e(\lambda\beta_{n})
\end{array}\right),
\end{alignat*}
and
\[
D_{\boldsymbol{\beta}-\boldsymbol{\alpha}}(\lambda)=\mbox{diag}\left(\left[e(\lambda L_{j})\right]_{j=1}^{n-1}\right).
\]
Then, for the determinant function $\mathbb{R}\ni\lambda\rightarrow D(\lambda)$
in (\ref{eq:D-1}), we get the following useful identity: 
\begin{equation}
D(\lambda)=e(-\lambda L_{tot})\:\det\left[\mbox{diag}\left(e(\lambda L_{j})\right)-B'\right]\label{eq:det-4}
\end{equation}
where $B'$ is the $(n-1)\times(n-1)$ matrix from (\ref{eq:B-1}).\end{rem}
\begin{cor}
The determinant $D$ in (\ref{eq:D-1}) is a function only of $\lambda\in\mathbb{R}$,
$(L_{i})\in\mathbb{R}_{+}^{n-1}$, i.e., $D(\lambda)=D(\lambda,L_{1},\ldots,L_{n-1},B')$.
\end{cor}

\begin{cor}
Let 
\begin{equation}
\mathbb{R}\ni\lambda\mapsto D(\lambda)=\det\left[\mbox{diag}\left(e(\lambda L_{j})\right)_{1}^{n-1}-B'\right]\label{eq:det-5}
\end{equation}
be the determinant factor on the RHS in (\ref{eq:det-4}). For $j=1,2,\ldots,n-1$,
let $D_{j}(\lambda)$ be the determinant of the $(n-2)\times(n-2)$
sub-matrix obtained from 
\begin{equation}
\left[\mbox{diag}\left(e(\lambda L_{j})\right)_{1}^{n-1}-B'\right]\label{eq:det-6}
\end{equation}
by omission of its $j^{th}$ row, and its $j^{th}$ column. Then
\begin{equation}
\frac{1}{2\pi i}\frac{d}{d\lambda}D(\lambda)=\sum_{j=1}^{n-1}L_{j}\, e(\lambda L_{j})D_{j}(\lambda).\label{eq:det-7}
\end{equation}
\end{cor}
\begin{proof}
Differentiate (\ref{eq:det-5}), viewing the determinant as a multi-linear
function on the $n-1$ columns in (\ref{eq:det-6}). Applying $\frac{1}{2\pi i}\frac{d}{d\lambda}$
to the $j^{th}$ column in (\ref{eq:det-6}) yields the desired formula
(\ref{eq:det-7}). To see this, note that the matrix in (\ref{eq:det-6})
is 
\[
\left(\begin{array}{cccc}
e(\lambda L_{1})-b_{12} & -b_{13} & \cdots & -b_{1n}\\
-b_{22} & e(\lambda L_{2})-b_{23} & \cdots & -b_{2n}\\
\vdots & \vdots & \ddots & \vdots\\
-b_{n-1,1} & -b_{n-1,2} & \cdots & e(\lambda L_{n-1})-b_{n-1,n}
\end{array}\right).
\]
\end{proof}
\begin{cor}
Let $\boldsymbol{\alpha}=\left(\alpha_{i}\right)$ and $\boldsymbol{\beta}=\left(\beta_{i}\right)$
be as above, i.e., the specified interval endpoints. Let $B\in U\left(n\right)$
be written as $\left(\begin{array}{cc}
\boldsymbol{u} & B'\\
c & \boldsymbol{w}^{*}
\end{array}\right)$. Then the following two conditions are equivalent:\end{cor}
\begin{enumerate}
\item \label{enu:1}The $\left(\alpha,\beta\right)$-$B$ problem is non-degenerate
for all $\alpha$ and $\beta$; and
\item \label{enu:2}$\left\Vert B'\right\Vert <1$ where $\left\Vert \cdot\right\Vert $
is the $\mathbb{C}^{n-1}$-operator norm.\end{enumerate}
\begin{proof}
Recall that (\ref{enu:1}) is the assertion that $\mathbb{R}\ni\lambda\mapsto B'_{\alpha,\beta}\left(\lambda\right)$
satisfies 
\begin{equation}
\det\left(I_{n-1}-B'_{\alpha,\beta}\left(\lambda\right)\right)\neq0\label{eq:det}
\end{equation}
for all $\lambda\in\mathbb{R}$. In other words, (\ref{enu:1}) states
that for all $\lambda\in\mathbb{R}$, $1$ is not in the spectrum
of $B'_{\alpha,\beta}\left(\lambda\right)$. 

Using (\ref{enu:2}) we see that this is equivalent to $\left\Vert B_{\alpha,\beta}'\left(\lambda\right)\right\Vert <1$.\end{proof}
\begin{cor}
Let $B=\left(\begin{array}{cc}
\boldsymbol{u} & B'\\
c & \boldsymbol{w}^{*}
\end{array}\right)\in U\left(n\right)$. Then $\boldsymbol{w}$ is an eigenvector for $B'^{*}B'$ with eigenvalue
$\left|c\right|^{2}$; and $\boldsymbol{u}$ is an eigenvector for
$B'B'^{*}$ with eigenvalue $\left|c\right|^{2}$. \end{cor}
\begin{proof}
From Corollary \ref{cor:B} we know that 
\[
B'\boldsymbol{w}+\overline{c}\boldsymbol{u}=B'^{*}\boldsymbol{u}+c\boldsymbol{w}=0.
\]
Now, apply $B'^{*}$ to the first, and $B'$ to the second, the desired
conclusion follows, i.e., we get the two eigenvalue equations:
\begin{alignat}{1}
B'^{*}B'\boldsymbol{w} & =\left|c\right|^{2}\boldsymbol{w};\;\mbox{and }\label{eq:tmp-14}\\
B'B'^{*}\boldsymbol{u} & =\left|c\right|^{2}\boldsymbol{u}.\label{eq:tmp-13}
\end{alignat}
 \end{proof}
\begin{cor}
Let $B=\left(\begin{array}{cc}
\boldsymbol{u} & B'\\
c & \boldsymbol{w}^{*}
\end{array}\right)\in U\left(n\right)$; then 
\begin{equation}
\left\Vert B'\right\Vert \geq\left|c\right|.\label{eq:Bc}
\end{equation}
\end{cor}
\begin{proof}
The proof divides into two cases. First if $\boldsymbol{w}=0$, then
$B'\in U\left(n-1\right)$ by Corollary \ref{cor:B-1}, and so $\left\Vert B'\right\Vert =1$,
and (\ref{eq:Bc}) holds. Conversely, suppose $\boldsymbol{w}\neq0$;
then by (\ref{eq:tmp-14}), $\left|c\right|\in sp\left(B'^{*}B'\right)$,
but from operator theory, we know that
\[
\left\Vert B'\right\Vert ^{2}=\max\left\{ s\in\mathbb{R};\: s\in sp\left(B'^{*}B'\right)\right\} =\left\Vert B'^{*}B'\right\Vert ^{2}=\left\Vert B'B'^{*}\right\Vert ^{2};
\]
hence $\left|c\right|^{2}\leq\left\Vert B'\right\Vert ^{2}$ which
is the desired conclusion (\ref{eq:Bc}). The inequality (\ref{eq:Bc})
may be sharp.\end{proof}
\begin{lem}
Let $n>2$, and let $A$ be an $(n-1)\times(n-1)$ matrix, $\boldsymbol{u},\boldsymbol{w}\in\mathbb{C}^{n-1}$,
$c\in\mathbb{C}$. Suppose $B:=\left(\begin{array}{ccc}
A & \vline & \boldsymbol{u}\\
\hline \boldsymbol{w} & \vline & c
\end{array}\right)\in U(n)$. Then $B=B^{*}$ if and only if
\begin{enumerate}
\item $A=A^{*}$,
\item $\boldsymbol{u}=\boldsymbol{w}$,
\item $c\in\mathbb{R}$,
\item $\left\Vert \boldsymbol{u}\right\Vert ^{2}+\left|c\right|^{2}=1$,
\item $A^{2}+\left\Vert \boldsymbol{u}\right\Vert ^{2}P_{\boldsymbol{u}}=I_{n-1}$,
and
\item $\boldsymbol{u}\in\mathscr{N}(A+cI_{n-1})$. 
\end{enumerate}
\end{lem}
\begin{proof}
A direct computation.\end{proof}
\begin{example}
Consider the following selfadjoint unitary $3\times3$ matrix $B$
and its cyclic permutation $\tilde{B}$ where
\[
\tilde{B}=\left(\begin{array}{ccc}
a & b & {\displaystyle \frac{g}{\sqrt{1+\left(\frac{a}{b}\right)^{2}}}}\\
b & {\displaystyle \frac{b^{2}}{a+c}-c} & {\displaystyle \frac{-g}{\sqrt{1+\left(\frac{b}{a}\right)^{2}}}}\\
{\displaystyle \frac{g}{\sqrt{1+\left(\frac{a}{b}\right)^{2}}}} & {\displaystyle \frac{-g}{\sqrt{1+\left(\frac{b}{a}\right)^{2}}}} & c
\end{array}\right)
\]
where $a,b,c>0$, and $g=\sqrt{1-c^{2}}$. One verifies that if $c$
is close to $0$, then $g$ is close to $1$; hence $b$ must be close
to $0$, and $a$ closed to 1. But the norm of the corner matrix
\[
B'=\left(\begin{array}{cc}
a\quad & \underset{}{b}\\
{\displaystyle b\quad} & {\displaystyle {\displaystyle \frac{b^{2}}{a+c}-c}}
\end{array}\right)
\]
is $a$ ($=$ its numerical range). Thus, the inequality (\ref{eq:Bc})
may be strict. \end{example}
\begin{rem}
Let $B=\left(\begin{array}{cc}
\boldsymbol{u} & B'\\
c & \boldsymbol{w}^{*}
\end{array}\right)$ and $ $$\widetilde{B}=\left(\begin{array}{cc}
B' & \mathbf{u}\\
\mathbf{w}^{*} & c
\end{array}\right).$ If $\zeta$ is an eigenvector for $B'$ with eigenvalue $e(\theta)$
for some real $\theta,$ then 
\[
\left(\begin{array}{cc}
B' & \mathbf{u}\\
\mathbf{w}^{*} & c
\end{array}\right)\left(\begin{array}{c}
\zeta\\
0
\end{array}\right)=\left(\begin{array}{c}
e(\theta)\zeta\\
\mathbf{w}^{*}\zeta
\end{array}\right).
\]
But $\widetilde{B}$ is unitary, in particular $\left(\begin{array}{c}
\zeta\\
0
\end{array}\right)$ and $\left(\begin{array}{c}
e(\theta)\zeta\\
\mathbf{w}^{*}\zeta
\end{array}\right)$ have the same norm, hence
\begin{equation}
\mathbf{w}^{*}\zeta=0.\label{eq:eigenvectors-B'-orthgonal-to-u}
\end{equation}
Slightly generalizing a claim in Corollary \ref{cor:zeta}. If $\widetilde{B}$
is selfadjoint, then $\mathbf{w}=\mathbf{u},$ hence implies (\ref{eq:unitary})
$B'\mathbf{u}=-c\mathbf{u}$. In particular, $\mathbf{u}$ is in the
range of $I_{n-1}-B',$ if $c\neq-1.$ On the other hand, if $c=-1,$
then (\ref{eq:eigenvectors-B'-orthgonal-to-u}) implies $\mathbf{u}=0.$ 

Consequently, if $\widetilde{B}$ is selfadjoint, then we are never
in case (2)(a) of Theorem \ref{thm:LAP}.
\end{rem}

\section{The Continuous Spectrum Is Simple}

The generalized eigenfunctions studied in the previous section (see
Theorem \ref{thm:UB} and eq. (\ref{eq:eigen})) yield separation
of variables, a harmonic part (in the spatial variable $x$ as $e_{\lambda}(x)$),
and a finite family of scattering coefficients $\{A_{j}^{(B)}(\lambda)\}$,
functions of the spectral variable $\lambda$. In this section we
study the meromorphic extension of scattering coefficients; extension
to non-real values of $\lambda$. 

We show (Theorem \ref{thm:sp-1}) that, if the first of the scattering
coefficients is normalized to $1$, then the continuous part of the
spectrum for each of the operators is purely Lebesgue, with spectral
measure having Radon-Nikodym derivative equal to the constant $1$.
We further show that each point on the real line $\mathbb{R}$ occurs
in the continuous spectrum with multiplicity $1$.

Let the open set $\Omega$ be as before, i.e., $\Omega$ is the complement
of $n$ bounded closed and disjoint intervals. The minimal momentum
operator will then have deficiency indices $(n,n)$, and as a result,
the boundary conditions are indexed by the matrix group $U(n)$. As
before, we denote the unbounded selfadjoint extension operators $P_{B}$
indexed by a fixed element $B\in U(n)$, and the corresponding unitary
one-parameter group is $U_{B}(t)$. These operators are acting in
the Hilbert space $L^{2}(\Omega)$. 

We restrict the element $B$ as in Theorem \ref{thm:soln}, i.e.,
it is assumed non-degenerate. In this generality we are able to establish
(Thm \ref{thm:sp-1}) the complete and detailed spectral resolution
for $P_{B}$, and therefore for the one-parameter group $U_{B}(t)$
as it acts on the Hilbert space $L^{2}(\Omega)$. We show that, if
the first coefficient in the formula for the generalized eigenfunction
system in (\ref{eq:GEF-1}), is chosen to be $1$, then the measure
$\sigma_{B}$ in the spectral resolution for $U_{B}(t)$ becomes Lebesgue
measure. Moreover, we show that the multiplicity is uniformly one.
In the theorem, we further compute all the details, closed formulas,
for the spectral theory.
\begin{thm}
\label{thm:sp-1}Let $\boldsymbol{\alpha}=\left(\alpha_{i}\right)$
and $\boldsymbol{\beta}=\left(\beta_{i}\right)$ be a system of interval
endpoints:
\begin{equation}
-\infty<\beta_{1}<\alpha_{1}<\beta_{2}<\cdots<\beta_{n}<\alpha_{n}<\infty,\label{eq:tmp-15}
\end{equation}
with $J_{0}=J_{-}=\left(-\infty,\beta_{1}\right)$, $J_{n}=J_{+}=\left(\alpha_{n},\infty\right)$,
and $J_{i}=\left(\alpha_{i},\beta_{i+1}\right)$, $i=1,\ldots,n-1$.
Let $B\in U\left(n\right)$ be chosen non-degenerate (fixed), and
let 
\begin{equation}
\psi_{\lambda}\left(x\right):=\psi_{\lambda}^{\left(B\right)}\left(x\right)=\left(\sum_{i=0}^{n}\chi_{i}\left(x\right)A_{i}^{\left(B\right)}\left(\lambda\right)\right)e_{\lambda}\left(x\right)\label{eq:tmp-16}
\end{equation}
be as in Theorem \ref{thm:soln}, where $\Omega=\bigcup_{i=0}^{n}J_{i}$,
$\chi_{i}:=\chi_{J_{i}}$, $0\leq i\leq n$, and where the functions
$\left(A_{i}^{\left(B\right)}\left(\cdot\right)\right)_{i=0}^{n}$
are chosen as in (\ref{eq:soln-3}) with $A_{0}^{\left(B\right)}\equiv1$. 

For $f\in L^{2}\left(\Omega\right)$, setting
\begin{equation}
\left(V_{B}f\right)\left(\lambda\right)=\left\langle \psi_{\lambda},f\right\rangle _{\Omega}=\int\overline{\psi_{\lambda}\left(y\right)}f\left(y\right)dy,\label{eq:VB}
\end{equation}
we then get the following orthogonal expansions:
\begin{equation}
f=\int_{\mathbb{R}}\left(V_{B}f\right)\left(\lambda\right)\psi_{\lambda}\left(\cdot\right)d\lambda\label{eq:exp}
\end{equation}
where the convergence in (\ref{eq:exp}) is to be taken in the $L^{2}$-sense
via
\begin{equation}
\left\Vert f\right\Vert _{L^{2}\left(\Omega\right)}^{2}=\int_{\mathbb{R}}\left|\left(V_{B}f\right)\left(\lambda\right)\right|^{2}d\lambda,\; f\in L^{2}\left(\Omega\right).\label{eq:norm}
\end{equation}

Moreover, we have
\begin{equation}
V_{B}U_{B}\left(t\right)=M_{t}V_{B},\: t\in\mathbb{R}\label{eq:int}
\end{equation}
where
\[
\left(M_{t}g\right)\left(\lambda\right)=e_{\lambda}\left(-t\right)g\left(\lambda\right)
\]
for all $t,\lambda\in\mathbb{R}$, and all $g\in L^{2}\left(\mathbb{R}\right).$
\end{thm}
\begin{figure}[H]
\[
\xymatrix{L^{2}\left(\Omega\right)\ar[r]^{U_{B}\left(t\right)}\ar[d]_{V_{B}} & L^{2}\left(\Omega\right)\ar[d]^{V_{B}}\\
L^{2}\left(\mathbb{R}\right)\ar[r]_{M_{t}} & L^{2}\left(\mathbb{R}\right)
}
\]

\caption{Intertwining}

\end{figure}

\begin{rem}
The reason for the word \textquotedblleft{}generalized\textquotedblright{}
referring to the family (\ref{eq:tmp-16}) of generalized eigenfunctions
is that, for a fixed value of the spectral parameter $\lambda$, the
function $\psi_{\lambda}$ is not in $L^{2}(\Omega)$, so strictly
speaking it is not an eigenfunction for the unbounded selfadjoint
operator $P_{B}$ in $L^{2}(\Omega)$. But there is a fairly standard
way around the difficulty, involving distributions, see e.g., \cite{JPT11-2,Mau68,Mik04}.\end{rem}
\begin{example}
\label{ex:mero}Set $n=2$, $B=\left(\begin{array}{cc}
a & b\\
-\overline{b} & \overline{a}
\end{array}\right)$, $a,b\in\mathbb{C}$, $\left|a\right|^{2}+\left|b\right|^{2}=1$.
With the normalization $A_{0}^{(B)}\equiv1$, we get the following
representation of the two function $\mathbb{R}\ni\lambda\mapsto A_{i}^{(B)}(\lambda)$,
$i=1,2$: Fix $-\infty<\beta_{1}<\alpha_{1}<\beta_{2}<\alpha_{2}<\infty$;
set $L:=\beta_{2}-\alpha_{1}$, and $G:=\alpha_{2}-\beta_{1}$; then
\begin{equation}
\begin{cases}
\underset{}{A_{1}^{(B)}(\lambda)}={\displaystyle \frac{a\, e_{\lambda}(\beta_{1}-\alpha_{1})}{1-b\, e_{\lambda}(L)}},\mbox{ and}\\
A_{2}^{(B)}(\lambda)={\displaystyle \frac{e_{\lambda}(L-G)-\overline{b}\,\overline{e_{\lambda}(G)}}{1-b\, e_{\lambda}(L)}}.
\end{cases}\label{eq:coeff}
\end{equation}
Note the poles in the presentation of the two functions in (\ref{eq:coeff-1}).
In the meromorphic extensions of the two functions, we have, for $z\in\mathbb{C}$,
\begin{equation}
\begin{cases}
\underset{}{A_{1}^{(B)}(z)={\displaystyle \frac{a\, e(z(\beta_{1}-\alpha_{1}))}{1-b\, e(zL)}}}\\
A_{2}^{(B)}(z)={\displaystyle \frac{e(z(L-G))-\overline{b}\, e(-zG)}{1-b\, e(zL)}.}
\end{cases}\label{eq:coeff-1}
\end{equation}
\end{example}
\begin{rem}
It follows from Corollaries \ref{cor:pt} and \ref{cor:pole} that
also in the general case with $\Omega$ open and associated $\boldsymbol{\alpha}=(\alpha_{i})_{i=1}^{n}$,
and $\boldsymbol{\beta}=(\beta_{i})_{i=1}^{n}$, the scattering coefficients
$A_{j}^{(B)}(\cdot)$ have meromorphic extensions. With this information,
one may derive a suitable de Branges Hilbert space of meromorphic
functions \cite{deB84,ADV09}, for a detailed and geometric analysis
of the general case. It follows that formulas (\ref{eq:coeff})-(\ref{eq:coeff-1})
in Example \ref{ex:mero} are indicative for the study of the general
case; only in the general case $n>2$, the extension of the meromorphic
functions $\mathbb{C}\ni z\mapsto A_{j}^{(B)}(z)$ is substantially
more difficult.\end{rem}
\begin{proof}
\textbf{Outline of proof in sketch.} Given $B$ in $U(n)$, we get
a specific selfadjoint operator $P_{B}$, as outlined in section \ref{sec:Intro}.
And there is therefore an associated strongly continuous unitary one-parameter
group $U_{B}(t)$ generated by $P_{B}$ and acting on $L^{2}(\Omega)$.
We begin with an application of the abstract spectral theorem: Given
the selfadjointness of the operator $P_{B}$, we may apply the spectral
theorem to it, but this yields only the abstract form of the spectral
resolution; not revealing very much specific information: At the outset,
the general theory does not say what the spectral data are, such as
detailed information about the measure $\sigma_{B}$ arising in the
direct integral representation for $P_{B}$. Given the properties
of $P_{B}$ it does say that $\sigma_{B}$ must be absolutely continuous
with respect to Lebesgue measure on $\mathbb{R}$. But it does not
say what the Radon-Nikodym derivative $F_{B}$ is. Our assertion is
that with the normalization $A_{0}=1$, we obtain $F_{B}=1$.

Having $\sigma_{B}$, we proceed to apply the theory of Lax-Phillips
\cite{LP68} to the unitary one-parameter group $U_{B}(t)$ as it
is acting on $L^{2}(\Omega)$. To do this, we must assume that $B$
is non-degenerate, so that $P_{B}$ will have no point-spectrum. To
apply Lax-Phillips, we do not need to know details about the measure
$\sigma_{B}$. Its abstract properties are enough. To begin with,
we first establish that $L^{2}(J_{0})$ serves as an incoming subspace
$\mathscr{D}_{-}$ in $L^{2}(\Omega)$ for the action of unitary one-parameter
group $U_{B}(t)$. Recall $J_{0}=(-\infty,\beta_{1})$. To show that
this incoming subspace $\mathscr{D}_{-}$ does satisfy the Lax-Phillips
axioms, again we do not need detailed information about the measure
$\sigma_{B}$. Finally, with an application of Lax-Phillips, and a
number of other steps, in the end, assuming $A_{0}=1$, we are able
to conclude that the measure $\sigma_{B}$ is Lebesgue measure on
$\mathbb{R}$.

\textbf{Proof in detail.} First note that the assumption on $B$ rules
out point spectrum; see Corollary \ref{cor:pt}. Using \cite{JPT11-2}
and the theory of generalized eigenfunctions \cite{DS88b,Mau68,Mik04,MM63},
we note that there is a Borel measure $\sigma_{B}\left(d\lambda\right)$
on $\mathbb{R}$, absolutely continuous with respect to Lebesgue measure
such that the formulas (\ref{eq:exp}) and (\ref{eq:norm}) hold with
$\sigma_{B}\left(d\lambda\right)$ on the RHS of the equations. Our
assertion is that $\sigma_{B}\left(d\lambda\right)=d\lambda=$ Lebesgue
measure on $\mathbb{R}$, i.e., then the Radon-Nikodym derivative
\begin{equation}
\frac{d\sigma_{B}\left(d\lambda\right)}{d\lambda}=F_{B}\left(\lambda\right)\equiv1.\label{eq:RN}
\end{equation}
The validity of (\ref{eq:RN}) uses the assumption $A_{0}^{\left(B\right)}\equiv1$
in an essential way. Hence in (\ref{eq:tmp-16}), we have 
\begin{equation}
\psi_{\lambda}^{(B)}\left(x\right)=\left(\chi_{\left(-\infty,\beta_{1}\right)}\left(x\right)+\sum_{j=1}^{n}A_{j}^{(B)}\left(\lambda\right)\chi_{j}\left(x\right)\right)e_{\lambda}\left(x\right).\label{eq:tmp-17}
\end{equation}
We will now suppress the $B$-dependence in $\psi_{\lambda}^{\left(B\right)}\left(\cdot\right)$
and $A_{j}^{\left(B\right)}\left(\cdot\right)$. It is understood
that $\psi_{\lambda}\left(\cdot\right)$ is a function on $\Omega$,
and each $A_{j}\left(\cdot\right)$ is a function on $\mathbb{R}$;
see Theorem \ref{thm:soln} for the explicit formulas. 

In the computation below, we will be using the normalized Fourier
transform $\hat{\cdot}$, and its inverse $\check{\cdot}$. 

Let $P_{j}=$ multiplication by $\chi_{j}$ for $0\leq j\leq n$,
viewed as projection operators in $L^{2}\left(\Omega\right)$. We
then have
\begin{equation}
\sum_{j=0}^{n}P_{j}=I=I_{L^{2}\left(\Omega\right)},\mbox{ and }P_{j}P_{k}=\delta_{j,k}P_{j}.\label{eq:proj}
\end{equation}

From (\ref{eq:tmp-17}), we then get the following expression for
$V_{B}:L^{2}\left(\Omega\right)\rightarrow L^{2}\left(\sigma_{B}\right)$:
\begin{equation}
\left(V_{B}f\right)\left(\lambda\right)=\left(P_{0}f\right)^{\wedge}\left(\lambda\right)+\sum_{j=1}^{n}\overline{A_{j}\left(\lambda\right)}\left(P_{j}f\right)^{\wedge}\left(\lambda\right),\label{eq:VB-1}
\end{equation}
for all $f\in L^{2}\left(\Omega\right)$, and all $\lambda\in\mathbb{R}$;
and
\begin{equation}
f=P_{0}f+\sum_{j=1}^{n}\chi_{j}\left(\cdot\right)\left(\overline{A_{j}\left(\cdot\right)}\left(P_{j}f\right)^{\wedge}\right)^{\vee}.\label{eq:proj-1}
\end{equation}
(It is understood in (\ref{eq:VB-1}), (\ref{eq:proj-1}) and the
sequel that $A_{j}=A_{j}^{(B)}$ depends on a choice of $B\in U(n)$.)

With $B\in U\left(n\right)$ specified as in the theorem, we get a
unique selfadjoint operator $P_{B}$ in $L^{2}\left(\Omega\right)$
as a selfadjoint extension of the minimal operator $\frac{1}{i2\pi}\frac{d}{dx}$,
i.e., the minimal operator specified by the condition: $f\in L^{2}\left(\Omega\right)$,
$f'\in L^{2}\left(\Omega\right)$ and $f=0$ on $\partial\Omega$,
see \cite{JPT11-2}.

Let, for $t\in\mathbb{R}$, 
\begin{equation}
U_{B}\left(t\right):L^{2}\left(\Omega\right)\rightarrow L^{2}\left(\Omega\right)\label{eq:UB}
\end{equation}
be the corresponding strongly continuous unitary one-parameter group
generated by $P_{B}$; see \cite{Sto90,vNeu32,vN49}. 

Applying (\ref{eq:proj-1}), we get the following formula for $U_{B}\left(t\right)f$,
$f\in L^{2}\left(\Omega\right)$, understood in the sense of $L^{2}$-convergence:
\begin{equation}
\left(U_{B}\left(t\right)f\right)\left(x\right)=\chi_{0}\left(x\right)\left(P_{0}f\right)\left(x-t\right)+\sum_{j=1}^{n}\chi_{j}\left(x\right)\left(A_{j}\left(\cdot\right)\left(P_{j}f\right)^{\wedge}\right)^{\vee}\left(x-t\right)\label{eq:UB-1}
\end{equation}
for all $f\in L^{2}\left(\Omega\right)$, and all $x\in\Omega$, $t\in\mathbb{R}$.

We now prove (\ref{eq:int}) as an operator-identity, i.e., the assertion
that that 
\begin{equation}
V_{B}:L^{2}\left(\Omega\right)\rightarrow L^{2}\left(\sigma_{B}\right)\label{eq:tmp-18}
\end{equation}
intertwines the two unitary one-parameter groups specified in (\ref{eq:int}). 

Let $f\in L^{2}\left(\Omega\right)$, $x\in\Omega$, and $\lambda,t\in\mathbb{R}$.
Then,
\begin{eqnarray*}
 &  & \left(V_{B}U_{B}\left(t\right)f\right)\left(\lambda\right)\\
 & = & \left\langle \psi_{\lambda},U_{B}\left(t\right)f\right\rangle _{\Omega}\,\,\,\,\,\,\,\,(\mbox{by (\ref{eq:VB})})\\
 & = & \left\langle U_{B}\left(-t\right)\psi_{\lambda},f\right\rangle _{\Omega}\,\,\,\,(\mbox{by the theory of generalized eigenfunctions})\\
 & = & \left\langle e_{\lambda}\left(t\right)\psi_{\lambda},f\right\rangle _{\Omega}\,\,\,\,\,\,\,\,\,\,\,(\mbox{by }(\ref{eq:tmp-17}))\\
 & = & e_{\lambda}\left(-t\right)\left\langle \psi_{\lambda},f\right\rangle _{\Omega}\,\,\,\,\,\,\,(\mbox{since }\left\langle \cdot,\cdot\right\rangle \mbox{ is conjugate linear in first variable})\\
 & = & e_{\lambda}\left(-t\right)\left(V_{B}f\right)\left(\lambda\right)\,\,\,\,(\mbox{by (\ref{eq:VB})})\\
 & = & \left(M_{t}V_{B}f\right)\left(\lambda\right).
\end{eqnarray*}
Since this holds for all $\lambda\in\mathbb{R}$, the desired formula
(\ref{eq:int}) is verified.

We now establish formulas (\ref{eq:exp}) and (\ref{eq:norm}) first
for $f\in L^{2}\left(J_{0}\right)=L^{2}\left(-\infty,\beta_{1}\right)$;
and we recall from \cite{JPT11-2} that this subspace serves as an
\emph{incoming} subspace $\mathscr{D}_{-}$ for $U_{B}\left(t\right)$
in the sense of Lax-Phillips \cite{LP68}; see also \cite{JPT11-2},
i.e.,
\begin{equation}
\mathscr{D}_{-}=\mathscr{H}_{0}=L^{2}\left(J_{0}\right)=L^{2}\left(-\infty,\beta_{1}\right).\label{eq:Din}
\end{equation}

\begin{flushleft}
\textbf{Proof of (\ref{eq:exp}):} For $f_{0}\in\mathscr{D}_{-}$,
and $x\in\Omega$; then (in the sense of $L^{2}$-convergence): 
\begin{alignat*}{1}
f_{0}\left(x\right) & =\chi_{J_{0}}\left(x\right)f_{0}\left(x\right)\\
 & =\chi_{J_{0}}\left(x\right)\int_{\mathbb{R}}e_{\lambda}\left(x\right)\left(P_{0}f_{0}\right)^{\wedge}\left(\lambda\right)d\lambda\:\:\:\:\:\:\:\:(\mbox{by (\ref{eq:VB-1})})\\
 & =\chi_{J_{0}}\left(x\right)\int_{\mathbb{R}}\psi_{\lambda}\left(x\right)\left(P_{0}f_{0}\right)^{\wedge}\left(\lambda\right)d\lambda\,\,\,\,\,\,\,\,\,(\mbox{by }(\ref{eq:tmp-17}))\\
 & =\int_{\mathbb{R}}\left(V_{B}f_{0}\right)\left(\lambda\right)\psi_{\lambda}\left(x\right)d\lambda\,\,\,\,\,\,\,(\mbox{by }(\ref{eq:VB})\mbox{ and }(\ref{eq:tmp-17}))
\end{alignat*}
which is the desired formula (\ref{eq:exp}).
\par\end{flushleft}

\begin{flushleft}
\textbf{Proof of (\ref{eq:norm}):} By the spectral theorem (see \cite{JPT11-2}),
the measure $\sigma_{B}\left(\lambda\right)$ satisfies
\begin{equation}
\left\Vert f\right\Vert _{L^{2}\left(\Omega\right)}^{2}=\int_{\mathbb{R}}\left|\left(V_{B}f\right)\left(\lambda\right)\right|^{2}\sigma_{B}\left(d\lambda\right),\label{eq:tmp-19}
\end{equation}
see also (\ref{eq:tmp-18}). Now specialize to $f=f_{0}\in\mathscr{D}_{-}\subset L^{2}\left(\Omega\right)$.
Using (\ref{eq:exp}), and Parseval's formula, we get
\begin{equation}
\left\Vert f_{0}\right\Vert _{L^{2}\left(\Omega\right)}^{2}=\int_{\mathbb{R}}\left|\hat{f}_{0}\left(\lambda\right)\right|^{2}\sigma_{B}\left(d\lambda\right)=\int_{\mathbb{R}}\left|\hat{f}_{0}\left(\lambda\right)\right|^{2}d\lambda\label{eq:tmp-20}
\end{equation}
which is (\ref{eq:norm}) on vectors $f_{0}\in\mathscr{D}_{-}$. To
see this, use (\ref{eq:VB-1}).
\par\end{flushleft}

The conclusions (\ref{eq:exp}) and (\ref{eq:norm}) for vector $f_{0}\in\mathscr{D}_{-}=\mathscr{H}_{0}$
may be stated in terms of the projection-valued measure
\[
E_{B}\left(\cdot\right):\left\{ \mbox{Borel-sets in }\mathbb{R}\right\} \rightarrow\left\{ \mbox{Projections in }L^{2}\left(\Omega\right)\right\} 
\]
as follows:
\begin{equation}
\left\Vert E_{B}\left(d\lambda\right)f_{0}\right\Vert _{\Omega}^{2}=\left|\hat{f}_{0}\left(\lambda\right)\right|^{2}d\lambda.\label{eq:E}
\end{equation}
It remains to prove that
\begin{equation}
\left\Vert E_{B}\left(d\lambda\right)f\right\Vert _{\Omega}^{2}=\left|\left(V_{B}f\right)\left(\lambda\right)\right|^{2}d\lambda\label{eq:E-1}
\end{equation}
holds, for all $f\in L^{2}(\Omega)$.

But by Lax-Phillips \cite{LP68} and \cite{JPT11-2}, the linear span
of the vectors
\begin{equation}
\left\{ U_{B}\left(t\right)f_{0}\:;\: t\in\mathbb{R},f_{0}\in\mathscr{D}_{-}\right\} \label{eq:span}
\end{equation}
is dense in $L^{2}\left(\Omega\right)$. This is where non-degeneracy
of $B$ is used.

As a result, it is easy to establish (\ref{eq:E-1}) when $f\in L^{2}\left(\Omega\right)$
has the form $f=U_{B}\left(t\right)f_{0}$, $t\in\mathbb{R}$, $f_{0}\in\mathscr{D}_{-}=\mathscr{H}_{0}$. 

We proceed to do this. We have:
\begin{equation}
\left\Vert E_{B}\left(d\lambda\right)U_{B}\left(t\right)f_{0}\right\Vert _{\Omega}^{2}=\left\Vert E_{B}\left(d\lambda\right)f_{0}\right\Vert _{\Omega}^{2};\label{eq:tmp-23}
\end{equation}
and for the RHS in (\ref{eq:E-1}) with $f=U_{B}\left(t\right)f_{0}$:
\begin{alignat}{1}
 & \left|\left(V_{B}U_{B}\left(t\right)f_{0}\right)\left(\lambda\right)\right|^{2}d\lambda\nonumber \\
= & \left|e_{\lambda}\left(-t\right)\left(V_{B}f_{0}\right)\left(\lambda\right)\right|^{2}d\lambda\,\,\,\,(\mbox{by }(\ref{eq:int}))\nonumber \\
= & \left|\hat{f}_{0}\left(\lambda\right)\right|^{2}d\lambda;\,\,\,\,\,\,\,\,\,\,\,\,\,\,\,\,\,\,\,\,\,\,\,\,\,\,\,\,\,\,\,\,(\mbox{by }(\ref{eq:VB-1}))\label{eq:tmp-24}
\end{alignat}
The two right-hand-sides in the last two equations (\ref{eq:tmp-23})
and (\ref{eq:tmp-24}) agree as a consequence of (\ref{eq:E}), and
we can therefore conclude that (\ref{eq:E-1}) holds for all $f=U_{B}\left(t\right)f_{0}$
as asserted.\end{proof}
\begin{rem}
The axioms for $E_{B}\left(\cdot\right)$ in (\ref{eq:E}) and (\ref{eq:E-1})
are as follows:

(i) $E_{B}\left(S\right)$ is a projection in $L^{2}\left(\Omega\right)$
for all Borel subsets $S\subset\mathbb{R}$, $S\in\mathscr{B}$; 

(ii) $\mathscr{B}\ni S\mapsto E_{B}\left(S\right)$ is countably additive;

(iii) $E_{B}\left(S_{1}\cap S_{2}\right)=E_{B}\left(S_{1}\right)E_{B}\left(S_{2}\right)$,
$\forall S_{1},S_{2}\in\mathscr{B}$;

(iv) $f=\int_{\mathbb{R}}E_{B}\left(d\lambda\right)f$ holds for all
$f\in L^{2}\left(\Omega\right)$; 

(v) $U_{B}\left(t\right)f=\int_{\mathbb{R}}e_{\lambda}\left(-t\right)E_{B}\left(d\lambda\right)f$
holds for all $f\in L^{2}\left(\Omega\right)$, $t\in\mathbb{R}$. 

The conclusion in (\ref{eq:E-1}) may be restated as follows:

For Borel sets $S$ ($\in\mathscr{B}$), let $M_{S}:=$ multiplication
by $\chi_{S}$ in $L^{2}\left(\Omega\right)$, and let $V_{B}:L^{2}\left(\Omega\right)\rightarrow L^{2}\left(\mathbb{R}\right)$
be the transform in (\ref{eq:VB}) and (\ref{eq:exp}); then 
\begin{equation}
E_{B}(S)=V_{B}^{*}M_{S}V_{B};\; S\in\mathscr{B}.\label{eq:EB}
\end{equation}

\end{rem}
\noindent \textbf{Convention.} For functions  $f$ on $\mathbb{R}$,
we set $M_{A}$ to be the corresponding multiplication operator $\left(M_{A}g\right)\left(\lambda\right)=A\left(\lambda\right)g\left(\lambda\right)$,
$\lambda\in\mathbb{R}$, with adjoint $M_{A}^{*}=M_{\overline{A}}$
, and $\overline{\cdot}$ denoting complex conjugation. On $L^{2}(\Omega)\subset L^{2}(\mathbb{R})$,
we view the Fourier transform as a unitary operator so $\mathscr{F}f=\hat{f}$,
and $\mathscr{F}^{*}g=g^{\vee}$, for all $f,g\in L^{2}(\mathbb{R})$.
\begin{cor}
\label{cor:multi}Let $\Omega$, and $B\in U(n)$ be specified as
in the theorem, and let $\{\psi_{\lambda}^{(B)}\}$ be the system
of GEFs in (\ref{eq:tmp-17}) with coefficients $\{A_{i}^{(B)}\}_{i=0}^{n}$.
Then the spectral transforms $V_{B}$ and $V_{B}^{*}$ from (\ref{eq:VB})
have the following representation as Fourier integral operators:

\begin{alignat}{1}
V_{B} & =\sum_{j=0}^{n}\mathscr{F}^{*}M_{A_{j}}^{*}\mathscr{F}P_{j},\;\mbox{and }\label{eq:vb}\\
V_{B}^{*} & =\sum_{j=0}^{n}P_{j}\mathscr{F}^{*}M_{A_{j}}\mathscr{F}\label{eq:vb-1}
\end{alignat}
where $P_{j}=M_{\chi_{J_{j}}}$, $0\leq j\leq n$, and recall
\[
\xymatrix{L^{2}(\Omega)\ar@/^{1pc}/[r]^{V_{B}} & L^{2}(\mathbb{R})\ar@/^{1pc}/[l]^{V_{B}^{*}}}
\]
\end{cor}
\begin{proof}
For all $f\in L^{2}(\Omega)$, by (\ref{eq:VB}), we have
\begin{alignat*}{1}
\left(V_{B}f\right)\left(\lambda\right) & =\int_{\Omega}\overline{\psi_{\lambda}\left(y\right)}f\left(y\right)dy\\
 & =\int_{\Omega}\overline{\left(\sum_{i=0}^{n}\chi_{i}\left(y\right)A_{i}^{\left(B\right)}\left(\lambda\right)\right)e_{\lambda}\left(y\right)}f(y)dy\\
 & =\sum_{i=0}^{n}\overline{A_{i}^{\left(B\right)}\left(\lambda\right)}\int_{\Omega}\overline{e_{\lambda}\left(y\right)}\chi_{i}\left(y\right)f(y)dy\\
 & =\sum_{i=0}^{n}\overline{A_{i}^{\left(B\right)}\left(\lambda\right)}\mathscr{F}\left(P_{i}f\right);
\end{alignat*}
and this yields (\ref{eq:vb}). On the other hand, for all $g\in L^{2}(\mathbb{R})$,
we have
\begin{alignat*}{1}
\left(V_{B}^{*}g\right)(x) & =\int_{\mathbb{R}}\psi_{\lambda}\left(x\right)g\left(\lambda\right)d\lambda\\
 & =\int_{\mathbb{R}}\left(\sum_{i=0}^{n}\chi_{i}\left(x\right)A_{i}^{\left(B\right)}\left(\lambda\right)\right)e_{\lambda}\left(x\right)g(\lambda)d\lambda\\
 & =\sum_{i=0}^{n}\chi_{i}\left(x\right)\int_{\mathbb{R}}A_{i}^{\left(B\right)}\left(\lambda\right)e_{\lambda}\left(x\right)g(\lambda)d\lambda
\end{alignat*}
which gives (\ref{eq:vb-1}).\end{proof}
\begin{rem}
For relevant details on Fourier integral operators, see e.g., \cite{Du11}.\end{rem}
\begin{cor}
\label{cor:sp}Select a pair of elements $B$ and $C$ in $U(n)$
specified as in Corollary \ref{cor:multi}. Let $(A_{i}^{(B)})$ and
$(A_{j}^{(C)})$ be the corresponding systems of scattering coefficients.
Then for the operator $V_{C}^{*}V_{B}$ we have:
\begin{equation}
V_{C}^{*}V_{B}=\sum_{i=0}^{n}\sum_{j=0}^{n}P_{i}\mathscr{F}^{*}A_{i}^{(C)}\overline{A_{j}^{(B)}}\mathscr{F}P_{j};\label{eq:multi-1}
\end{equation}
i.e. in each $(i,j)$-scattering block $V_{C}^{*}V_{B}$ has the function
\begin{equation}
\mathbb{R}\ni\lambda\mapsto A_{i}^{(C)}(\lambda)\overline{A_{j}^{(B)}(\lambda)}\label{eq:multi}
\end{equation}
as a Fourier-multiplier. \end{cor}
\begin{proof}
Immediate from (\ref{eq:vb}) and (\ref{eq:vb-1}) in Corollary \ref{cor:multi}.\end{proof}
\begin{cor}
\label{cor:sp-1}Fix $1\leq i\leq n-1$, and let $P_{i}$ be the projection
from $L^{2}(\Omega)$ onto $L^{2}(J_{i})=L^{2}\left(\alpha_{i},\beta_{i+1}\right)$.
Then for all $f\in L^{2}\left(\Omega\right)$, we have 
\begin{equation}
\left(P_{i}f\right)^{\wedge}\left(\lambda\right)=\int_{\mathbb{R}}Shann_{i}\left(\lambda-\xi\right)\left|A_{i}\right|^{2}\left(\xi\right)\left(P_{i}f\right)^{\wedge}\left(\xi\right)d\xi,\label{eq:fi}
\end{equation}
where
\begin{alignat}{1}
Shann_{i}(\xi) & :=\int_{J_{i}}\overline{e_{\xi}\left(x\right)}dx\nonumber \\
 & =e_{\xi}\left(-\frac{\alpha_{i}+\beta_{i+1}}{2}\right)\frac{\sin\left(\pi\xi\left(\beta_{i+1}-\alpha_{i}\right)\right)}{\pi\xi}\label{eq:shann}
\end{alignat}
is the Shannon kernel on the bounded interval $J_{i}$; see \cite{DyMc72}.\end{cor}
\begin{proof}
By (\ref{eq:VB-1}), we have 
\[
\left(V_{B}P_{i}f\right)\left(\lambda\right)=\overline{A_{i}\left(\lambda\right)}\left(P_{i}f\right)^{\wedge}\left(\lambda\right);
\]
hence, by (\ref{eq:exp}), 
\begin{align*}
\left(P_{i}f\right)\left(x\right) & =P_{i}\int_{\mathbb{R}}\left(V_{B}P_{i}f\right)\left(\lambda\right)\psi_{\lambda}\left(x\right)d\lambda\\
 & =\chi_{i}\left(x\right)\int_{\mathbb{R}}\left(V_{B}P_{i}f\right)\left(\lambda\right)\left(\chi_{i}\left(x\right)\psi_{\lambda}\left(x\right)\right)d\lambda\\
 & =\chi_{i}\left(x\right)\int_{\mathbb{R}}\left|A_{i}\left(\lambda\right)\right|^{2}\left(P_{i}f\right)^{\wedge}\left(\lambda\right)e_{\lambda}\left(x\right)d\lambda.
\end{align*}
Therefore,
\[
\left(P_{i}f\right)^{\wedge}\left(\lambda\right)=\int_{\mathbb{R}}\widehat{\chi_{i}}\left(\lambda-\xi\right)\left|A_{i}\left(\xi\right)\right|^{2}\left(P_{i}f\right)^{\wedge}\left(\xi\right)d\xi
\]
and (\ref{eq:fi}) holds.
\end{proof}

\subsection{An Inner Product on the System of Boundary Conditions}

While large families within the selfadjoint extensions $P_{B}$, $B\in U(n)$
are unitarily equivalent, there are much more refined measures that
pick out specific scattering theoretic properties for the selfadjoint
operators and the corresponding family of unitary one-parameter groups.
Below we compute two such; one is an inner product, or a correlation
function, defined initially on $U(n)$ and then extended by sesqui-linearity.
The second is the family of scattering semigroups, see \cite{LP68}.
\begin{cor}[An Inner Product on the System of Boundary Conditions.]
\label{cor:shann}Let $\boldsymbol{\alpha}=(\alpha_{i})_{i=1}^{n}$
and $\boldsymbol{\beta}=(\beta_{i})_{i=1}^{n}$ be fixed as above.
For each of the finite intervals $J_{j}:=(\alpha_{j},\beta_{j+1})$,
$j=1,\ldots,n-1$ in $\Omega$, let $Sh_{j}=Sh_{J_{j}}$ be the corresponding
Shannon kernel 
\begin{equation}
Sh_{j}(\lambda)=\int_{J_{j}}e(\lambda x)\, dx,\;\lambda\in\mathbb{R}.\label{eq:sh}
\end{equation}
For a pair of bounded Borel functions $g_{1}$ and $g_{2}$ on $\mathbb{R}$,
set 
\begin{equation}
\left\langle g_{1},\left|Sh_{j}\right|^{2}g_{2}\right\rangle :=\int_{\mathbb{R}}\overline{g_{1}(\lambda)}g_{2}(\lambda)\,\left|Sh_{j}\right|^{2}(\lambda)d\lambda.\label{eq:sh-1}
\end{equation}
For elements $B\in U(n)$ specified as in Theorem \ref{thm:sp-1},
let $V_{B}$ and $V_{B}^{*}$ be the associated transforms. 

For pairs of elements $B,C\in U(n)$ we have
\begin{equation}
\left\langle V_{B}\chi_{j},V_{C}\chi_{j}\right\rangle _{L^{2}(\mathbb{R})}=\left\langle A_{j}^{(B)},\left|Sh_{j}\right|^{2}A_{j}^{(C)}\right\rangle \label{eq:sh-2}
\end{equation}
and that 
\begin{equation}
\left\langle V_{B}\chi_{fin},V_{C}\chi_{fin}\right\rangle _{L^{2}(\mathbb{R})}=\sum_{j=1}^{n-1}\left\langle A_{j}^{(B)},\left|Sh_{j}\right|^{2}A_{j}^{(C)}\right\rangle \label{eq:sh-3}
\end{equation}
where $\chi_{fin}:=\sum_{j=1}^{n-1}\chi_{J_{j}}=$ indicator function
of the union $\cup_{j=1}^{n-1}J_{j}$ of the finite intervals.

Note that (\ref{eq:sh-3}) extends by sesqui-linearity to a Hilbert
inner product $\left\langle B,C\right\rangle $; and then 
\[
\left\langle B,B\right\rangle =\sum_{j=1}^{n-1}\left\langle A_{j}^{(B)},\left|Sh_{j}\right|^{2}A_{j}^{(B)}\right\rangle .
\]
\end{cor}
\begin{proof}
Follow from Theorem \ref{thm:sp-1} and Corollaries \ref{cor:sp-1}
and \ref{cor:sp-1}.\end{proof}
\begin{example}
\label{ex:inner}Let $n=2$. Fix a system of interval endpoints 
\[
-\infty<\beta_{1}<\alpha_{1}<\beta_{2}<\alpha_{2}<\infty
\]
and let $J_{-}=J_{0}=(-\infty,\beta_{1})$, $J_{1}=(\alpha_{1},\beta_{2})$,
and $J_{2}=J_{\infty}=(\alpha_{2},\infty)$. 

Let $B=\left(\begin{array}{cc}
a & b\\
-\overline{b} & \overline{a}
\end{array}\right)$ and $C=\left(\begin{array}{cc}
c & d\\
-\overline{d} & \overline{c}
\end{array}\right)$, where $a,b,c,d\in\mathbb{C}$, and $\left|a\right|^{2}+\left|b\right|^{2}=\left|c\right|^{2}+\left|d\right|^{2}=1$.
By (\ref{eq:coeff}), we have 
\begin{align}
A_{1}^{(B)}(\lambda) & ={\displaystyle \frac{a\, e_{\lambda}(\beta_{1}-\alpha_{1})}{1-b\, e_{\lambda}(L_{1})}}\label{eq:icoeff}\\
A_{1}^{(C)}(\lambda) & ={\displaystyle \frac{c\, e_{\lambda}(\beta_{1}-\alpha_{1})}{1-d\, e_{\lambda}(L_{1})}}\label{eq:icoeff-1}
\end{align}
where $L_{1}=length(J_{1})=\beta_{2}-\alpha_{1}$. Then
\begin{alignat}{1}
\left\langle B,C\right\rangle  & =\int_{\mathbb{R}}A_{1}^{(B)}(\lambda)\overline{A_{1}^{(C)}(\lambda)}\left|Sh_{1}(\lambda)\right|^{2}d\lambda\nonumber \\
 & =\left(a\,\overline{c}\right)\int_{\mathbb{R}}{\displaystyle \frac{1}{\left(1-b\, e(\lambda L_{1})\right)\left(1-\overline{d}\, e(-\lambda L_{1})\right)}}\left|Sh_{1}(\lambda)\right|^{2}d\lambda.\label{eq:iBC}
\end{alignat}
In particular, if $B=C$, we get 
\begin{equation}
\left\langle B,B\right\rangle =\left|a\right|^{2}\int_{\mathbb{R}}\frac{1}{1-2\left|b\right|\cos(2\pi(\varphi+L_{1}\lambda))+\left|b\right|^{2}}\left|Sh_{1}(\lambda)\right|^{2}d\lambda\label{eq:poisson-1}
\end{equation}
where $b:=\left|b\right|e(\varphi)$. Also see (\ref{eq:poisson-1}).\end{example}
\begin{cor}
Let $B=\left(\begin{array}{cc}
a & b\\
-\overline{b} & \overline{a}
\end{array}\right)\in U(2)$ be as above \emph{(}$b=\left|b\right|e(\varphi)$, $\left|a\right|^{2}+\left|b\right|^{2}=1$\emph{)},
then 
\begin{equation}
\left(\mbox{PER}\left|Sh_{1}(\cdot)\right|^{2}\right)(\lambda)\equiv L_{1}^{2}.\;\mbox{\emph{(See [BJ02].)}}\label{eq:w1}
\end{equation}
For the Poisson-kernel ,
\begin{equation}
P_{b}(\xi):=\frac{1-\left|b\right|^{2}}{1-2\left|b\right|\cos(2\pi\xi)+\left|b\right|^{2}},\quad\xi\in\mathbb{R},\label{eq:poisson-2}
\end{equation}
we have:
\begin{equation}
\left\langle B,B\right\rangle =L_{1}^{2}\int_{0}^{\frac{1}{L_{1}}}P_{b}(\varphi+L_{1}\lambda)d\lambda=L_{1},\label{eq:innerB}
\end{equation}
and 
\begin{alignat}{1}
\left\langle B,C\right\rangle  & =\left(a\,\overline{c}\right)L_{1}^{2}\int_{0}^{\frac{1}{L_{1}}}{\displaystyle \frac{d\lambda}{\left(1-b\, e(\lambda L_{1})\right)\left(1-\overline{d}\, e(-\lambda L_{1})\right)}}\nonumber \\
 & =\frac{a\,\overline{c}}{1-b\,\overline{d}}\, L_{1}\label{eq:innerBC}
\end{alignat}
for all $B=\left(\begin{array}{cc}
a & b\\
-\overline{b} & \overline{a}
\end{array}\right)$ and $C=\left(\begin{array}{cc}
c & d\\
-\overline{d} & \overline{c}
\end{array}\right)$ $\in U(2)$. \end{cor}
\begin{proof}
See Example \ref{ex:inner}. 

The justification for the identity (\ref{eq:w1}) is from wavelet
theory. Indeed, the assertion (\ref{eq:w1}) is equivalent with the
fact that the Shannon wavelet is an ONB-wavelet in $L^{2}(\mathbb{R})$.
The summation in (\ref{eq:PER}) below is justified since $\left|Sh_{1}\right|^{2}$
is in $L^{1}(\mathbb{R})$, and as a result $\mbox{PER}\left|Sh_{1}\right|^{2}$
is in $L^{1}$ in any period-interval.

Since the first function inside the integrals in (\ref{eq:iBC-1})
and (\ref{eq:poisson-1}) is periodic with period $1/L_{1}$, we introduce
a periodized version of the function as follows
\begin{equation}
\mbox{PER}\left|Sh_{1}\right|^{2}(\lambda):=\sum_{n\in\mathbb{Z}}\left|Sh_{1}\left(\lambda+\frac{n}{L_{1}}\right)\right|^{2}\equiv L_{1}^{2},\;(\mbox{See [BJ02]}.)\label{eq:PER}
\end{equation}
Note 
\begin{equation}
\mbox{PER}\left|Sh_{1}\right|^{2}(\lambda+\frac{1}{L_{1}})=\mbox{PER}\left|Sh_{1}\right|^{2}(\lambda),\;\forall\lambda\in\mathbb{R}\label{eq:PER-1}
\end{equation}
and as a result, we get
\begin{alignat}{1}
\left\langle B,C\right\rangle  & =\left(a\,\overline{c}\right)\int_{0}^{\frac{1}{L_{1}}}{\displaystyle \frac{\left(\mbox{PER}\left|Sh_{1}\right|^{2}\right)(\lambda)}{\left(1-b\, e(\lambda L_{1})\right)\left(1-\overline{d}\, e(-\lambda L_{1})\right)}}d\lambda\nonumber \\
 & =\left(a\,\overline{c}\right)L_{1}^{2}\int_{0}^{\frac{1}{L_{1}}}{\displaystyle \frac{1}{\left(1-b\, e(\lambda L_{1})\right)\left(1-\overline{d}\, e(-\lambda L_{1})\right)}}d\lambda\nonumber \\
 & =\left(a\,\overline{c}\right)L_{1}^{2}\int_{0}^{\frac{1}{L_{1}}}\sum_{m,n=0}^{\infty}b^{m}\overline{d^{n}}\, e((m-n)\lambda L_{1})d\lambda\nonumber \\
 & =\left(a\,\overline{c}\right)L_{1}^{2}\left(\sum_{n=0}^{\infty}(b\overline{d})^{n}\right)\int_{0}^{\frac{1}{L_{1}}}d\lambda\nonumber \\
 & =\frac{a\,\overline{c}}{1-b\,\overline{d}}\, L_{1}.\label{eq:iBC-1}
\end{alignat}
Eq (\ref{eq:innerB}) follows from this. 

A direct verification of (\ref{eq:innerB}) is given as follows:
\begin{alignat}{1}
\left\langle B,B\right\rangle  & =\left|a\right|^{2}\int_{0}^{\frac{1}{L_{1}}}\frac{\left(\mbox{PER}\left|Sh_{1}\right|^{2}\right)(\lambda)}{1-2\left|b\right|\cos(2\pi(\varphi+L_{1}\lambda))+\left|b\right|^{2}}d\lambda\nonumber \\
 & =\left|a\right|^{2}L_{1}^{2}\int_{0}^{\frac{1}{L_{1}}}\frac{1}{1-2\left|b\right|\cos(2\pi(\varphi+L_{1}\lambda))+\left|b\right|^{2}}d\lambda\nonumber \\
 & =L_{1}^{2}\int_{0}^{\frac{1}{L_{1}}}\frac{1-\left|b\right|^{2}}{1-2\left|b\right|\cos(2\pi(\varphi+L_{1}\lambda))+\left|b\right|^{2}}d\lambda=L_{1}.\label{eq:iBB}
\end{alignat}
The conclusion (\ref{eq:innerB}) follows from (\ref{eq:iBB}), (\ref{eq:poisson-2}),
and the normalization property of the Poisson kernel. 
\end{proof}

\begin{rem}
Eqs (\ref{eq:w1}) and (\ref{eq:PER}) can be verified as follows
(also see \cite{BJ02}): Let $g:=\chi_{J_{1}}$, and $\tilde{g}(x):=\overline{g(-x)}$,
so that 
\[
(g*\tilde{g})^{\wedge}(\lambda)=\left|Sh_{1}\right|^{2}(\lambda)=\left|\frac{\sin(\pi\lambda L_{1})}{\pi\lambda}\right|^{2}.
\]
It follows that
\end{rem}
\begin{alignat*}{1}
\int_{0}^{\frac{1}{L}}e^{i2\pi L_{1}n\lambda}\left(\mbox{PER}\left|Sh_{1}\right|^{2}\right)(\lambda) & =\int_{\mathbb{R}}e^{i2\pi L_{1}n\lambda}\left|Sh_{j}\right|^{2}(\lambda)d\lambda\\
 & =\left(g*\tilde{g}\right)(nL_{1})\\
 & =\begin{cases}
L_{1}^{2} & n=0\\
0 & n\neq0.
\end{cases}
\end{alignat*}
Hence, $\left(\mbox{PER}\left|Sh_{j}\right|^{2}\right)(\lambda)$,
as an $\frac{1}{L_{1}}$-periodic function, has Fourier series 
\[
\left(\mbox{PER}\left|Sh_{1}\right|^{2}\right)(\lambda)=c_{0}\equiv L_{1}^{2}.
\]

\subsection{Unitary Dilations}

For the definitions of key notions, unitary dilation for semigroups
of contractions, minimal unitary dilation, uniqueness up to unitary
equivalence, and intertwining operators, the reader is referred to
e.g., \cite{ArPr01,Jo81,LP68,Vas07}, but the literature in the subject
is extensive.

For each of the selfadjoint extension operators $P_{B}$, we now compute
three associated scattering semigroups. There is one for each of the
infinite half-lines contained in $\Omega$, and one for the union
of the bounded components in $\Omega$. 

Now, this means that when $B$ is fixed in $U(n)$ then the unitary
one-parameter group $U_{B}(t)$ will be a unitary dilation of each
of the three semigroups, one for the incoming subspace in $L^{2}(\Omega)$,
one for the outgoing, and a third for bounded components. This fact
yields additional scattering theoretic information for our problem.
\begin{rem}[Unitary dilations]
\label{rmk:dilation} Now, as before pick $n\ge2$, and fix the open
set $\Omega$ with interval endpoints as in section \ref{sub:sp-eig}
(see Figures \ref{fig:eigen} and \ref{fig:decouple}.) When a boundary
condition $B$ is fixed, we get an associated unitary one parameter
group $U_{B}(t)$. We claim that this is, at the same time, a unitary
dilation of three different semigroups of contractive operators (called
\textquotedblleft{}contraction semigroups.\textquotedblright{})

For every $B\in U(n)$ we have a strongly continuous unitary representation
$U_{B}(\cdot)$ of $(\mathbb{R},+)$, i.e., of the additive group
of $\mathbb{R}$. When $t\in\mathbb{R}$ is fixed, $U_{B}(t)$ is
a unitary operator acting in $L^{2}(\Omega)$.

Associated with $(U_{B}(t),L^{2}(\Omega))$, one has three contraction
semigroups. Each of the three semigroups is the result of cutting
down $U_{B}(t)$, $t>0$, with three separate projections.

First, let $P_{-}$ be the projection onto $L^{2}(J_{-})$, see Figure
\ref{fig:decouple}, let $P_{+}$ be the projection onto $L^{2}(J_{+})$,
and finally $P_{m}$ denotes the projection onto \textquotedblleft{}the
rest,\textquotedblright{} i.e., onto $L^{2}$ of the union of the
$n-1$ finite intervals $J_{i}$, see Fig \ref{fig:decouple}. 

The three semigroups are now as follows:\end{rem}
\begin{enumerate}
\item \label{enu:Z_minus}$Z_{-}(t):=P_{-}U_{B}(t)P_{-}$, $t>0$; its infinitesimal
generator has von Neumann indices $(0,1)$.
\item \label{enu:Z_plus}$Z_{+}(t):=P_{+}U_{B}(t)P_{+}$, $t>0$; its infinitesimal
generator has von Neumann indices $(1,0)$
\item $Z_{m}(t):=P_{m}U_{B}(t)P_{m}$, $t>0$; its infinitesimal generator
is maximal dissipative, see \cite{JoMu80,LP68}.
\end{enumerate}
Moreover one easily checks (\cite{Ko09,JoMu80,LP68}) that the first
is a semigroup of co-isometries, the second a semigroup of isometries,
and the third one $Z_{m}(t)$ is a semigroup of contraction operators;
often called the Lax-Phillips scattering semigroup. 

\begin{figure}
\includegraphics[scale=0.8]{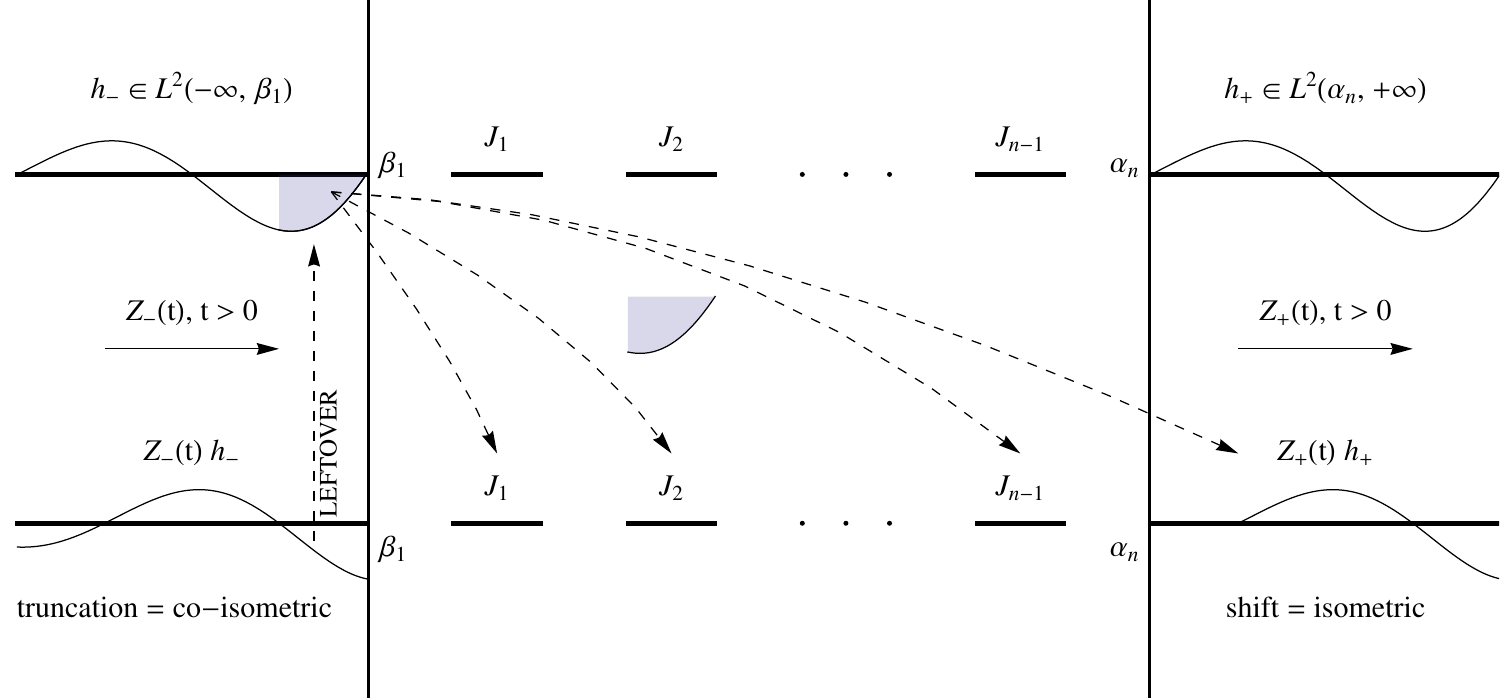}

\caption{\label{fig:scatter}The action of the two semigroups, $Z_{-}(t)$
acting on $L^{2}(-\infty,\beta_{1})$, and $Z_{+}(t)$ acting on $L^{2}(\alpha_{n},\infty)$,
$t>0$ fixed. The semigroup $Z_{+}(t)$ is called a semigroup of shifts:
Each operator $Z_{+}(t)$ shifts functions to the right. A given function
$h_{+}$ on the halfline $(\alpha_{n},\infty)$ is shifted like a
Xerox copy. $Z_{+}(t)$ moves it $t$ units to the right. The $L^{2}$
norm is preserved: It is a semigroup of isometries. By contrast, moving
instead to the left in $(\alpha_{n},\infty)$, we get a semigroup
of co-isometries; hence truncation, and the $L^{2}$ norm is decreased.
By symmetry, this duality works in the reverse on the LHS half-line
$(-\infty,\beta_{1})$. In general the adjoint of a semigroup of isometries
is a semigroup of co-isometries.}

\end{figure}

\begin{lem}
If $Y_{\pm}$ denote the respective infinitesimal generators of the
two contraction semigroups $\{Z_{\pm}(t)\}_{t\in\mathbb{R}_{+}}$
in (\ref{enu:Z_minus}) and (\ref{enu:Z_plus}) above, then for their
respective dense domains we have:
\begin{equation}
\mathscr{D}(Y_{+})=\{f\:;\: f\mbox{ and }f'\in L^{2}(\alpha_{n},\infty),f(\alpha_{n})=0\}\label{eq:Dom_plus}
\end{equation}
dense in $\mathscr{H}_{+}=L^{2}(J_{+})$, $J_{+}=(\alpha_{n},\infty)$;
and
\begin{equation}
\mathscr{D}(Y_{-})=\{f\:;\: f\mbox{ and }f'\in L^{2}(\infty,\beta_{1})\}\label{eq:Dom_minus}
\end{equation}
dense in $\mathscr{H}_{-}=L^{2}(J_{-})$, $J_{-}=(-\infty,\beta_{1})$. 

As before the interval systems $\boldsymbol{\alpha}=(\alpha_{i})$,
and $\boldsymbol{\beta}=(\beta_{i})$ are specified as in Figure \ref{fig:lengths}. \end{lem}
\begin{proof}
See e.g., \cite{LP68}.\end{proof}
\begin{rem}
Using \cite{Ko09}, one further checks that $(U_{B}(t),L^{2}(\Omega))$
serves as a unitary dilation of all three semigroups. Recalling \cite{Ko09}
finally, that, in general, unitary dilations may or may not be minimal.
As an application of Theorem \ref{thm:sp-1}, it follows that $(U_{B}(t),L^{2}(\Omega))$
is a minimal dilation of $Z_{-}(t)$, $t>0$, if and only if $B$
is non-degenerate. In this case, it is also a minimal unitary dilation
for the semigroup on the right $Z_{+}(t)$.

But minimal unitary dilations are unique up to unitary equivalence
\cite{Ko09}. As a result, we get a system of unitary intertwining
operators. Below we outline formulas for these intertwining operators,
and their relevance for scattering theory, and for bound-states.\end{rem}
\begin{cor}
If $B\in U(n)$ is non-degenerate, then the scattering coefficients
$A_{j}^{(B)}$ in (\ref{eq:tmp-16}) satisfy $\left|A_{j}^{(B)}(\lambda)\right|>0$
for all $\lambda\in\mathbb{R}$. Moreover if $B_{1},B_{2}\in U(n)$
are both non-degenerate, then there is a unitary intertwining operator
$W$ in $L^{2}(\Omega)$ subject to the following two conditions:
\begin{equation}
\begin{cases}
Wh=h,\;,\forall h\in L^{2}(-\infty,\beta_{1})\\
WU_{B_{1}}(t)h=U_{B_{2}}(t)h,\;\forall t\in\mathbb{R},h\in L^{2}(-\infty,\beta_{1}).
\end{cases}\label{eq:BB}
\end{equation}
In Fourier domain, it is determined as follows:
\begin{equation}
\chi_{J_{j}}(x)W\left(\chi_{J_{j}}(\cdot)g^{\vee}(\cdot)\right)=\chi_{J_{j}}(x)\left(\frac{A_{j}^{(B_{2})}}{A_{j}^{(B_{1})}}g\right)^{\vee}(x)\label{eq:scatter}
\end{equation}
for $1\leq j\leq n-1$, $g\in L^{2}(\mathbb{R})$, $x\in\Omega$;
where $g^{\vee}=\mathscr{F}^{*}g$ denotes the inverse Fourier transform
in $L^{2}(\mathbb{R})$.
\end{cor}

\section{Degenerate Cases\label{sec:deg}}

In this section we make a comparison between families of selfadjoint
extensions that have purely continuous spectrum, and the cases with
embedded point-spectrum. We outline detailed scattering properties;
and in particular, we give examples of non-periodic periodic spectrum;
see Theorem \ref{thm:nonperiodic}, and the caption in Figure \ref{fig:dense-orbits}.

The following features from quantum theory are reflected in property
certain of our operators $P_{B}$ which allow decomposition (degeneracy);
see e.g., Theorem \ref{thm:Bdecomp} above, and Figure \ref{fig:decouple};
as well as details in the section below. States in quantum mechanics
are represented by wave functions, and they in turn by vectors (of
unit-norm) in Hilbert space. In the paragraphs below, we will use
wave-particle duality (from quantum theory) without further discussion;
hence referring on occasion to particles as opposed to wave functions.
Bound-states are states that satisfy some additional confinement property.
Consider now a quantum system where particles (waves) are subject
to confinement, by a potential, or by a spatial region, e.g., a box,
or an interval. These particles then have a tendency to remain localized
in one or more regions of space. In our present analysis we will identify
this case by a discrete set of spectral-points, embedded in a continuum
spectrum (embedded point-spectrum); hence the presence of eigenvectors
for the operator $P_{B}$ under consideration. The spectrum, for us,
will refer to $P_{B}$, one in a family of selfadjoint operators (quantum
mechanical observables, such as momentum, or position).

In quantum systems (where the number of particles is conserved), a
bound-state is a unit-norm vector in a Hilbert space which is also
an eigenvector. They may result from two or more particles whose interaction
energy is less than the total energy of each separate particle. Hence
these particles cannot be separated unless energy is spent. The mathematical
consequence is that the corresponding energy spectrum for a bound-state
is discrete, and in our case, embedded in continuous spectrum. Bound-states
may be stable or unstable, and this distinction will be illustrated
for our model below. Positive interaction energy for bound-states
corresponds to \textquotedbl{}energy barriers,\textquotedbl{} and
a fraction of the states will tunnel through the barriers, and eventually
decay. Stable bound-states are associated to, among other things,
stationary wave-functions, and they may show up as a poles in a scattering
matrix (details below.) 

Fix $n=3$. Recall that $\Omega=\bigcup_{i=0}^{4}J_{i}$, where $J_{1}=(\alpha_{1},\beta_{2})$,
$J_{2}=(\alpha_{2},\beta_{3})$, $J_{-}=J_{0}=(-\infty,\beta_{1})$,
and $J_{+}=J_{3}=(\alpha_{3},\infty)$. For $B\in U(3)$, the generalized
eigenfunction is 
\begin{equation}
\psi_{\lambda}\left(x\right):=\psi_{\lambda}^{(B)}(x)=\left(\sum_{i=0}^{4}A_{i}(\lambda)\chi_{i}(x)\right)e_{\lambda}(x);\label{eq:GEF-1}
\end{equation}
with $\chi_{i}:=\chi_{J_{i}}$, $i=0,1,2,3$. The coefficients $\left(A_{i}\right)_{i=0}^{3}$
satisfies the boundary condition 
\begin{equation}
B_{\alpha,\beta}(\lambda)\left(\begin{array}{c}
A_{0}\\
A_{1}\\
A_{2}
\end{array}\right)=\left(\begin{array}{c}
A_{1}\\
A_{2}\\
A_{3}
\end{array}\right)\label{eq:bd-5}
\end{equation}
as in (\ref{eq:gef-1}).

\subsection{\label{sub:case1}Case 1: With $U_{B}(t)$ indecomposable}

Let 
\[
B=\left(\begin{array}{ccc}
a & b & 0\\
-\overline{b} & \overline{a} & 0\\
0 & 0 & 1
\end{array}\right)\in U(3)
\]
where $\left|a\right|^{2}+\left|b\right|^{2}=1$. Then 
\[
B'=\left(\begin{array}{cc}
b & 0\\
\overline{a} & 0
\end{array}\right),\;\boldsymbol{w}=\left(\begin{array}{c}
0\\
1
\end{array}\right),\:\mbox{and }B'^{*}B'=\left(\begin{array}{cc}
1 & 0\\
0 & 0
\end{array}\right)=I-P_{\boldsymbol{w}}.
\]
Note that $\left\Vert B'\right\Vert =1$. 

Standing assumption $0<\left|b\right|<1$. The notation used in the
example is the one introduced above.

\textbf{Conclusions:} After a computation we arrive at a closed-form
formula for all four generalized eigenfunction (GEF) coefficients
$A_{i}$, with the local index $i$ from $0$ to $3$; and, as a result,
a closed-form formula for the GEFs $\psi_{\lambda}^{B}$ in Theorem
\ref{thm:UB}; see also (\ref{eq:tmp-17}).

To ensure that the measure in the spectral representation $\sigma_{B}$
is Lebesgue measure, we pick $A_{0}=1$. Some noteworthy properties
of the coefficients: The coefficient $A_{1}$ for the first of the
finite intervals inside $\Omega$, carries more information than the
remaining three coefficients. Studying transformation of states in
$L^{2}(\Omega)$ under unitary one-parameter group $U_{B}(t)$, with
$t$ increasing, we note that the last GEF-coefficient $A_{3}(\lambda)$
measures transition into the infinite half-line to the right. It turns
out that the last coefficient, $A_{3}$ is just a phase factor times
the unitary scattering operator $A_{2}$. All coefficients, phase
factors, and time-delay depend on the respective lengths of the finite
intervals in $\Omega$, as well as the lengths of the gaps between
them.

The $A_{2}$ function is a scattering operator (see \cite{JPT11-2})
adjusted both with a phase factor and an additive time-delay. Hence,
three of the four GEF-coefficient have modulus $1$, i.e, $\left|A_{i}\left(\lambda\right)\right|=1$
, for $i=0,2$ and $3$. The coefficient $\left|A_{1}\right|^{2}$
carries a probabilistic interpretation. It is a scaled Poisson kernel,
with the scaling depending on a two-state distribution $\left|a\right|^{2}+\left|b\right|^{2}=1$,
where $a$ and $b$ are complex, the $SU(2)$ entries from $B$.

As a result, in the spectral decomposition (Theorem \ref{thm:sp-1}),
we get local densities $=1$, except at one place, for first of the
finite intervals $J_{1}$, where the distribution density is $\left|A_{1}\right|^{2}$.
So by contrast to the case $n=2$ \cite{JPT11-2}, in the present
model, we do not have Poisson uniformly contributing to $\sigma_{B}$.
The spectrum of $U_{B}(t)$ is pure Lebesgue spectrum, with no embedded
point-spectrum. And the global Hilbert space $L^{2}(\Omega)$ does
not decompose.

The boundary conditions are as follows:
\[
\left(\begin{array}{c}
f(\alpha_{1})\\
f(\alpha_{2})\\
f(\alpha_{3})
\end{array}\right)=\left(\begin{array}{ccc}
a & b & 0\\
-\overline{b} & \overline{a} & 0\\
0 & 0 & 1
\end{array}\right)\left(\begin{array}{c}
f(\beta_{1})\\
f(\beta_{2})\\
f(\beta_{3})
\end{array}\right).
\]
In details, we get the following transformations:

\[
\begin{cases}
f(\alpha_{1}) & =af(\beta_{1})+bf(\beta_{2})\\
f(\alpha_{2}) & =-\overline{b}f(\beta_{1})+\overline{a}f(\beta_{2})\\
f(\alpha_{3}) & =f(\beta_{3}).
\end{cases}
\]

\begin{figure}[H]
\includegraphics[scale=0.8]{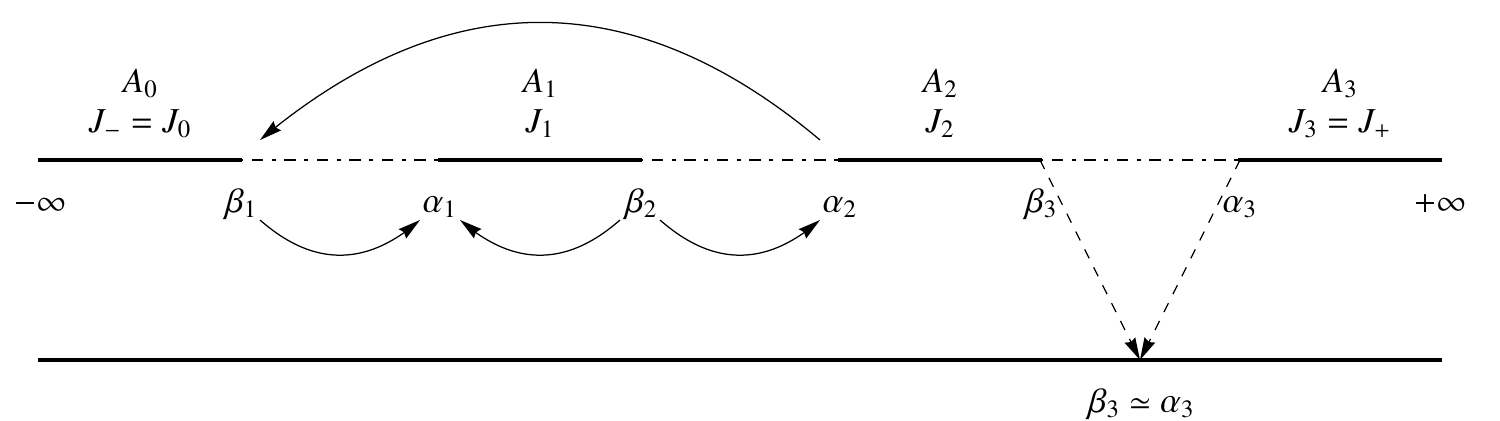}

\caption{Transition between intervals in $\Omega$}

\end{figure}

Setting
\begin{alignat*}{1}
B_{\alpha,\beta}(\lambda) & =D_{\alpha}(\lambda)^{*}BD_{\beta}(\lambda)\\
 & =\left(\begin{array}{ccc}
a\, e_{\lambda}(\beta_{1}-\alpha_{1}) & b\, e_{\lambda}(\beta_{2}-\alpha_{1}) & \underset{}{0}\\
-\overline{b}\, e_{\lambda}(\beta_{1}-\alpha_{2}) & \overline{a}\, e_{\lambda}(\beta_{2}-\alpha_{2}) & \underset{}{0}\\
0 & 0 & e_{\lambda}(\beta_{3}-\alpha_{3})
\end{array}\right)
\end{alignat*}
so that
\[
B'_{\alpha,\beta}(\lambda)=\left(\begin{array}{cc}
\underset{}{b\, e_{\lambda}(\beta_{2}-\alpha_{1})} & 0\\
\overline{a}\, e_{\lambda}(\beta_{2}-\alpha_{2}) & 0
\end{array}\right);
\]
and we have
\begin{equation}
B_{\alpha,\beta}(\lambda)\left(\begin{array}{c}
A_{0}\\
A_{1}\\
A_{2}
\end{array}\right)=\left(\begin{array}{c}
A_{1}\\
A_{2}\\
A_{3}
\end{array}\right);\label{eq:bd-4}
\end{equation}
i.e.,
\begin{alignat*}{1}
a\, e_{\lambda}(\beta_{1}-\alpha_{1})A_{0}+b\, e_{\lambda}(\beta_{2}-\alpha_{1})A_{1} & =A_{1}\\
-\overline{b}\, e_{\lambda}(\beta_{1}-\alpha_{2})A_{0}+\overline{a}\, e_{\lambda}(\beta_{2}-\alpha_{2})A_{1} & =A_{2}\\
e_{\lambda}(\beta_{3}-\alpha_{3})A_{2} & =A_{3}.
\end{alignat*}

The characteristic polynomial of $B'_{\alpha,\beta}(\lambda)$ is
\[
\det\left(\begin{array}{cc}
\underset{}{x-b\, e_{\lambda}(\beta_{2}-\alpha_{1})} & 0\\
-\overline{a}\, e_{\lambda}(\beta_{2}-\alpha_{2}) & x
\end{array}\right)=x\left(x-b\, e_{\lambda}(l_{1})\right)
\]
with roots $x=0$, $x=b\, e_{\lambda}(l_{1})\neq1$. But since $\left|b\right|<1$
we get $1\notin sp(B'_{\alpha,\beta}(\lambda))$, $\forall\lambda\in\mathbb{R}$,
hence $\Lambda_{pt}=\phi$, and $(I_{2}-B_{\alpha,\beta}(\lambda))^{-1}$
is well-defined for all $\lambda\in\mathbb{R}$.

Note that
\begin{alignat*}{1}
(I_{2}-B_{\alpha,\beta}(\lambda))^{-1} & =\left(\begin{array}{cc}
\underset{}{1-b\, e_{\lambda}(\beta_{2}-\alpha_{1})} & 0\\
-\overline{a}\, e_{\lambda}(\beta_{2}-\alpha_{2}) & 1
\end{array}\right)\\
 & =\frac{1}{1-b\, e_{\lambda}(\beta_{2}-\alpha_{1})}\left(\begin{array}{cc}
\underset{}{1} & 0\\
\overline{a}\, e_{\lambda}(\beta_{2}-\alpha_{2}) & 1-b\, e_{\lambda}(\beta_{2}-\alpha_{1})
\end{array}\right).
\end{alignat*}
Setting $A_{0}=1$, it follows that
\begin{alignat*}{1}
\left(\begin{array}{c}
A_{1}\\
A_{2}
\end{array}\right) & =(I_{2}-B_{\alpha,\beta}(\lambda))^{-1}\left(\begin{array}{c}
\underset{}{a\, e_{\lambda}(\beta_{1}-\alpha_{1})}\\
-\overline{b}\, e_{\lambda}(\beta_{1}-\alpha_{2})
\end{array}\right)\\
 & =\frac{1}{1-b\, e_{\lambda}(\beta_{2}-\alpha_{1})}\left(\begin{array}{c}
\underset{}{a\, e_{\lambda}(\beta_{1}-\alpha_{1})}\\
{\displaystyle e_{\lambda}(\beta_{1}+\beta_{2}-\alpha_{1}-\alpha_{2})-\overline{b}\, e_{\lambda}(\beta_{1}-\alpha_{2})}
\end{array}\right).
\end{alignat*}
Finally, 
\begin{alignat*}{1}
A_{3} & =\frac{1}{1-b\, e_{\lambda}(\beta_{2}-\alpha_{1})}\left\langle \left(\begin{array}{c}
0\\
e_{\lambda}(\beta_{3}-\alpha_{3})
\end{array}\right),\left(\begin{array}{c}
\underset{}{{\displaystyle a\, e_{\lambda}(\beta_{1}-\alpha_{1})}}\\
e_{\lambda}(\beta_{1}+\beta_{2}-\alpha_{1}-\alpha_{2})-\overline{b}\, e_{\lambda}(l_{1})
\end{array}\right)\right\rangle \\
 & =e_{\lambda}(\beta_{3}-\alpha_{3})A_{2}
\end{alignat*}

We summarize the results in the lemma below:
\begin{lem}
The solution to (\ref{eq:bd-4}) is given by

\begin{alignat}{1}
A_{0} & =\underset{}{1}\label{eq:a0}\\
A_{1} & \underset{}{=\frac{a\, e_{\lambda}(\beta_{1}-\alpha_{1})}{1-b\, e_{\lambda}(\beta_{2}-\alpha_{1})}}\label{eq:a1}\\
A_{2} & =\underset{}{\frac{{\displaystyle e_{\lambda}(\beta_{1}-\alpha_{2})\left(e_{\lambda}(\beta_{2}-\alpha_{1})-\overline{b}\right)}}{1-b\, e_{\lambda}(\beta_{2}-\alpha_{1})}}\label{eq:a2}\\
A_{3} & \underset{}{=e_{\lambda}(\beta_{3}-\alpha_{3})A_{2}}\label{eq:a3}
\end{alignat}

\end{lem}

\begin{lem}
Setting $b=\left|b\right|e(\varphi)$, $\varphi\in\mathbb{R}$, then
\begin{equation}
\left|A_{1}\right|^{2}=\frac{\left|a\right|^{2}}{1-2\left|b\right|\cos(2\pi(\varphi+l_{1}\lambda))+\left|b\right|^{2}}.\label{eq:poisson}
\end{equation}
Hence $|A_{1}|^{2}$ is the Poisson kernel with parameter $b$.\end{lem}
\begin{proof}
This follows from (\ref{eq:a1}).
\end{proof}

\subsection{\label{sub:case2}Case 2: With $U_{B}(t)$ decomposable}

Let 
\begin{equation}
B=\left(\begin{array}{ccc}
0 & a & b\\
0 & -\overline{b} & \overline{a}\\
1 & 0 & 0
\end{array}\right)\in U(3)\label{eq:B-5}
\end{equation}
where 
\begin{equation}
B'=\left(\begin{array}{cc}
a & b\\
-\overline{b} & \overline{a}
\end{array}\right)\in SU(2),\label{eq:B-6}
\end{equation}
i.e., $\left|a\right|^{2}+\left|b\right|^{2}=1$; and $\boldsymbol{u}=\boldsymbol{w}=\left(\begin{array}{c}
0\\
0
\end{array}\right)$, and $c=1$.

\textbf{Summary of conclusions in the Example.} Standing assumption
$0<\left|b\right|<1$.

The notation used in the example is as before, but the element $B$
in $U(3)$ is now different. We also fix a system $\boldsymbol{\alpha}$
and $\boldsymbol{\beta}$ of interval endpoints, subject to the standard
position, see (\ref{eq:Omega}) through (\ref{eq:Jex}) in Section
\ref{sub:bform}.

As always, the conclusions will depend on both $B$ and the prescribed
pair $\boldsymbol{\alpha}$ and $\boldsymbol{\beta}$: Again, we arrive
at a closed-form formula for the generalized eigenfunction (GEF) $\psi_{\lambda}^{B}$;
see Theorem \ref{thm:UB} and (\ref{eq:tmp-17}). But this time, we
get embedded point-spectrum in the continuum (boundstates in physics
lingo.)

The discrete set $\Lambda_{pt}$ making up the point-spectrum depends
on both the lengths of the two finite intervals $J_{1}$ and $J_{2}$,
as well as on the gap between them, and the gaps to the infinite half-lines.

As before, to get the continuous part of $\sigma_{B}$ to be Lebesgue
measure on $\mathbb{R}$, we pick $A_{0}=1$.

Studying transformation of states in $L^{2}(\Omega)$ under unitary
one-parameter group $U_{B}(t)$, with $t$ increasing, we note that
incoming states from the infinite half-line to the left turn into
boundstates. But the action of $U_{B}(t)$ on the global Hilbert space
$L^{2}(\Omega)$ now decomposes as an orthogonal sum of continuous
states, and boundstates.

As a result, in the spectral decomposition (Theorem \ref{thm:sp-1}),
we get local densities $=1$, for the continuous part, and a set of
Dirac-combs for the discrete part. But by contrast to the case $n=2$
\cite{JPT11-2}, in the present model, we get non-periodic Dirac-combs.
The spectrum of $U_{B}(t)$ is a mix of Lebesgue spectrum and embedded
point-spectrum.

The boundary condition takes the form
\[
\left(\begin{array}{c}
f(\alpha_{1})\\
f(\alpha_{2})\\
f(\alpha_{3})
\end{array}\right)=\left(\begin{array}{ccc}
0 & a & b\\
0 & -\overline{b} & \overline{a}\\
1 & 0 & 0
\end{array}\right)\left(\begin{array}{c}
f(\beta_{1})\\
f(\beta_{2})\\
f(\beta_{3})
\end{array}\right)
\]
so that
\begin{equation}
\begin{cases}
f(\alpha_{1}) & =af(\beta_{1})+bf(\beta_{2})\\
f(\alpha_{2}) & =-\overline{b}f(\beta_{1})+\overline{a}f(\beta_{2})\\
f(\alpha_{3}) & =f(\beta_{1}).
\end{cases}\label{eq:bd-6}
\end{equation}
See the second line in Fig \ref{fig:case2} (and also Fig. \ref{fig:decouple})
for a geometric representation of the last equation $f(\alpha_{3})=f(\beta_{1})$
in the system (\ref{eq:bd-6}) of boundary conditions.

\begin{figure}
\includegraphics[scale=0.8]{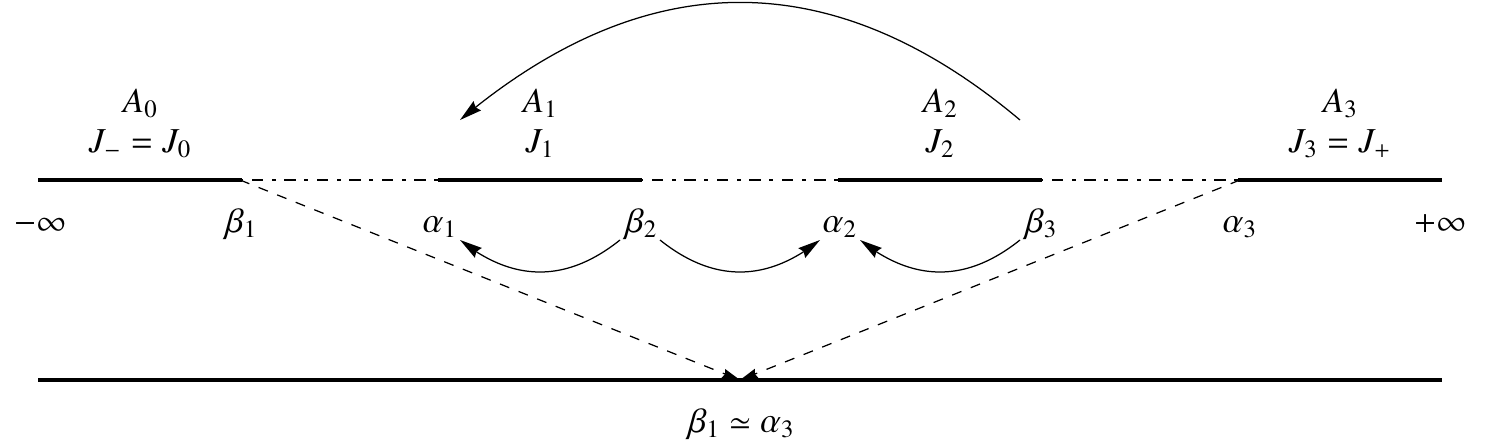}
\begin{alignat*}{1}
L^{2}(\Omega) & =L^{2}(J_{-}\cup J_{+})\oplus L^{2}(J_{1}\cup J_{2})\\
U_{B}(t) & =U_{B}^{cont}(t)\oplus U_{B}^{boundstate}(t)
\end{alignat*}

\caption{\label{fig:case2}Transition between intervals in $\Omega$.}
\end{figure}

In this case,
\[
B_{\alpha,\beta}(\lambda)=\left(\begin{array}{ccc}
\underset{}{0} & a\, e_{\lambda}(\beta_{2}-\alpha_{1}) & b\, e_{\lambda}(\beta_{3}-\alpha_{1})\\
\underset{}{0} & -\overline{b}\, e_{\lambda}(\beta_{2}-\alpha_{2}) & \overline{a}\, e_{\lambda}(\beta_{3}-\alpha_{2})\\
1 & 0 & 0
\end{array}\right)
\]
where
\[
B'_{\alpha,\beta}(\lambda)=\left(\begin{array}{cc}
\underset{}{a\, e_{\lambda}(\beta_{2}-\alpha_{1})} & b\, e_{\lambda}(\beta_{3}-\alpha_{1})\\
-\overline{b}\, e_{\lambda}(\beta_{2}-\alpha_{2}) & \overline{a}\, e_{\lambda}(\beta_{3}-\alpha_{2})
\end{array}\right)\in SU(2),\;\forall\lambda\in\mathbb{R}.
\]
Now, the boundary condition (\ref{eq:bd-5}) becomes
\begin{alignat*}{1}
a\, e_{\lambda}(\beta_{2}-\alpha_{1})A_{1}+b\, e_{\lambda}(\beta_{3}-\alpha_{1})A_{2} & =A_{1}\\
-\overline{b}\, e_{\lambda}(\beta_{2}-\alpha_{2})A_{1}+\overline{a}\, e_{\lambda}(\beta_{3}-\alpha_{2})A_{2} & =A_{2}\\
e_{\lambda}(\beta_{1}-\alpha_{3})A_{0} & =A_{3}.
\end{alignat*}
We set $A_{0}(\lambda)\equiv1$ for all $\lambda\in\mathbb{R}$. 

As a result, we see that the vector
\begin{equation}
\boldsymbol{A}(\lambda)=\left(\begin{array}{c}
A_{1}(\lambda)\\
A_{2}(\lambda)
\end{array}\right)\label{eq:soln-4}
\end{equation}
must be an eigenvector of $ $$B'_{\alpha,\beta}(\lambda)$ for $\lambda$
to be in the spectrum of $P_{B}$, or equivalently for the unitary
one-parameter group $U_{B}(t)$.

As a result we get
\begin{equation}
\Lambda_{pt}=\{\lambda\in\mathbb{R}\:;\:\det(I_{2}-B'_{\alpha,\beta}(\lambda))^{-1}=0\}.\label{eq:pt-3}
\end{equation}
Hence, as the interval endpoints $\boldsymbol{\alpha}=\left(\alpha_{i}\right)$
and $\boldsymbol{\beta}=\left(\beta_{i}\right)$ fixed, the set $\Lambda_{pt}$
results as the solution manifold for 
\begin{equation}
\det\left(\begin{array}{cc}
\underset{}{1-a\, e_{\lambda}(\beta_{2}-\alpha_{1})} & -b\, e_{\lambda}(\beta_{3}-\alpha_{1})\\
\overline{b}\, e_{\lambda}(\beta_{2}-\alpha_{2}) & 1-\overline{a}\, e_{\lambda}(\beta_{3}-\alpha_{2})
\end{array}\right)=0.\label{eq:det-1}
\end{equation}
Notice (\ref{eq:det-1}) is independent of $\beta_{1}$ and $\alpha_{3}$.
\begin{example}
\label{ex:case2}Let 
\begin{align*}
\boldsymbol{\alpha} & =\{1,2,3+\varphi\},\varphi>0\\
\boldsymbol{\beta} & =\{0,\frac{3}{2},3\},\;\mbox{and }\\
a & =b=\frac{1}{\sqrt{2}}.
\end{align*}
Substitute into (\ref{eq:det-1}): 
\[
\det\left(\begin{array}{cc}
\underset{}{1-\frac{1}{\sqrt{2}}\, e_{\lambda}(\frac{1}{2})} & -\frac{1}{\sqrt{2}}\, e_{\lambda}(2)\\
\frac{1}{\sqrt{2}}\, e_{\lambda}(-\frac{1}{2}) & 1-\frac{1}{\sqrt{2}}\, e_{\lambda}(1)
\end{array}\right)=1-\frac{1}{\sqrt{2}}e(\frac{1}{2}\lambda)-\frac{1}{\sqrt{2}}e(\lambda)+e(\frac{3}{2}\lambda).
\]
As a result 
\[
z(\lambda):=e(\frac{1}{2}\lambda)=\cos(\pi\lambda)+i\sin(\pi\lambda)
\]
must satisfy the following cubic equation
\begin{equation}
1-\frac{1}{\sqrt{2}}z-\frac{1}{\sqrt{2}}z{}^{2}+z{}^{3}=(1+z)(z^{2}-(1+\frac{1}{\sqrt{2}})z+1)=0.\label{eq:tmp-25}
\end{equation}
Hence, 
\[
z=-1\Longleftrightarrow\lambda\in1+2\mathbb{Z};
\]
or 
\[
z^{2}-(1+\frac{1}{\sqrt{2}})z+1=0\Longleftrightarrow z_{\pm}=\frac{\left(1+\frac{1}{\sqrt{2}}\right)\pm i\sqrt{\frac{5}{2}-\sqrt{2}}}{2}.
\]
Note $\left|z_{\pm}\right|=1$. Let $\lambda_{\pm}$ be such that$ $
$z_{\pm}=e(\frac{1}{2}\lambda_{\pm})$, with $\lambda_{\pm}\in\mathbb{R}$.
We conclude that
\[
\Lambda_{pt}=\left(1+2\mathbb{Z}\right)\cup\{\lambda_{\pm}+\frac{1}{2}\mathbb{Z}\}.
\]
\end{example}
\begin{rem}
In any example with $\Omega$ as in Figure \ref{fig:case2}, i.e.,
when $\Omega$ is the complement of three finite closed intervals,
$\Omega$ will have two bounded components, i.e., open intervals $J_{i}$,
$i=1,2$; and two unbounded. If further $U_{B}(t)$ is assumed decomposable,
there will be one summand $U^{bdst}{}_{B}(t)$ acting on $L^{2}(J_{1}\cup J_{2})$
of the union of the two intervals $J_{i}$.

Example \ref{ex:case2} produces one particular configuration for
this possibility, so a computation of the spectrum of $U^{bdst}{}_{B}(t)$
when there are boundstates. In an earlier paper \cite{JPT11-1} we
found all the configurations for spectrum for each one of the possible
momentum operators in $L^{2}$ of the union any pair of finite open
intervals $J_{i}$.

This in turn is a question of interest both for the study of both
quantum systems, and of spectral pairs, see e.g., \cite{Fu74,JP98,JP99,DJ11,Laba01}.
\end{rem}

\begin{rem}
More generally, if 
\[
B_{1}=\left(\begin{array}{ccccc}
 &  &  & \vline & 0\\
 & \huge\mbox{\textbf{g}} &  & \vline & \vdots\\
 &  &  & \vline & 0\\
\hline 0 & \cdots & 0 & \vline & 1
\end{array}\right)\in U(n)
\]
with $\boldsymbol{g}\in SU(n-1)$, then the corresponding unitary
one-parameter group $U_{B_{1}}(t)$ does not decompose. See case 1
of section \ref{sub:case1}.

On the other hand, for 
\[
B_{2}=\left(\begin{array}{ccccc}
0 & \vline\\
\vdots & \vline &  & \huge\mbox{\textbf{g}}\\
0 & \vline\\
\hline 1 & \vline & 0 & \cdots & 0
\end{array}\right)\in U(n)
\]
where $\boldsymbol{g}\in SU(n-1)$ as before, the unitary group $U_{B_{2}}(t)$
decomposes. See case 2 of section \ref{sub:case2}.

Set 
\[
S=\left(\begin{array}{cccc}
0 & \cdots & 0 & 1\\
1 & 0 &  & 0\\
 & \ddots & \ddots & \vdots\\
\huge\mbox{0} &  & 1 & 0
\end{array}\right)
\]
then (see (\ref{eq:BS}))
\[
\left(\begin{array}{ccccc}
0 & \vline\\
\vdots & \vline &  & \huge\mbox{\textbf{g}}\\
0 & \vline\\
\hline 1 & \vline & 0 & \cdots & 0
\end{array}\right)S=\left(\begin{array}{ccccc}
 &  &  & \vline & 0\\
 & \huge\mbox{\textbf{g}} &  & \vline & \vdots\\
 &  &  & \vline & 0\\
\hline 0 & \cdots & 0 & \vline & 1
\end{array}\right).
\]

\end{rem}
\medskip{}

\begin{thm}
\label{thm:nonperiodic}Let $L_{i}=\mbox{length}(J_{i})$, $i=1,2$,
be the lengths of the two bounded intervals $J_{1}$ and $J_{2}$.
Then, 
\begin{alignat}{1}
D(\lambda,a,L_{1},L_{2}) & :=\det\left(B'_{\alpha,\beta}(\lambda)\right)\nonumber \\
 & =1+e_{\lambda}(L_{1}+L_{2})-a\, e_{\lambda}(L_{1})-\overline{a}\, e_{\lambda}(L_{2});\label{eq:ep}
\end{alignat}
where $\det(B'_{\alpha,\beta}(\lambda))$ is defined in (\ref{eq:det-1}).
The solution manifold $\Lambda_{pt}$ (\ref{eq:pt-3}), i.e., the
embedded point-spectrum, is the set of zeros 
\begin{equation}
Z(a,L_{1},L_{2}):=\{\lambda\in\mathbb{R}\:;\: D(\lambda,a,L_{1},L_{2})=0\}\label{eq:Z}
\end{equation}
of the exponential polynomial in (\ref{eq:ep}). Moreover, setting
$a:=w\, e(\varphi_{0})$, $0<w<1$, eq (\ref{eq:ep}) is equivalent
to 
\begin{equation}
e(\lambda L_{2}-\varphi_{0})=\frac{1-w\, e(\lambda L_{1}+\varphi_{0})}{w-e(\lambda L_{1}+\varphi_{0})}.\label{eq:m}
\end{equation}
\end{thm}
\begin{proof}
Eq. (\ref{eq:ep}) follows from a direct computation. As noted in
\cite{JPT11-1}, both sides of (\ref{eq:m}) can be interpreted as
periodic motions on the torus $\mathbb{T}^{1}$: 

(i) The LHS is a uniform motion with constant velocity;

(ii) The Möbius transformation on the RHS has the form $e^{ig(\lambda)}$,
where 
\begin{equation}
g(t):=-\frac{1}{2}+\frac{1}{2\pi}\mathrm{Im}\int_{0}^{t}\frac{\gamma'}{\gamma}=-\frac{1}{2}-\int_{0}^{t}\frac{1-w^{2}}{1-2w\cos(2\pi u)+w^{2}}du.\label{eq:g}
\end{equation}
The solution to (\ref{eq:m}) is obtained at the intersection of the
two motions. In particular, the solution (point-spectrum) is periodic
if and only if $L_{2}/L_{1}$ is rational. See Figure \ref{fig:dense-orbits}. 
\end{proof}
\begin{figure}
\begin{centering}
\begin{tabular}{cc}
\includegraphics{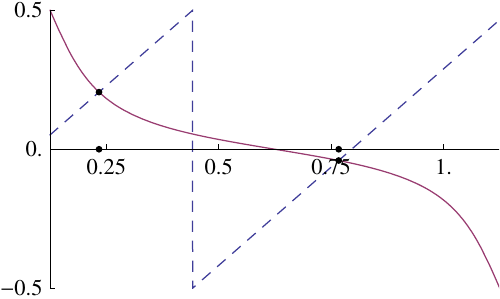} & \includegraphics{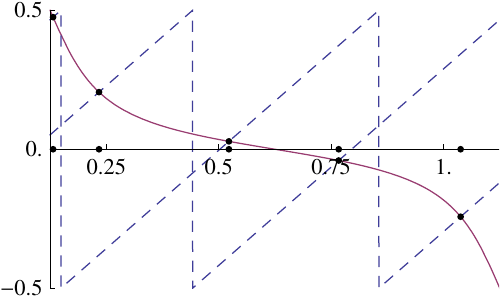}\tabularnewline
$\Lambda_{0}$ &  $\cup_{i=0}^{1}(\Lambda_{i}-i)$\tabularnewline
 & \tabularnewline
\includegraphics{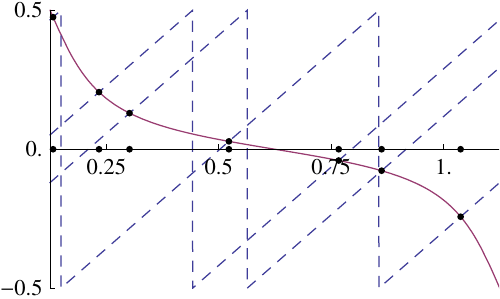} & \includegraphics{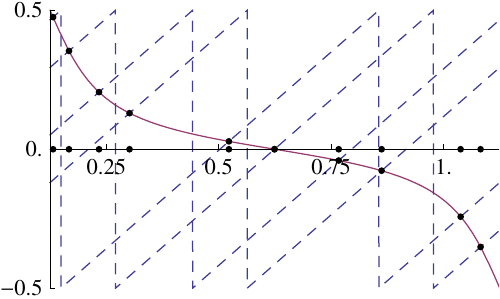}\tabularnewline
$\cup_{i=0}^{2}(\Lambda_{i}-i)$ & $\cup_{i=0}^{3}(\Lambda_{i}-i)$\tabularnewline
 & \tabularnewline
\includegraphics{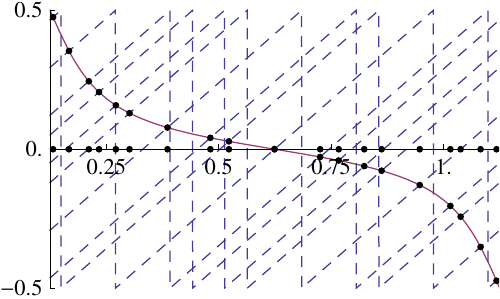} & \includegraphics{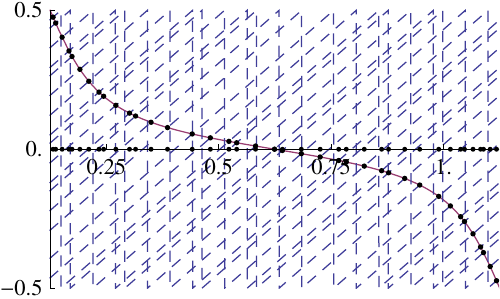}\tabularnewline
 $\cup_{i=0}^{7}(\Lambda_{i}-i)$ &  $\cup_{i=0}^{15}(\Lambda_{i}-i)$\tabularnewline
\end{tabular}
\par\end{centering}

\caption{\label{fig:dense-orbits} Dense orbits: The spectrum as a random-number
generator \cite{JPT11-1}. This is obtained in a number of steps:
\textbf{Step 1.} Consider the curves for argument function from the
right-hand side of (\ref{eq:m}) and the lines with slope $L_{2}$
from the left-hand side of (\ref{eq:m}). Assuming $L_{2}/L_{1}$
is irrational. \textbf{Step 2.} Identify the asymptotes. Let $S_{i}$,
$i\in\mathbb{Z}$, be the $i^{th}$ interval between neighboring branch
cuts; fix $S_{0}$ to be the closest to 0. Note that all the intervals
have the same fixed unit-length, and they extend both to the left
and to the right of $S_{0}$. \textbf{Step 3.} The embedded point-spectrum
$\Lambda_{pt}$ is discrete and infinite, it intersects all these
intervals between branch cuts as $\Lambda_{i}:=\Lambda_{pt}\cap S_{i}$.
Now, translate all of these finite intersections down to $S_{0}$.
Conclusion: Since the line-slope $L_{2}/L_{1}$ is irrational, the
set $\cup_{i\in\mathbb{Z}}(\Lambda_{i}-i)$ is dense in $S_{0}$.}
\end{figure}

\begin{rem}
For all $\lambda\in\Lambda_{pt}=Z(a,L_{1},L_{2})$, the vector of
coefficients (\ref{eq:soln-4}) satisfies 
\begin{align*}
A_{1}\left(\lambda\right) & =\frac{b}{1-a\, e\left(\lambda L_{1}\right)}\, e\left(\lambda\left(\beta_{3}-\alpha_{1}\right)\right)A_{2}\left(\lambda\right)\\
A_{2}\left(\lambda\right) & =\frac{-\overline{b}}{1-\overline{a}\, e\left(\lambda L_{2}\right)}\, e\left(\lambda\left(\beta_{2}-\alpha_{2}\right)\right)A_{1}\left(\lambda\right).
\end{align*}
Note that $\left|A_{1}\left(\lambda\right)\right|=\left|A_{2}\left(\lambda\right)\right|$.
As shown in \cite{JPT11-1}, the pairs of sets 
\begin{equation}
(J_{1}\cup J_{2},\Lambda_{pt})\label{eq:pair}
\end{equation}
forms a \emph{spectral pair} if and only if 
\begin{equation}
A_{1}\left(\lambda\right)=A_{2}\left(\lambda\right),\:\forall\lambda\in\Lambda_{pt};\label{eq:ss}
\end{equation}
and when (\ref{eq:ss}) holds, $J_{1}\cup J_{2}$ is said to be a
\emph{spectral set}. 

Indeed, the union of $[1,\frac{3}{2}]$ and $[2,3]$ in Example \ref{ex:case2}
is not a spectral set; one way to see that is to notice that it is
not a tile for the real line under translations: you cannot fill the
gap $[\frac{3}{2},2]$. For more details, see \cite{JPT11-1}.
\end{rem}

\subsection{Other Examples}
\begin{example}
For $n=3$, let
\[
B=\left(\begin{array}{ccc}
0 & 1 & 0\\
0 & 0 & -1\\
1 & 0 & 0
\end{array}\right)
\]
and so $B'_{\alpha,\beta}(\lambda)=\left(\begin{array}{cc}
e_{\lambda}(L_{1}) & 0\\
0 & -e_{\lambda}(L_{2})
\end{array}\right)$, where $L_{i}=\mbox{length}(J_{i})$, $i=1,2$, as before. Solving
the equation 
\[
\det\left(I_{2}-B'_{\alpha,\beta}(\lambda)\right)=\left(1-e_{\lambda}\left(L_{1}\right)\right)\left(1+e_{\lambda}\left(L_{2}\right)\right)=0
\]
we get 
\[
\Lambda_{pt}=\left(\frac{1}{L_{1}}\mathbb{Z}\right)\cup\left(\frac{\frac{1}{2}+\mathbb{Z}}{L_{2}}\right).
\]
As shown in Figure \ref{fig:deg3}, the Lebesgue spectrum arises from
lumping together $L^{2}(J_{-})$ and $L^{2}(J_{+})$; and the embedded
point-spectrum $\Lambda_{pt}$ accounts for the boundstates in $L^{2}(J_{1})\oplus L^{2}(J_{2})$. 
\end{example}
\begin{figure}[H]
\includegraphics[scale=0.8]{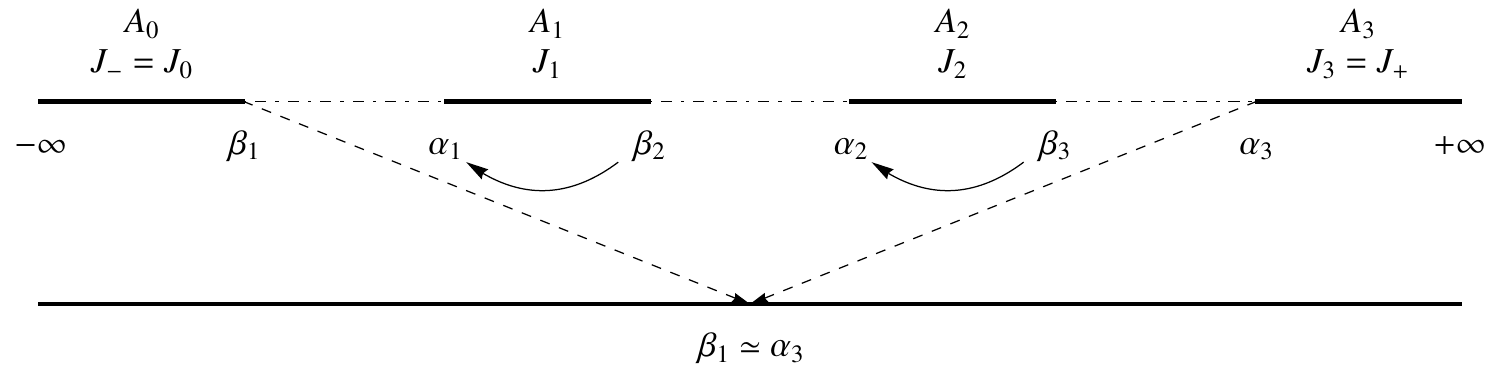}

\caption{\label{fig:deg3}Embedded point spectrum.}
\end{figure}

\begin{example}
Let $n=3$, and 
\[
B=\left(\begin{array}{ccc}
0 & 1 & 0\\
b_{21} & 0 & b_{23}\\
b_{31} & 0 & b_{33}
\end{array}\right)\in U(3)
\]
assuming $\left|b_{23}\right|\neq1$. Here, $B'_{\alpha,\beta}(\lambda)=\left(\begin{array}{cc}
e_{\lambda}\left(L_{1}\right) & 0\\
0 & b_{23}\, e_{\lambda}\left(L_{2}\right)
\end{array}\right)$. The determinant criterion 
\[
\det\left(\begin{array}{cc}
e_{\lambda}\left(L_{1}\right)-1 & 0\\
0 & b_{23}\, e_{\lambda}\left(L_{2}\right)-1
\end{array}\right)=0
\]
yields 
\[
\Lambda_{pt}=\frac{1}{L_{1}}\mathbb{Z}.
\]
As illustrated in Figure \ref{fig:deg4}, we have 
\[
U_{B}(t)=U_{B}^{boundstate}\oplus U_{B}^{cont}(t)
\]
acting on $L^{2}(J_{1})\oplus L^{2}(J_{-}\cup J_{2}\cup J_{+})$.
For a detailed analysis of $U_{B}^{cont}(t)$, see \cite{JPT11-2}. 
\end{example}
\begin{figure}[H]
\includegraphics[scale=0.8]{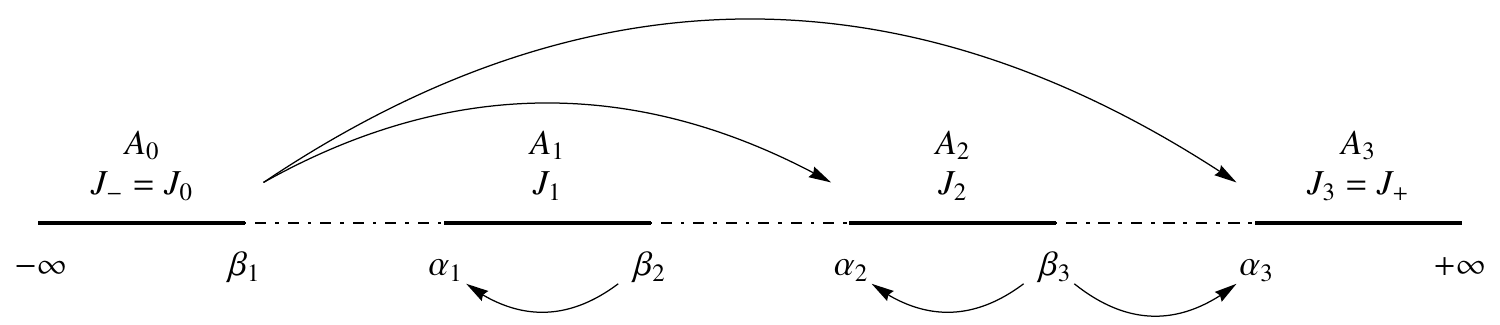}

\caption{\label{fig:deg4}Embedded point-spectrum.}
\end{figure}

\section{Decomposability\label{sec:Decomposability}}

As we outlined in sections \ref{sec:sp} - \ref{sec:deg}, as $B\in U(n)$
varies, the unitary one-parameter groups $U_{B}(t)$ act in $L^{2}(\Omega)$.
Now the given open subset $\Omega$ is a disjoint union of its connected
components, i.e., of a specific set of intervals. As a result, $L^{2}(\Omega)$
splits up as an orthogonal direct sum of a corresponding number of
closed subspaces; one $L^{2}$-space for each of the component intervals.
But it is also true that that the typical scattering theory for $U_{B}(t)$
corresponds to an action in $L^{2}(\Omega)$ that mixes these closed
subspaces in $L^{2}(\Omega)$. Indeed, when $B\in U(n)$ is fixed,
our results Corollaries \ref{cor:multi}, \ref{cor:sp}, \ref{cor:shann};
Remark \ref{rmk:dilation}, and Figure \ref{fig:scatter} yield formulas
for transition probabilities, referring to transition between the
interval-subspaces, and governing the global behavior of $U_{B}(t)$
as it acts in $L^{2}(\Omega)$. The term \textquotedblleft{}decomposability\textquotedblright{}
in the title above refers to invariance under $U_{B}(t)$, for all
$t\in\mathbb{R}$, of some of the interval-subspaces in $L^{2}(\Omega$);
clusters of subspaces.

In this section it is convenient to use a slightly different labeling
of the selfadjoint operators $P_{B}.$ Let $\Omega:=\bigcup_{k=0}^{n}J_{k}$
where $J_{0}:=]-\infty,\beta_{n}],$ $J_{k}:=[\alpha_{k,}\beta_{k}],$
$k=1,2,\ldots,n-1,$ and $J_{n}:=[\alpha_{n},\infty[.$ So $\Omega$
is the complement of $n$ intervals: $\Omega=\mathbb{R}\setminus\bigcup_{k=1}^{n}]\beta_{k},\alpha_{k}[.$
The selfadjoint restriction of $P$ are indexed by the unitaries $B$
from $\ell^{2}(\alpha_{k})\to\ell^{2}(\beta_{k}).$ Identifying the
spaces $\ell^{2}(\alpha_{k})$ and $\ell^{2}(\beta_{k})$ with $\mathbb{C}^{n}$
we realize $B$ as an $n\times n$ matrix. 

\begin{figure}[H]
\includegraphics[scale=0.8]{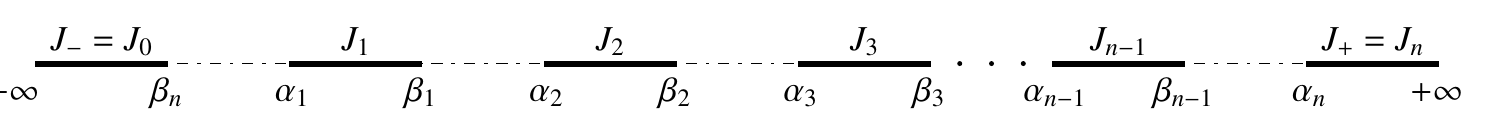}

\caption{The complement of $n$ bounded intervals in $\mathbb{R}$ ($n>2$).}

\end{figure}

As usual the domain of the maximal operator is the absolutely continuous
functions on $\Omega$ and the selfadjoint restrictions $P_{B}$ are
in one-to-one correspondence with the unitaries $B.$ The domain of
the selfadjoint restriction $P_{B}$ determined by $B$ is the set
of absolutely continuous functions $f:\Omega\to\mathbb{C}$ satisfying
the set of boundary conditions 
\begin{equation}
B\left[\begin{array}{c}
f(\alpha_{1})\\
f(\alpha_{2})\\
\vdots\\
f(\alpha_{n})
\end{array}\right]=\left[\begin{array}{c}
f(\beta_{1})\\
f(\beta_{2})\\
\vdots\\
f(\beta_{n})
\end{array}\right],\label{eq:001}
\end{equation}
and $P_{B}f=\frac{1}{i2\pi}f'.$ Suppose $B$ is block diagonal, this
is 
\[
B=\left[\begin{array}{cc}
B_{1} & 0\\
0 & B_{2}
\end{array}\right]=B_{1}\oplus B_{2}
\]
where $B_{1}$ is a $k\times k$ matrix and $B_{2}$ is a $(n-k)\times(n-k)$
matrix. Then $B_{j},j=1,2$ are unitaries and we can write $\Omega=\Omega_{1}\cup\Omega_{2}$
where
\begin{align*}
\Omega_{1} & :=J_{1}\cup J_{2}\cup\cdots\cup J_{k}\text{ and}\\
\Omega_{2} & :=J_{0}\cup J_{k+1}\cup J_{k+2}\cup\cdots\cup J_{n}.
\end{align*}
Consequently, $L^{2}\left(\Omega\right)=L^{2}\left(\Omega_{1}\right)\oplus L^{2}\left(\Omega_{2}\right)$
and $P_{B}=P_{B_{1}}\oplus P_{B_{2}}$ where $P_{B_{1}}$ is the momentum
operator determined by 
\begin{equation}
B_{1}\left[\begin{array}{c}
f(\alpha_{1})\\
f(\alpha_{2})\\
\vdots\\
f(\alpha_{k})
\end{array}\right]=\left[\begin{array}{c}
f(\beta_{1})\\
f(\beta_{2})\\
\vdots\\
f(\beta_{k})
\end{array}\right],\label{eq:001-1}
\end{equation}
and $P_{B_{2}}$ is the momentum operator determined by 
\begin{equation}
B_{2}\left[\begin{array}{c}
f(\alpha_{k+1})\\
f(\alpha_{k+2})\\
\vdots\\
f(\alpha_{n})
\end{array}\right]=\left[\begin{array}{c}
f(\beta_{k+1})\\
f(\beta_{k+2})\\
\vdots\\
f(\beta_{n})
\end{array}\right].\label{eq:001-2}
\end{equation}
Hence, if $B$ is block diagonal it is sufficient to study $P_{B_{1}}$
and $P_{B_{2}}.$ 
\begin{rem}
A reason for grouping the unbounded intervals this way is that the
deficiency indices work this way. The restriction of $P$ to each
$C_{c}^{\infty}\left(J_{k}\right),$ $k=1,\ldots,n-1$ and the restriction
of $P$ to $C_{c}^{\infty}\left(J_{0}\cup J_{n}\right)$ all have
deficiency indices $(1,1).$ Consequently, the restriction of $P_{B_{1}}$
to $C_{c}^{\infty}\left(\Omega_{1}\right)$ has deficiency indices
$(k,k)$ and the restriction of $P_{B_{2}}$ to $C_{c}^{\infty}\left(\Omega_{2}\right)$
has deficiency indices $(n-k,n-k).$ Furthermore, if $k=2$ the $P_{B_{1}}$
problem is investigated in \cite{JPT11-1} and if $n-k=2$ the $P_{B_{2}}$
problem is investigated in \cite{JPT11-2}. 
\end{rem}
Recall, a \emph{permutation matrix} is an $n\times n$ matrix obtained
from the identity matrix $I_{n}=\textrm{diag}\left(1,1,\ldots,1\right)$
by permuting of the columns of $I_{n}.$ 
\begin{defn}
We say two unitary matrices $A$ and $B$ are \emph{permutation equivalent,}
if there is a permutation matrix $S$ such that $B=S^{*}AS.$ We say
a unitary matrix $B$ is \emph{decomposable,} if $B$ is permutation
equivalent to a block diagonal matrix and we say $B$ is \emph{indecomposable,}
if $B$ is not decomposable. \end{defn}
\begin{example}
Let $n=4$, and 
\[
B=\left[\begin{array}{cccc}
b_{11} & 0 & b_{13} & 0\\
0 & b_{22} & 0 & b_{24}\\
b_{31} & 0 & b_{33} & 0\\
0 & b_{42} & 0 & b_{44}
\end{array}\right]\in U(4).
\]
The boundary condition reads
\[
\left[\begin{array}{cccc}
b_{11} & 0 & b_{13} & 0\\
0 & b_{22} & 0 & b_{24}\\
b_{31} & 0 & b_{33} & 0\\
0 & b_{42} & 0 & b_{44}
\end{array}\right]\left[\begin{array}{c}
f(\alpha_{1})\\
f(\alpha_{2})\\
f(\alpha_{3})\\
f(\alpha_{4})
\end{array}\right]=\left[\begin{array}{c}
f(\beta_{1})\\
f(\beta_{2})\\
f(\beta_{3})\\
f(\beta_{4})
\end{array}\right].
\]
Note that $B$ is permutation equivalent to 
\[
A=\left[\begin{array}{cccc}
b_{11} & b_{13} & 0 & 0\\
b_{31} & b_{33} & 0 & 0\\
0 & 0 & b_{22} & b_{24}\\
0 & 0 & b_{42} & b_{44}
\end{array}\right];
\]
and it follows that the system decouples as shown in Fig. \ref{fig:permutation}.

\begin{figure}
\includegraphics[scale=0.8]{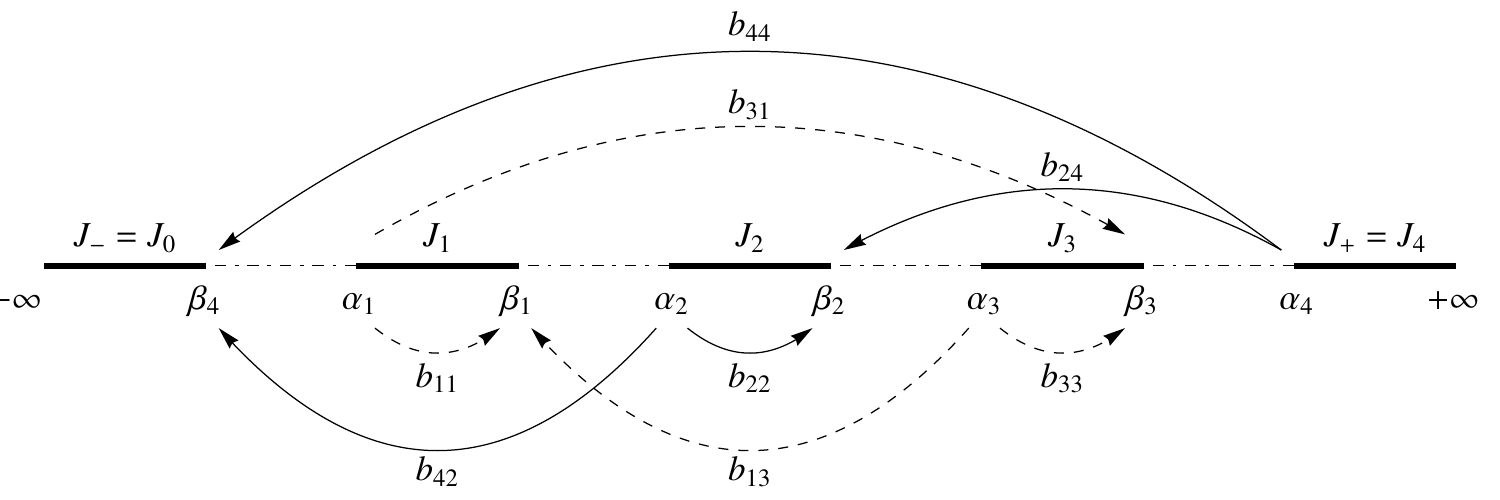}

\caption{\label{fig:permutation}$B$ is permutation equivalent to $A$. The
system decouples into a direct sum of two subsystems: \textbf{(i)}
The dashed diagram contains two bounded intervals $J_{1}$ and $J_{3}$,
and this corresponds to \cite{JPT11-1}; \textbf{(ii)} The solid diagram
consists of one bounded component $J_{2}$ and two unbounded components
$J_{\pm}$; and it is investigated in \cite{JPT11-2}.}
\end{figure}

\end{example}
Supposing $B$ is permutation equivalent to $A$ we can write (\ref{eq:001})
as 
\[
AS\left[\begin{array}{c}
f(\alpha_{1})\\
f(\alpha_{2})\\
\vdots\\
f(\alpha_{n})
\end{array}\right]=S\left[\begin{array}{c}
f(\beta_{1})\\
f(\beta_{2})\\
\vdots\\
f(\beta_{n})
\end{array}\right]
\]
using that $S$ is a permutation this can be written as 
\[
A\left[\begin{array}{c}
f(\alpha_{i_{1}})\\
f(\alpha_{i_{2}})\\
\vdots\\
f(\alpha_{i_{n}})
\end{array}\right]=\left[\begin{array}{c}
f(\beta_{i_{1}})\\
f(\beta_{i_{2}})\\
\vdots\\
f(\beta_{i_{n}})
\end{array}\right].
\]
So Suppose $B$ is permutation equivalent to a block diagonal matrix
$A=A_{1}\oplus A_{2},$ then $B$ is permutation equivalent to a block
diagonal matrix $A$ such that $i_{n}=n.$ Putting it together we
have 
\begin{thm}
\label{thm:Decomposition}If $B$ is decomposable, then we can write
$S^{*}BS=B_{1}\oplus B_{2}\oplus\cdots\oplus B_{k}$ where each $B_{j}$
is indecomposable and $S$ is a permutation. The $P_{B_{j}},j=1,2,\ldots,k-1$
problems only contain bounded intervals and the $P_{B_{k}}$ problem
contains the unbounded intervals and, perhaps some of the bounded
intervals. 
\end{thm}
Since the unbounded intervals are ``special'' it is useful to write
as 
\[
B=\left[\begin{array}{cc}
B' & \mathbf{u}\\
\mathbf{w} & c
\end{array}\right]
\]

\begin{lem}
If $B$ is decomposable, then $B'$ is degenerate, i.e., has an eigenvalue
with absolute value one. 
\end{lem}

\section{Eigenfunctions\label{sec:Eigenfunctions}}

Fix some unitary matrix $B.$ The generalized eigenfunctions
\begin{equation}
\psi_{\lambda}(x):=\left(A_{0}(\lambda)\chi_{]-\infty,\beta_{n}[}(x)+\sum_{j=1}^{n-1}A_{j}(\lambda)\chi_{[\alpha_{j},\beta_{j}[}(x)+A_{n}(\lambda)\chi_{[\alpha_{n},\infty[}(x)\right)e_{\lambda}(x)\label{eq:002}
\end{equation}
that satisfy (\ref{eq:001}) for the generalized eigenspace corresponding
to $\lambda.$ The coefficient $A_{j}=A_{j}(\lambda)$ is obtained
by solving the differential equation $\frac{d}{dx}\psi=2\pi i\psi$
on the interval $J_{j}.$ Plugging (\ref{eq:004}) into (\ref{eq:001})
we see the generalized eigenfunctions are determined by the solutions
$A_{0},A_{1},\cdots,A_{n}$ to the system of $n$ linear equations
in $n+1$ unknowns: 
\begin{equation}
B\left[\begin{array}{c}
A_{1}e(\lambda\alpha_{1})\\
A_{2}e(\lambda\alpha_{2})\\
\vdots\\
A_{n-1}e(\lambda\alpha_{n-1})\\
A_{0}e(\lambda\alpha_{n})
\end{array}\right]=\left[\begin{array}{c}
A_{1}e(\lambda\beta_{1})\\
A_{2}e(\lambda\beta_{2})\\
\vdots\\
A_{n-1}e(\lambda\beta_{n-1})\\
A_{n}e(\lambda\beta_{n})
\end{array}\right].\label{eq:003}
\end{equation}
Let $D_{\alpha}:=\mathrm{diag}(e(\lambda\alpha_{1}),e(\lambda\alpha_{2}),\ldots,e(\lambda\alpha_{n})),$
$D_{\beta}:=\mathrm{diag}(e(\lambda\beta_{1}),e(\lambda\beta_{2}),\ldots,e(\lambda\beta_{n})),$
and $B_{\alpha.\beta}:=D_{\beta}^{*}BD_{\alpha}.$ Then our eigenvector
equation can be written as 
\begin{equation}
B_{\alpha.\beta}\left[\begin{array}{c}
A_{1}\\
A_{2}\\
\vdots\\
A_{n-1}\\
A_{0}
\end{array}\right]=\left[\begin{array}{c}
A_{1}\\
A_{2}\\
\vdots\\
A_{n-1}\\
A_{n}
\end{array}\right].\label{eq:004}
\end{equation}
Writing $\mathbb{C}^{n}=\mathbb{C}^{n-1}\oplus\mathbb{C}$ we have
the decomposition 
\begin{equation}
B_{\alpha,\beta}=\left[\begin{array}{cc}
B' & \mathbf{u}\\
\mathbf{w} & c
\end{array}\right]\label{eq:B'}
\end{equation}
where $c$ is a complex number, $\mathbf{u},\mathbf{w}$ are in $\mathbb{C}^{n-1}$
and $B'$ is a $(n-1)\times(n-1)$ matrix. With this notation we can
write (\ref{eq:004}) as $\left[\begin{array}{cc}
B' & \mathbf{u}\\
\mathbf{w} & c
\end{array}\right]\left[\begin{array}{c}
A'\\
A_{0}
\end{array}\right]=\left[\begin{array}{c}
A'\\
A_{n}
\end{array}\right],$ where $A'=\left[A_{1},A_{2},\ldots,A_{n-1}\right].$ 
\begin{thm}
If $\mathbf{u}$ is in the range of $I'-B'$ and $\eta_{0}$ is such
that $\mathbf{u}=\left(I'-B'\right)\eta_{0},$ then the solutions
to (\ref{eq:004}) are determined by: 
\begin{align*}
A' & =A_{0}\eta_{0}+\zeta\\
A_{n} & =A_{0}\left(c+\mathbf{w}\eta_{0}\right)+\mathbf{w}\zeta,
\end{align*}
where $\zeta\in\ker\left(I'-B'\right)$ and $A_{0}\in\mathbb{C}.$
If $\mathbf{u}$ is not in range of $I'-B',$ then the solutions to
(\ref{eq:004}) are determined by:
\begin{align*}
A & '=\zeta\\
A_{n} & =\mathbf{w}\zeta
\end{align*}
where $\zeta\in\ker\left(I'-B'\right)$ and $A_{0}=0.$ \end{thm}
\begin{proof}
We can write (\ref{eq:004}) as 
\begin{align*}
B'A'+A_{0}\mathbf{u} & =A'\\
\mathbf{w}A'+cA_{0} & =A_{n}.
\end{align*}
The result is immediate from this. 
\end{proof}
Neither the theorem not the first corollary require $B$ to be unitary,
but the second corollary needs $\mathbf{u}=0$ implies $\mathbf{w}=0,$
which is a consequence of the assumption that $B$ is unitary. 
\begin{cor}
\label{cor:Multiplicity-is-One}If $1$ is not an eigenvalue for $B',$
then the solutions to (\ref{eq:004}) are determined by: 
\begin{align*}
A' & =A_{0}\left(I'-B'\right)^{-1}\mathbf{u}\\
A_{n} & =A_{0}\left(c+\mathbf{w}\left(I'-B'\right)^{-1}\mathbf{u}\right).
\end{align*}
In particular, the set of solutions to (\ref{eq:004}) is one dimensional. \end{cor}
\begin{proof}
If $1$ is not an eigenvalue for $B',$ then the kernel of $I'-B'$
equals $\{0\}$ and the range of $I'-B'$ is $\mathbb{C}^{n-1},$
in particular, $\mathbf{u}$ is in the range of $I'-B'$. \end{proof}
\begin{cor}
If $\mathbf{u}=0,$ then the solutions to (\ref{eq:004}) are determined
by: 
\begin{align*}
A' & =\zeta\\
A_{n} & =A_{0}c+\mathbf{w}\zeta
\end{align*}
where $\zeta\in\ker\left(I'-B'\right)$ and $A_{0}\in\mathbb{C}.$ \end{cor}
\begin{proof}
If $\mathbf{u}=0,$ then $\mathbf{u}$ is in the range of $I'-B'.$
Since $B$ is unitary $|c|=1,$ hence $\mathbf{w}=0.$
\end{proof}
An immediate consequence of Theorem \ref{thm:soln} and Corollary
\ref{cor:B-1} we have

\begin{thm}
If $B'$ is not degenerate, then the spectrum of $P_{B}$ has uniform
multiplicity one. \end{thm}
\begin{cor}
Suppose $B$ is decomposable with decomposition $\bigoplus_{j=1}^{k-1}B_{j}\oplus B_{k}$
in the sense of Theorem \ref{thm:Decomposition} and each $B_{j}$
is not degenerate, then the spectrum of $P_{B}=\bigoplus_{j=1}^{k-1}P_{B_{j}}\oplus P_{B_{k}}$
where the spectrum of $P_{B_{j}}$ is a set $\Lambda_{j}$ of simple
eigenvalues and $P_{B_{k}}$ has spectrum equal to the real line and
the spectral measure is absolutely continuous with respect to Lebesgue
measure. 
\end{cor}
In particular, the set of eigenvalues of $P_{B}$ is $\bigcup_{j=1}^{k-1}\Lambda_{j}$
and the multiplicity of an eigenvalue $\lambda$ is the number of
elements in $\left\{ j=1,2,\ldots,k-1\mid\lambda\in\Lambda_{j}\right\} .$

\section{\label{sec:infinite}Scratching the Surface of Infinity}

In this section we consider some cases when the give open set $\Omega$
has an infinite number of connected components. As in the discussion
above, we still assume that two of the components are the infinite
half-lines. Our motivation for studying the infinite case is four-fold: 

One is the study of geometric analysis of Cantor sets; so the infinite
case includes a host of examples when $\Omega$ is the complement
in $\mathbb{R}$ of one of the Cantor sets studied in earlier recent
papers \cite{DJ07,DJ11,JP98,PW01}. The other is our interest in boundary
value problems when the boundary is different from the more traditional
choices. And finally, the case when the von Neumann-deficiency indices
are $(\infty,\infty)$ offers new challenges; see e.g., \cite{DS88b};
involving now reproducing kernels, and more refined spectral theory. 

Finally we point out how the spectral theoretic conclusions for the
infinite case differ from those that hold in the finite case (see
details above for the finite case.) For example, for finitely many
intervals (Theorem \ref{thm:Bdecomp}) we computed that the Beurling
density of embedded point spectrum equals the total length of the
finite intervals. By contrast, we show below that when $\Omega$ has
an infinite number of connected components, there is the possibility
of dense point spectrum; see Example \ref{ex:dense}.

Let $I_{k}=(r_{k},s_{k})$ be a sequence of pairwise disjoint open
subintervals of the open interval $(0,1).$ Let 
\[
\Omega=(-\infty,0)\cup(1,\infty)\cup\bigcup_{k=0}^{\infty}I_{k}.
\]
The functions satisfying the eigenfunction equation $\frac{1}{i2\pi}\frac{d}{dx}\psi_{\lambda}=\lambda\psi_{\lambda}$
are the functions 
\[
\psi_{\lambda}(x)=\left(A_{-\infty}(\lambda)\chi_{(-\infty,0)}(x)+A_{\infty}(\lambda)\chi_{(1,\infty)}(x)+\sum_{k=0}^{\infty}A_{k}(\lambda)\chi_{I_{k}}(x)\right)e_{\lambda}(x),
\]
where $A_{-\infty},A_{\infty},$ and $A_{k}$ are constants depending
on $\lambda$. Let $r_{0}=1$ and $s_{0}=0.$ 
\begin{example}
An example of this is the complement of the middle thirds Cantor set
$C.$ We can write the complement of the Cantor set $C$ as 
\[
(-\infty,0)\cup(1,\infty)\cup\bigcup_{j=0}^{\infty}\bigcup_{k=1}^{2^{j}}\left(a_{j,k},a_{j,k}+3^{-(j+1)}\right)
\]
where in base $3$ 
\[
a_{0,1}=.1,
\]
\[
a_{1,1}=.01,a_{1,2}=.21
\]
\[
a_{2,1}=.001,a_{2,1}=.021,a_{2,3}=.201,a_{2,4}=.221
\]
and so on. So $a_{j,k},$ $k=1,\ldots,2^{j}$ are the numbers with
finite base three expansions of the form
\[
0.x_{1}x_{2}\cdots x_{j}1,x_{\ell}\in\{0,2\}.
\]
In this case the generalized eigenfunctions are 
\[
\psi_{\lambda}(x)=\left(A_{-\infty}(\lambda)\chi_{(-\infty,0)}(x)+A_{\infty}(\lambda)\chi_{(1,\infty)}(x)+\sum_{j=0}^{\infty}\sum_{k=1}^{2^{j}}A_{j,k}(\lambda)\chi_{\left(a_{j,k}a_{j,k}+3^{-(j+1)}\right)}(x)\right)e_{\lambda}(x).
\]

\end{example}
Consider a selfadjoint restriction $P_{B}$ of the maximal momentum
operator on $\Omega$ such that $A_{k}\in\ell^{2}$ and 
\[
BD_{r}(\lambda)\left[\begin{array}{c}
A_{\infty}\\
A_{1}\\
A_{2}\\
A_{3}\\
\vdots
\end{array}\right]=D_{s}(\lambda)\left[\begin{array}{c}
A_{-\infty}\\
A_{1}\\
A_{2}\\
A_{3}\\
\vdots
\end{array}\right],
\]
 where
\begin{align*}
D_{r}(\lambda) & =\mathrm{diag}\left(e(\lambda r_{0}),e(\lambda r_{1}),e(\lambda r_{2}),\cdots\right)=\mathrm{diag}\left(e(\lambda),e(\lambda r_{1}),e(\lambda r_{2}),\cdots\right)\\
D_{s}(\lambda) & =\mathrm{diag}\left(e(\lambda s_{0}),e(\lambda s_{1}),e(\lambda s_{2}),\cdots\right)=\mathrm{diag}\left(1,e(\lambda s_{1}),e(\lambda s_{2}),\cdots\right)
\end{align*}
 and $B$ is some unitary on $\ell^{2}.$ 
\begin{thm}
If $B=\mathrm{diag}\left(1,1,\ldots\right),$ then the spectrum of
$P_{B}$ is the real line and the embedded point spectrum is $\Lambda_{p}=\bigcup_{k=1}^{\infty}\frac{1}{\ell_{k}}\mathbb{Z},$
where $\ell_{k}=s_{k}-r_{k}$ is the length of $I_{k}.$ The multiplicity
of $\lambda\in\Lambda_{p}$ equals the cardinality of the set $\{k\mid\lambda\ell_{k}\in\mathbb{Z}\}.$ \end{thm}
\begin{proof}
Similar to the proof of Theorem \ref{thm:Bdecomp}. \end{proof}
\begin{example}
Some examples illustrating this are: \end{example}
\begin{enumerate}
\item If $\ell_{k}=2^{-k},$ then $\Lambda_{p}=2\mathbb{Z}.$ Let $\mathbb{Z}_{\mathrm{odd}}$
be the odd integers. The eigenvalues in $2^{k}\mathbb{Z}_{\mathrm{odd}}$
have multiplicity $k$ and $0$ has infinite multiplicity. 
\item For the complement of the middle thirds Cantor set $\Lambda_{p}=3\mathbb{Z}.$
The eigenvalues that are multiples of $3^{k}$ but not of $3^{k+1}$
have multiplicity $2^{k}-1$ and $0$ has infinite multiplicity. 
\item If $\ell_{k}/\ell_{j}$ is irrational for all $j\neq k,$ then $0$
has infinite multiplicity and all other eigenvalues have multiplicity
one. \end{enumerate}
\begin{cor}
If $B=\mathrm{diag}\left(e(\theta_{0}),e(\theta_{1}),\ldots\right),$
then the spectrum of $P_{B}$ is the real line and the embedded point
spectrum is $\Lambda_{p}=\bigcup_{k=1}^{\infty}\left(\frac{\theta_{k}}{\ell_{k}}+\frac{1}{\ell_{k}}\mathbb{Z}\right),$
where $\ell_{k}=s_{k}-r_{k}$ is the length of $I_{k}.$ The multiplicity
of $\lambda\in\Lambda_{p}$ equals the cardinality of the set $\{k\mid\lambda\ell_{k}-\theta_{k}\in\mathbb{Z}\}.$
\end{cor}
When we have a finite number of intervals the point spectrum has uniform
density equal to the sum of the lengths of the intervals, see Theorem
\ref{thm:Bdecomp}. The following example shows that this need not
be the case for infinitely many intervals. 
\begin{example}
\label{ex:dense}Suppose $B=\mathrm{diag}\left(e(\theta_{0}),e(\theta_{1}),\ldots\right)$
and $\ell_{k}=2^{-k}.$ Then $2^{k}\left(\theta_{k}+m\right)=2^{j}\left(\theta_{j}+n\right)$
if and only if $2^{j+k}\left(\theta_{k}-\theta_{j}\right)=2^{k+j}\left(n-m\right).$
Hence, if $\theta_{k}-\theta_{j}$ is not an integer when $k\neq j,$
then each eigenvalue has multiplicity one. Note $2^{k}\theta_{k}$
is an eigenvalue for each $k.$ Hence, if $2^{k}\theta_{k}\to\lambda_{0}$
then $\lambda_{0}$ is a limit point of $\Lambda_{p}.$ Similarly,
by a suitable choice of the sequence $\theta_{k},$ we can arrange
that $P_{B}$ has dense point spectrum. \end{example}
\begin{thm}
If we write $\ell^{2}=\mathbb{C}\oplus\ell^{2},$ then $B$ takes
the form 
\[
B=\left(\begin{array}{cc}
c & \mathbf{w}^{*}\\
\mathbf{u} & B'
\end{array}\right).
\]
If the spectrum of $B'$ does not intersect the unit circle, then
the spectrum $P_{B}$ is the real line and each point in the spectrum
has multiplicity one, in particular, the point spectrum is empty. \end{thm}
\begin{proof}
This is similar to parts of the proof of Theorem \ref{thm:LAP} and
Theorem \ref{thm:soln}. 
\end{proof}

\section*{Acknowledgments}

The co-authors, some or all, had helpful conversations with many colleagues,
and wish to thank especially Professors Daniel Alpay, Ilwoo Cho, Dorin
Dutkay, Alex Iosevich, Paul Muhly, Yang Wang, and Qingbo Huang. And
going back in time, Bent Fuglede (PJ, SP), and Robert T. Powers, Ralph
S. Phillips, Derek Robinson (PJ). The first named author was supported
in part by the National Science Foundation, via a VIGRE grant.

\bibliographystyle{amsalpha}
\bibliography{Number3}

\end{document}